\documentclass[12pt, a4paper, oneside]{amsart}
\pdfoutput=1
\usepackage{amsmath, amssymb, amsthm, amsbsy, bbm, amsfonts, enumitem, color}

\usepackage[all]{xy}
\usepackage[english]{babel}

\usepackage[pdfpagelabels]{hyperref}

\usepackage{graphicx, color}

\usepackage[pdfpagelabels]{hyperref}

\theoremstyle{plain}
\newtheorem{thm}{Theorem}[section]
\newtheorem*{thm*}{Theorem}
\newtheorem{prop}[thm]{Proposition}
\newtheorem*{prop*}{Proposition}
\newtheorem{lemma}[thm]{Lemma}
\newtheorem*{lemma*}{Lemma}

\newtheorem*{corollary*}{Corollary}

\newtheorem{question}[thm]{Open Problem}

\theoremstyle{definition}
\newtheorem{definition}[thm]{Definition}
\newtheorem*{definition*}{Definition}
\newtheorem{example}[thm]{Example}

\theoremstyle{remark}
\newtheorem{remark}[thm]{Remark}

\newtheorem{exercise}[thm]{Exercise}

\newcommand{\R}{\mathbb{R}} 
\newcommand{\E}{\mathbb{E}} 
\newcommand{\h}{\mathbb{H}} 
\newcommand{\N}{\mathbb{N}}
\newcommand{\Z}{\mathbb{Z}}

\newcommand{\define}{\mathrel{\mathop:}=}
\newcommand{\ddefine}{\mathrel{=\mathop:}}

\newcommand{\Stab}{{\mathrm{Stab}}} 

\newcommand{\seg}{\mathrm{seg}} 
\newcommand{\Aut}{\mathrm{Aut}} 

\newcommand{\cat}{\mathrm{CAT}}

\newcommand{\Isom}{\mathrm{Iso}} 
\newcommand{\dom}{\mathrm{dom}} 
\newcommand{\fix}{\mathrm{Fix}} 
\newcommand{\s}{\mathbb{S}} 
\newcommand{\gseg}{\overline} 

\newcommand{\cubes}{\mathcal{C}} 
\newcommand{\glue}{\mathcal{S}} 
\newcommand{\proj}{\mathrm{p}} 
\newcommand{\lk}{\mathrm{lk}} 
\newcommand{\less}{\preceq} 
\newcommand{\Ends}{e} 
\newcommand{\Cay}{\Gamma} 
\newcommand{\gact}{\circlearrowright} 
\newcommand{\mvert}{\;\middle\vert\;}
\newcommand{\cH}{\mathcal{H}}

\newcommand{\petra}[1]{\textcolor{magenta}{\bf{#1}}}

\numberwithin{equation}{thm}

\begin{document}

\hypersetup{pdfauthor={Petra Schwer},pdftitle={Kubische Komplexe}}.
\noindent

\title{Lecture notes on CAT(0) cube complexes}
\author{Petra Schwer}
\date{\today}

\maketitle

\tableofcontents


\section*{On these notes}
These are the notes of a two-hour one-semester course on CAT(0) cube
complexes which I taught in Münster during winter term 2012/2013 and a four-hour one-semester course at KIT during the winter term 2014/15.
Since most of the students did neither have a strong background in
metric geometry nor in group theory I tried to keep the material as
elementary and self-contained as possible.

I am sure there are many typos to find. Please let me know if you spot
one (or two). Any remarks or suggestions on how to improve these
notes are welcome.

%
%

\newpage 
\section{CAT(0) metric spaces}\label{Sec_cat0}

We start slowly with this introductory chapter on non-positive curvature. Since metric geometry won't be the focus of this book we will only introduce the basic notions and definitions of $\cat(\kappa)$-spaces and prove some of their elementary properties. More details on the topic can be found in the little green book by Ballmann \cite{Ballmann} or in Bridson and Haefliger's monograph \cite{BH}. 

The notion of a  $\cat(\kappa)$-space is a curvature concept which works in the setting of geodesic metric spaces. It generalizes in a very natural way the notion of a manifold with non-positive sectional curvature. Its definition is by means of comparison with the model spaces $M_\kappa$ of  dimension two and constant sectional curvature $\kappa$. That is either spheres of radius $\sqrt{\kappa}$, hyperbolic space with the appropriately rescaled metric and the flat Euclidean plane. A metric space will be defined of curvature less than or equal to $\kappa$ if its triangles are thinner than the respective comparison triangles in the model space $M_\kappa$. 

We start by reviewing the notion of a geodesic metric space. 

\begin{definition}[Geodesics]\label{def:geodesics}
Let $(X,d)$ be a metric space and $x,y$ points in $X$. A \emph{geodesic} $\gamma$ in $X$ from $x$ to $y$, we write $\gamma: x \rightsquigarrow y$, is a continuous map $\gamma:[0, l]\to X, l \in \R^+$ such that $\gamma(0)=x, \gamma(l)=y$ and 
\begin{equation}\label{eqn:geodesic}
d(\gamma(t), \gamma(t'))=\vert t-t'\vert \forall t,t'\in [0,l]. 
\end{equation}
The \emph{geodesic segment}  $\gseg{xy}$ between $x$ and $y$ is the image of $\gamma$ in $X$. 
A \emph{geodesic ray} is a continuous map $\gamma:[0, \infty)\to X$ such that (\ref{eqn:geodesic}) is  satisfied for all $t,t'\in \R^+$, whereas a \emph{geodesic line} is a continuous map $\gamma:\R\to X$ with  (\ref{eqn:geodesic}) true $\forall t, t'\in \R$. If for a continuous map $\gamma:[a,b]\to X$ there exists $\forall t\in [a, b]$ an $\varepsilon >0$ such that  $\gamma\vert_[t-\varepsilon, t+\varepsilon]$ is a geodesic we call $\gamma$ a  \emph{local geodesic}.
\end{definition}

\begin{definition}[Geodesic metric space]\label{def:geodesicmetricspace}
A  metric space $(X,d)$ is \emph{(uniquely) geodesic} if for all pairs of points $x,y\in X$ there exists a (unique) geodesic $\gamma:x\rightsquigarrow y$.  We say $(X, d)$ is \emph{$r$-uniquely} geodesic iff $\forall x,y\in X$ such that $d(x,y)< r$ there exists a unique geodesic $\gamma:x\rightsquigarrow y$. 
\end{definition}

\begin{example}[(Non-)geodesic metric spaces]
\begin{enumerate}
 \item $\R^2\setminus\{0\}$ is not a geodesic metric space.
 \item $\s^2$ is geodesic and $\pi$-uniquely geodesic.
 \item $\R^2$ with the standard Euclidean metric is uniquely geodesic. Now if we consider $(\R^2, d_1)$ where $d_1(x,y)\define \vert y-x\vert_1$ is the norm of the difference of the vectors in the $l_1$ norm on $\R^2$, i.e. $d_1(x, y)=\vert x_1-y_1\vert +\vert x_2-y_2\vert$. Then $\R^2$ with this metric is not uniquely geodesic as you can see illustrated in Figure ~\ref{fig_02}.  
\end{enumerate}
\end{example}

\begin{figure}[htbp]
\begin{center}
   	\resizebox{!}{0.4\textwidth}{\includegraphics{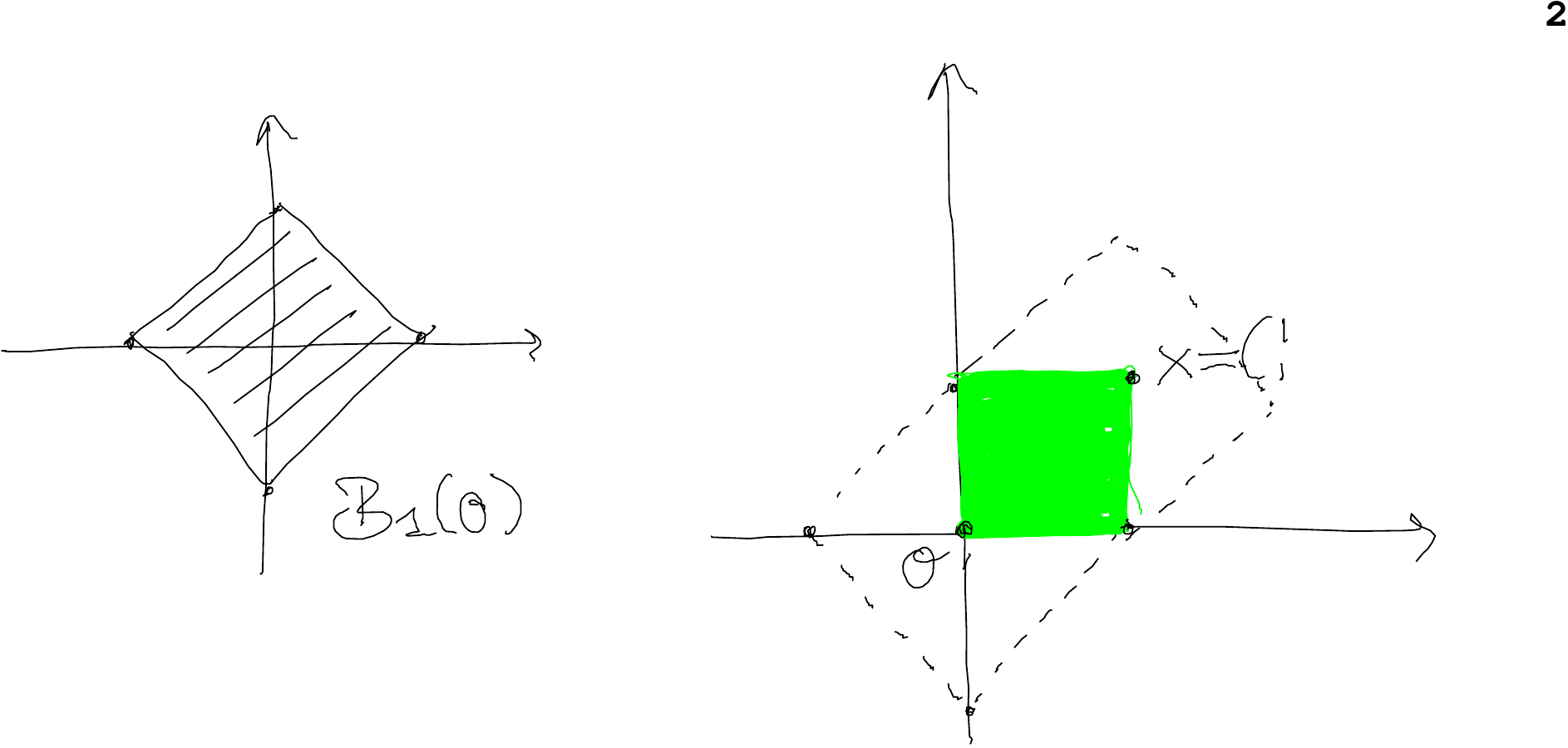}}
\caption[geodesic]{$\R^2$ with the $l_1$-metric is not uniquely geodesic. }
\label{fig_02}
\end{center}
\end{figure}

\begin{definition}[Model spaces]
For $\kappa\in\R$ let the 2-dimensional \emph{model space} $M_\kappa^2$ be the unique differentiable surface with constant sectional curvature $\kappa$. We write $D_\kappa$ for the diameter of $M_\kappa^2$. 
For a detailed account on these spaces see for example Chapter I.6 of \cite{BH}.
Here are the most important examples. Knowing these will be sufficient for the remainder of this course:
\begin{itemize}
 \item $M_1^2 =\s^2$ the round 2-dimensional sphere with the standard metric, that is $\s^2\cong S_1(0)\subset \R^3$, with $D_1=\pi$. 
 \item $M_0^2 = \E^2$ that is $\R^2$ with the standard Euclidean metric with $D_0=\infty$.
 \item $M_{-1}^2 =\h ^2$ the hyperbolic plane having $D_{-1}=\infty$. 
\end{itemize}
All other model spaces are obtained from $M_1^2$ or $M_{-1}^2$ by rescaling the metric. 
\end{definition}

\begin{figure}[htbp]
\begin{center}
   	\resizebox{!}{0.3\textwidth}{\includegraphics{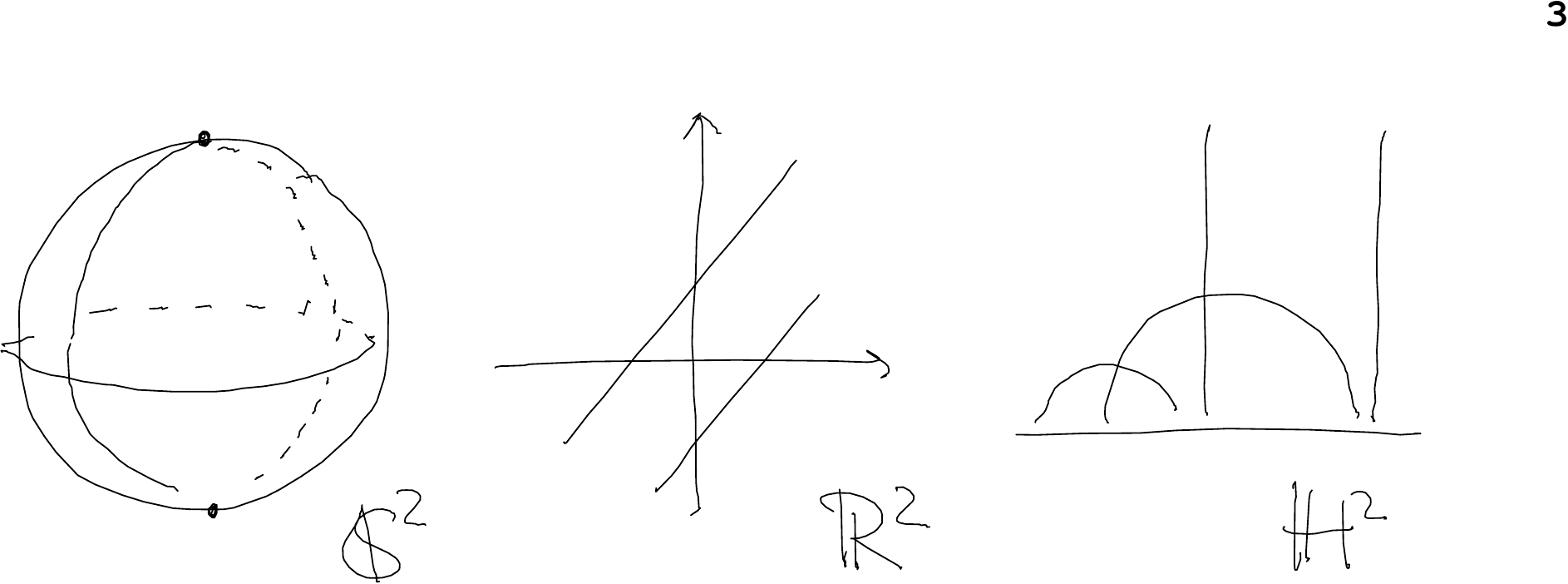}}
\caption[ms]{The model spaces for $\kappa=1, 0, -1$ (left to right). }
\label{fig_03}
\end{center}
\end{figure}

\begin{definition}[Geodesic triangles]
 A \emph{(geodesic) triangle} $\Delta=\Delta(x,y,z)$ in a metric space $(X, d)$ is a collection of three geodesic segments $\seg{xy}$, $\seg{xz}$ and $\seg{zx}$, the \emph{sides} of $\Delta$. 
A  \emph{comparison triangle} in the 2-dimensional model space $(M_\kappa^2, d_\kappa)$  for a given triangle $\Delta(x,y,z)$ in $X$ is a triangle in $M_\kappa^2$ such that 
$d_\kappa(\bar x, \bar y) = d(x,y)$,
$d_\kappa(\bar y, \bar z) = d(y,z)$ and 
$d_\kappa(\bar x, \bar z) = d(x,z)$. 
\end{definition}

\begin{exercise}
	Such a comparison triangle always exists and is unique up to congruence if the circumference of $\Delta$ is less than $\frac{2\pi}{\sqrt{\kappa}}=2D_\kappa$. 
\end{exercise} 

\begin{definition}[$\cat(\kappa)$ triangles]
 A triangle $\Delta$ in a metric space $(X,d)$ has the \emph{$\cat(\kappa)$ property } ( or is $\cat(\kappa)$) if its circumference is less than $2D_\kappa$ and 
\begin{equation}
 d(p,q)\leq d_\kappa(\bar p, \bar q) \text{ for all points }  p,q \text{ on the sides of } \Delta
\end{equation}
where $\bar p, \bar q$ are comparison points in $\bar\Delta$ for $p$ and $q$. Here a \emph{comparison point} $\bar p$ of a point $p$ on a side $\gseg{ab}$ of $\Delta$ is a point on the side $\gseg{\bar a\bar b}$ of a comparison triangle $\bar \Delta$ such that $d_\kappa (\bar p,\bar a) = d(p,a)$ and $d_\kappa (\bar p,\bar b) = d(p,b)$. 
A $D_\kappa$-geodesic space in which all triangles are $\cat(\kappa)$ in the sense just defined is a \emph{$\cat(\kappa)$-space}. A space $X$ is \emph{locally $\cat(\kappa)$} if for all 
$ x\in X$ there is $ r_x>0$ such that $B_{r_x}(x)$ with the restricted metric is a $\cat(\kappa)$-space. 
\end{definition}

\begin{figure}[htbp]
\begin{center}
   	\resizebox{!}{0.4\textwidth}{\includegraphics{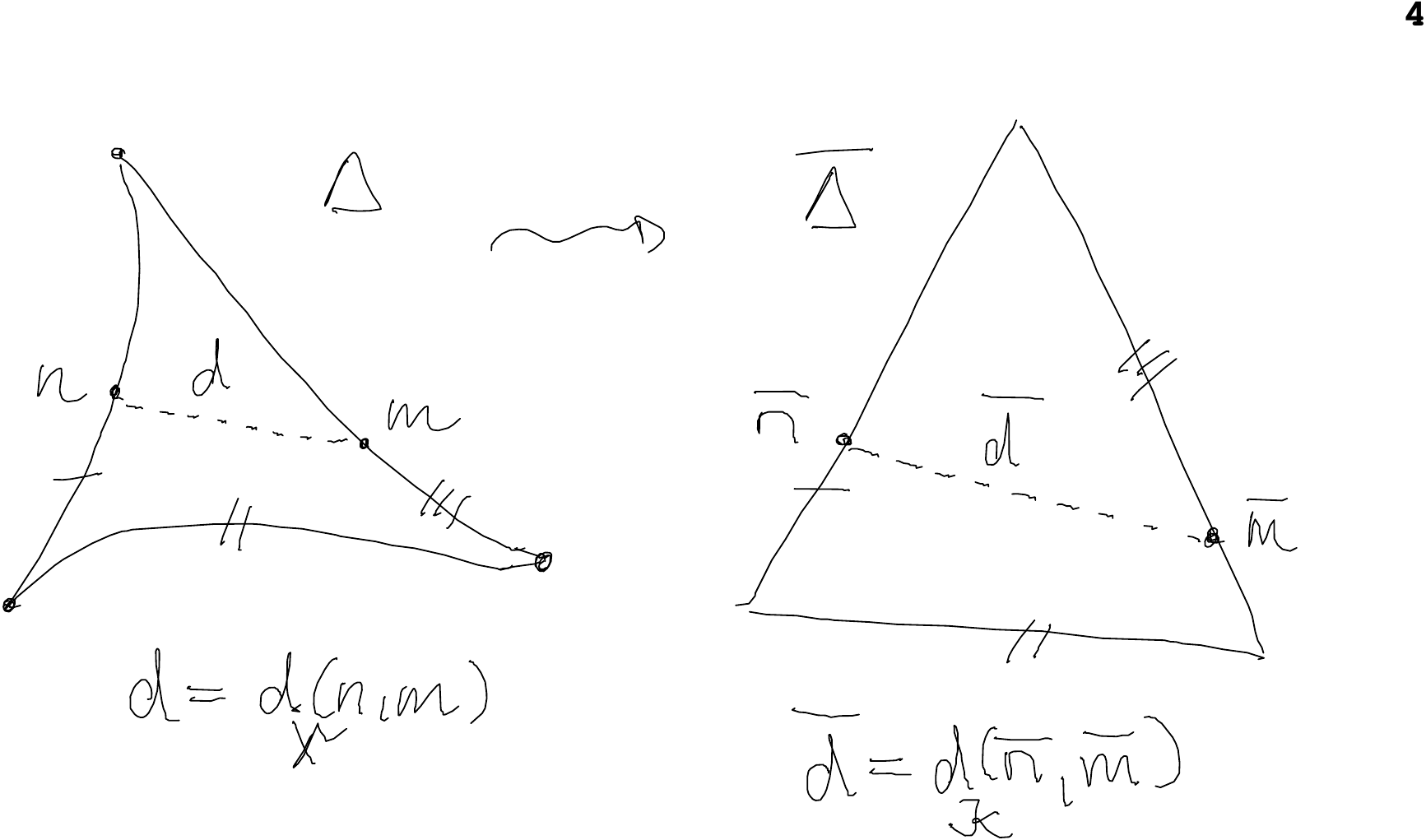}}
\caption[cat]{The $\cat(\kappa)$ property for triangles.}
\label{fig_04}
\end{center}
\end{figure}

\begin{remark}[Origin of the acronym CAT]
The acronym CAT probably stands for {C}artan, {A}lexandrov and {T}oponogov who where among the first describing spaces of this kind. Locally $\cat(\kappa)$ spaces are sometimes calles \emph{non-positively curved} and $\cat(\kappa)$-spaces are also known as Alexandrov spaces with upper curvature bound. 
\end{remark}

\begin{exercise}
	Verify the following
	\begin{enumerate}
		\item $CAT(\kappa) \rightarrow CAT(\kappa')$ for all $\kappa' >\kappa$. 
		\item $\E^2$ is CAT(0). Or more generally (and harder) $M_\kappa^2$ is CAT($\kappa$). 
		\item Simplicial trees with the metric induced by defining the length of each simplex to be 1 is CAT($\kappa$) for all $\kappa\leq 0$. So we could say they are CAT($-\infty$). And in fact one can prove that they are uniquely characterized by this property. 
		\item The flat torus (e.g. modelef by the unit square in $\R^2$ woth opposite sides identified) is locally CAT(0) but not CAT(0). 
	\end{enumerate}
\end{exercise}


\begin{prop}[Properties of $\cat(\kappa)$ spaces]\label{prop:1.10}
If  $(X,d)$ is a $\cat(\kappa)$ space then
\begin{enumerate}
 \item\label{1.10.1} $X$ is $D_\kappa$-uniquely geodesic.
 \item\label{1.10.2} the ball $B_r$ of radius $r$ is contractible for all $r<D_\kappa$. In particular $X$ is contractible if $\kappa\leq 0$. 
 \item\label{1.10.3} every local geodesic of length $<D_\kappa$ is a geodesic.
\end{enumerate}
\end{prop}

\begin{proof}
To prove \ref{1.10.1} let $x,y\in X$ be two points with distance $d(x,y)< D_\kappa$. Further let $\gamma:x\rightsquigarrow y $ and $\gamma':x\rightsquigarrow y$ be two different geodesics connecting $x$ and $y$, i.e. $\gseg{xy}=\gamma([0,d(x,y)])\neq \gamma'([0,d(x,y)])=\gseg{xy}'$.
Choose points $p\in \gseg{xy}, p'\in \gseg{xy}'$ with $d(x,p)=d(x,p')$. 
The comparison triangle $\bar\Delta$ for $\Delta=\Delta(x,p,y)$ with sides $\gamma([0, d(x,p)])$, $\gamma([d(x,p), d(x,y)])$ and $\gseg{xy}'$ is degenerate and hence $\bar p= \bar p'$. From the $\cat(\kappa)$ property we deduce $0=d(\bar p, \bar p')\geq d(p,p') \geq 0$ and hence $p=p'$. 

\begin{figure}[htbp]
\begin{center}
   	\resizebox{!}{0.3\textwidth}{\includegraphics{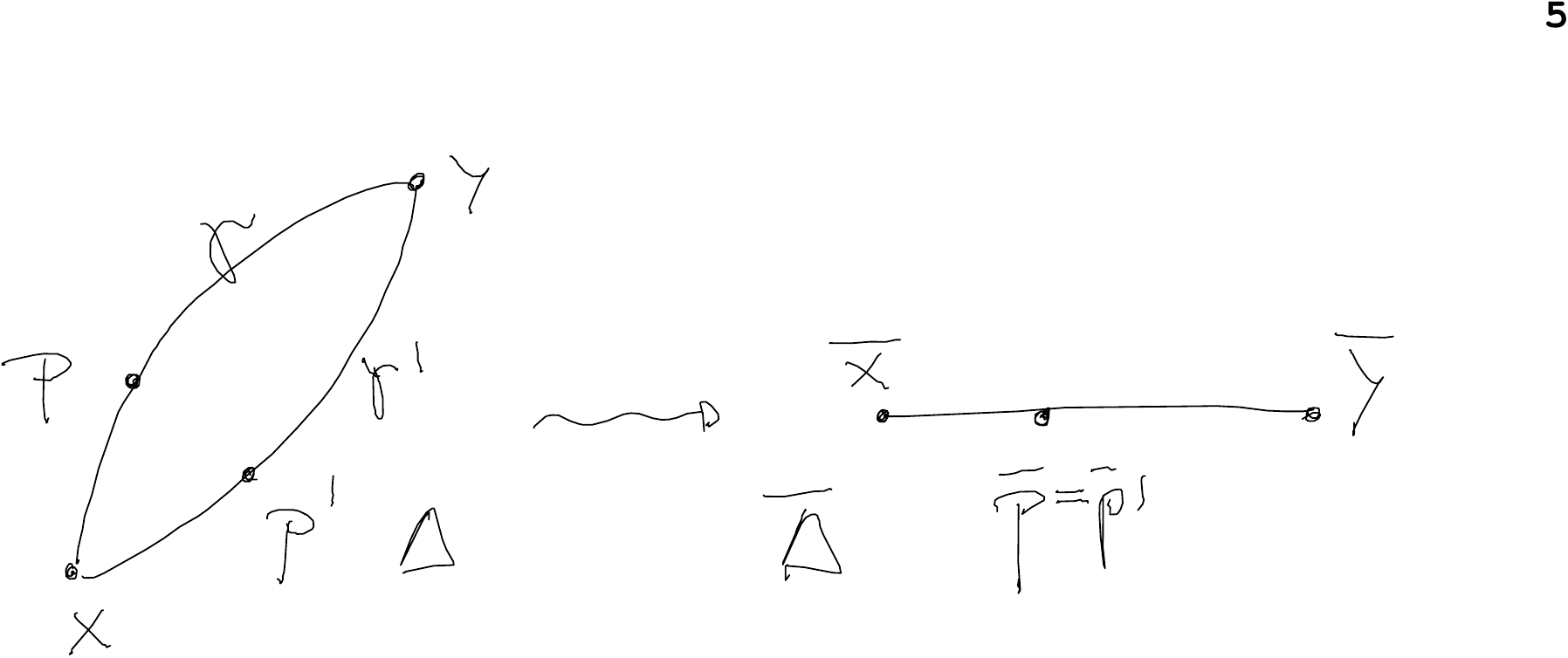}}
\caption[cat]{Illustration of a step in the proof of \ref{prop:1.10}.\ref{1.10.1}.}
\label{fig_05}
\end{center}
\end{figure}

We may prove \ref{1.10.2} using \ref{1.10.1} as follows: For any radius $0<r<D_\kappa$ the map $f:B_r(x)\times [0,1] \to X$ that sends pairs $(y,t)$ to the point $p$ on $\gseg{xy}$ with $d(y,p)=td(x,y)$ is a continuous retraction from $B_r(x)$ to $x$, since by \ref{1.10.1} there is a unique geodesic connecting $x$ and $y$ if $d(x, y)<D_\kappa$. Thus \ref{1.10.2}

\begin{figure}[htbp]
\begin{center}
   	\resizebox{!}{0.3\textwidth}{\includegraphics{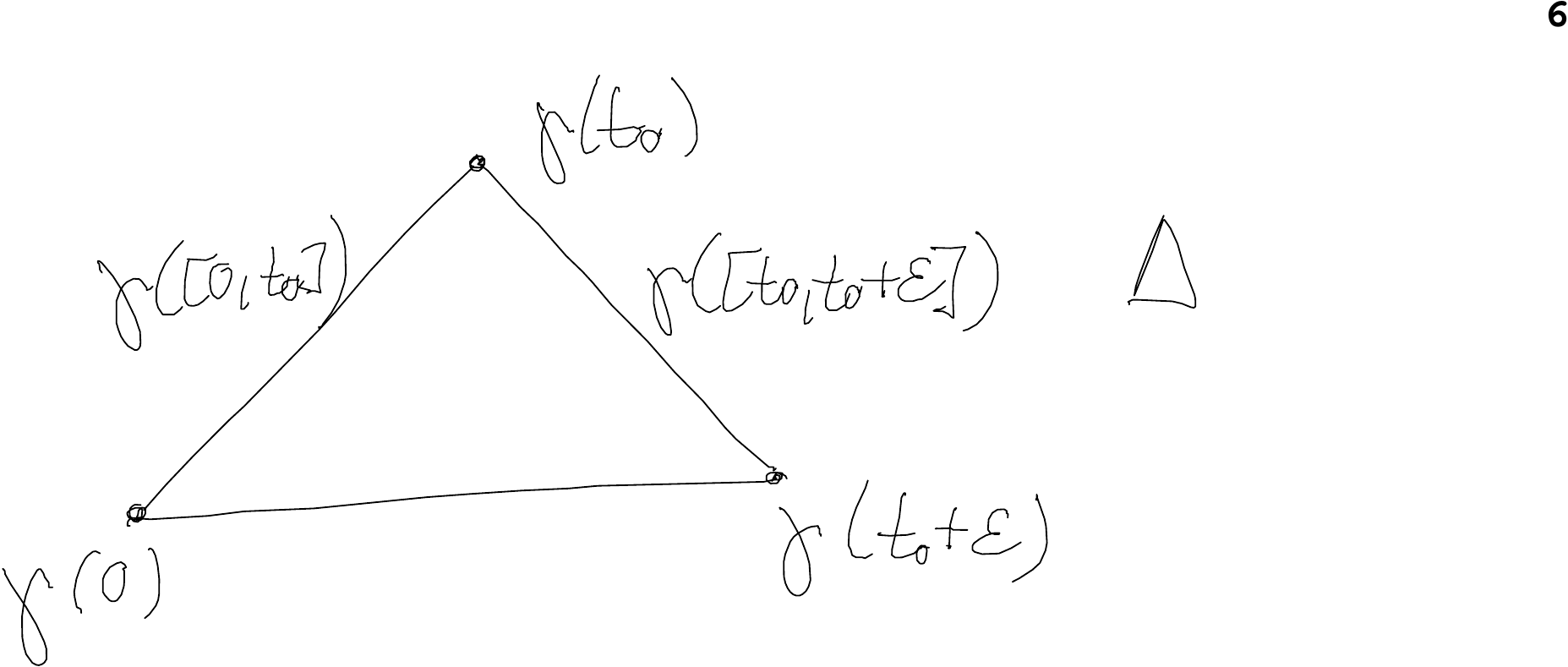}}
\caption[cat]{Illustration of a step in the proof of \ref{prop:1.10}.\ref{1.10.3}.}
\label{fig_06}
\end{center}
\end{figure}

Suppose $\gamma:[0,l] \to X$, $l<D_\kappa$ is a local geodesic and define a set $S\define\{ t\,\vert\, \gamma\vert_{[0,t]} \text{ is a geodesic}\}$. Since $S$ is closed in $[0,l]$ it remains to show that it is also open in $[0,l]$. By definition of local geodesics there is $0<t_0<l$ and $\varepsilon >0$ such that $\gamma$ restricted to $[t_0-\varepsilon, t_0+\varepsilon]$ is a geodesic. 
Consider the triangle $\Delta=\Delta(\gamma(0), \gamma(t_0), \gamma(t_0+\varepsilon))$ where the sides are $\gamma([0,t_0]), \gamma([t_0, t_0+\varepsilon])$ and the unique geodesic $\sigma$ from $\gamma(0)$ to .$\gamma(t_0+\varepsilon)$. 
One can prove that the comparison triangle $\bar \Delta$ is degenerate:  Suppose it was not degenerate and apply the $\cat(\kappa)$ condition to points $x\in \gseg{\gamma(0)\gamma(t_0)}$ and $y\in \gseg{\gamma(t_0)\gamma(t_0+\varepsilon)}$ where both $x$ and $y$ are chosen to lie closer than $\varepsilon$  to $\gamma(t_0)$. This will give a contradiction to the fact that $\gamma\vert_{[t_0-\varepsilon, t_0+\varepsilon]}$ is a geodesic. Thus $\bar\Delta$ is in fact degenerate and with this we may conclude (easily calculate) that $\gamma\vert_{[0, t_0+\varepsilon]}$ is a geodesic. Moreover $d(\gamma(0), \gamma(t_0+\varepsilon)) = t_0+\varepsilon$ and $(t_0, t_0+\varepsilon) \subset S$. Therefore $S$ is open.  
\end{proof}

We will now list several other important properties of non-positively curved spaces. References for proofs will be given. 

\begin{thm}[Cartan-Hadamard]
If a complete, locally $\cat(0)$ metric space is simply connected then it is $\cat(0)$.
\end{thm}
\begin{proof} See \cite[Prop. II.4.1.(2)]{BH}. \end{proof}

\begin{definition}
A subset $C$ f a metric space $X$ is \emph{convex} if $\forall x,y\in C$ there exists a geodesic $\gamma:x\rightsquigarrow y$ and all geodesic segments $\gseg{xy}$ are contained in $C$.
\end{definition}

\begin{prop}[Projections onto convex sets]\label{prop:1.13}
Let $X$ be a complete $\cat(0)$-space and $A\subset X$ a  closed convex subset. Then
\begin{enumerate}
 \item $\forall x\in X$ there exists a unique point $\pi_A(x) \in A$ such that
		$$d(x,\pi_A(x)) =\inf_{a\in A} d(x,a) .$$
 \item $\pi_A:X\to A\,:x\mapsto \pi_A(x)$ is distance non-increasing, that is 
		$$d(\pi_A(x),\pi_A(y))\leq d(x,y) \text{ for all } x,y\in X.$$
 \item  if $y\in\gseg{x\pi_A(x)}$ then $\pi_A(x)=\pi_A(y).$ 
\end{enumerate}
\end{prop}
\begin{proof} See \cite[Prop II.2.4]{BH} 
\end{proof}

\begin{definition}[Isometry]\label{def:1.15}
An \emph{isometry} $\phi:X\to X'$ between metric spaces $(X,d)$ and  $(X',d')$ is a bijective map such that 
$$d(x,y)=d'(\phi(x), \phi(y)) \,\forall x,y\in X.$$
If $X'=X$ we write $\Isom(X)$ for the group of all isometries $\phi:X\to X$.
\end{definition}

\begin{prop}[Fixedpoint theorem]\label{prop:1.14}
The fixed point set of an isometry 
of a $\cat(0)$-space 
is closed and convex.
\end{prop}
\begin{proof}(Sketch)\newline
First prove that fixed-point sets in Hausdorff-spaces are closed. Thus $\fix(\phi)$, where $\phi:X\to X$ is an isometry of a $\cat(0)$-space,  needs to be closed. 
In order to see that the fixed-point set is convex it is enough to prove that geodesics between fixed points $x,y$ are pointwise fixed. But this may be deduced from the fact that $\cat(0)$-spaces are uniquely geodesic, see Proposition~\ref{prop:1.10}.\ref{1.10.1}.
\end{proof}

\newpage
\section{Cubical complexes and Gromov's link condition}

We now introduces cubical complexes and Gromov's combinatorial characterization of the CAT(0) condition in this class of spaces. 

\begin{definition}[Cubes] Let $C=[0,1]^n\subset \R^n$ be an \emph{$n$-cube}. A \emph{codimension 1 face} of $C$ is given by $$F_{i, \epsilon}\define\{x\in C\vert x_i=\epsilon\} \text{ for } \epsilon\in\{0,1\}, i=1, \ldots, n.$$
All other (proper) faces of $C$ are non-emty intersections of codimension 1 faces. Sometimes it will be useful to consider $C$ also as a face of itself. We say that $x$ is an \emph{inner point} of $C$ if $x$ is not contained in any (proper) face of $C$. 
\end{definition}

\begin{figure}[htbp]
\begin{center}
   	\resizebox{!}{0.45\textwidth}{\includegraphics{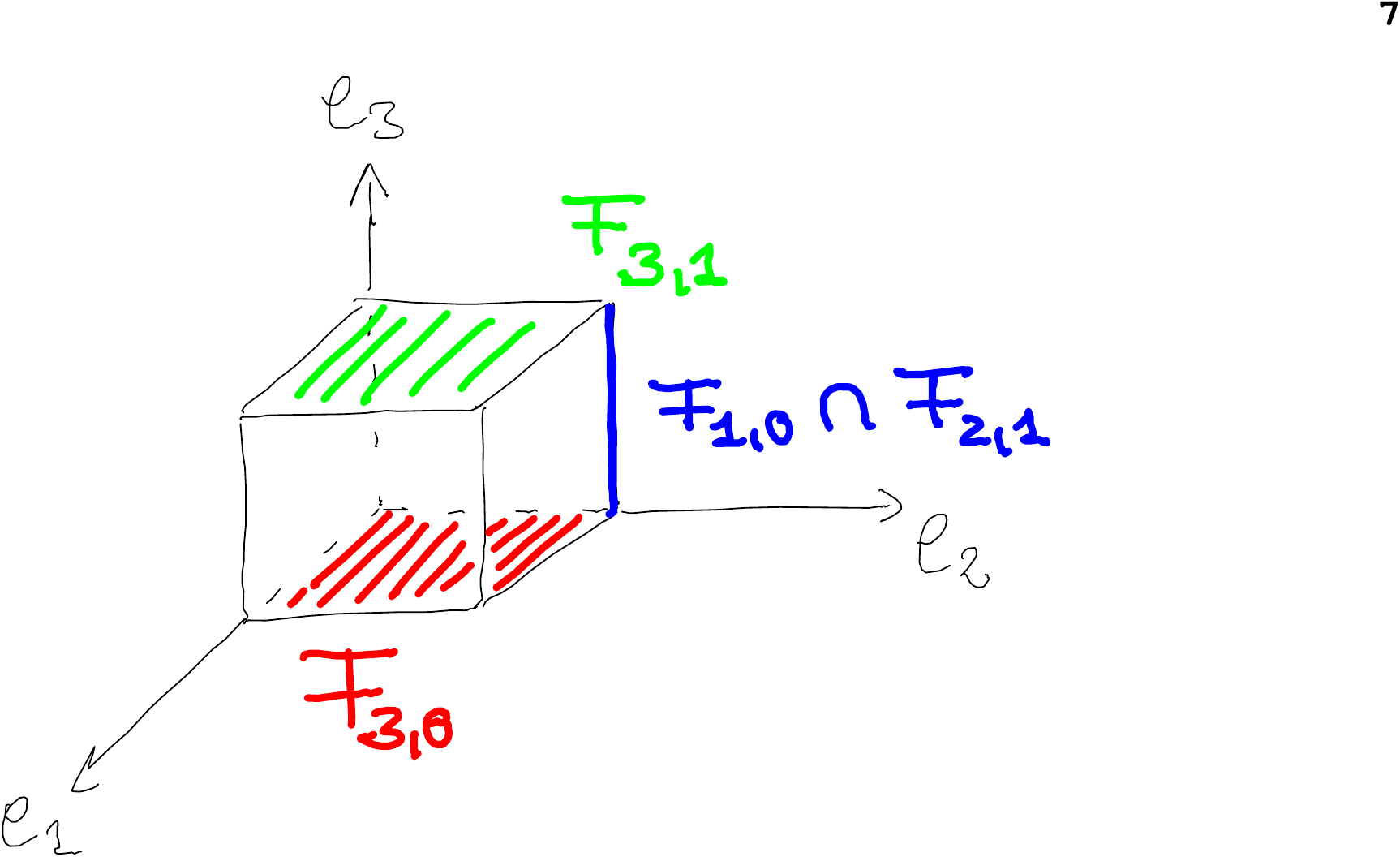}}
\caption[cube]{Faces of a 2-cube. }
\label{fig_07}
\end{center}
\end{figure}


\begin{definition}[Cubical complexes]\label{def:2.1}
Let $C, C'$ be two cubes with faces $F\subseteq C$ and $F'\subseteq C'$ \footnote{Note that here possibly $F=C$ or $F'=C'$}. A \emph{glueing} of $C$ and $C'$ is an isometry $\phi: F\to F'$.  

Suppose  $\cubes$ is a set of cubes and $\glue$ a family of gluings of elements of $\cubes$, that is $\forall C\in\cubes$ there is $n_C\in\N$ such that $C\cong[0,1]^{n_C}$ and every $\phi\in\glue$ is an isometry $\phi:F\to F'$ where $F,F'$ are faces of cubes $C,C' \in \cubes$. 
Assume $(\cubes, \glue)$ satisfies the following two conditions
\begin{enumerate}
 \item\label{2.1.1} No cube is glued to itself.
 \item\label{2.1.2} For all $C\neq C' \in \cubes$ there is at most one gluing of $C$ and $C'$. 
\end{enumerate}
then this pair defines the set of a \emph{cubical complex} $(X,d)$ by
$$X\define (\bigsqcup_{C\in\cubes} C)\diagup_\sim$$
where $\sim$ is the equivalence relation generated by 
$$\{x\sim\phi(x) \;\vert\; \phi\in\glue, x\in \dom(\phi)\}.$$
The metric $d$ on $X$ is the length metric induced by the restricted Euclidean metric on each cube in $\cubes$. 
\end{definition}


\begin{definition}\label{def:2.2}
The distance of a pair of points $x,y$ of a metric space $(X,d)$ measured in the \emph{length metric} $d_l$  is defined to be the infimum of the length of $\gamma$ where $\gamma$ is a rectifiable curve connecting $x$ and $y$. If no such curve exists we put $d_l(x,y)=\infty$. 
\end{definition}

Here \emph{rectifiable} means to have finite length, where the \emph{length} of a curve $\sigma: [a,b]\to X$ is given by 
$$l(\sigma)=\sup_{a=t_0\leq\ldots \leq t_n=b } \sum_{i=0}^{n-1}d(\sigma(t_i), \sigma(t_{i+1})).$$

\begin{example}\label{ex_2.2}
The length metric on $X\define \R^2\setminus Q_1$, where $Q_1=\{(x,y)\vert x>0, y>0\}$ is the open first quadrant, differs from the restricted Euclidean metric $d_e$ on the same set. In particular $(X, d_l)$ is a geodesic space whereas $(X, d_e)$ is not. (Compare Figure~\ref{fig_08}.)
\end{example}

\begin{figure}[htbp]
\begin{center}
   	\resizebox{!}{0.3\textwidth}{\includegraphics{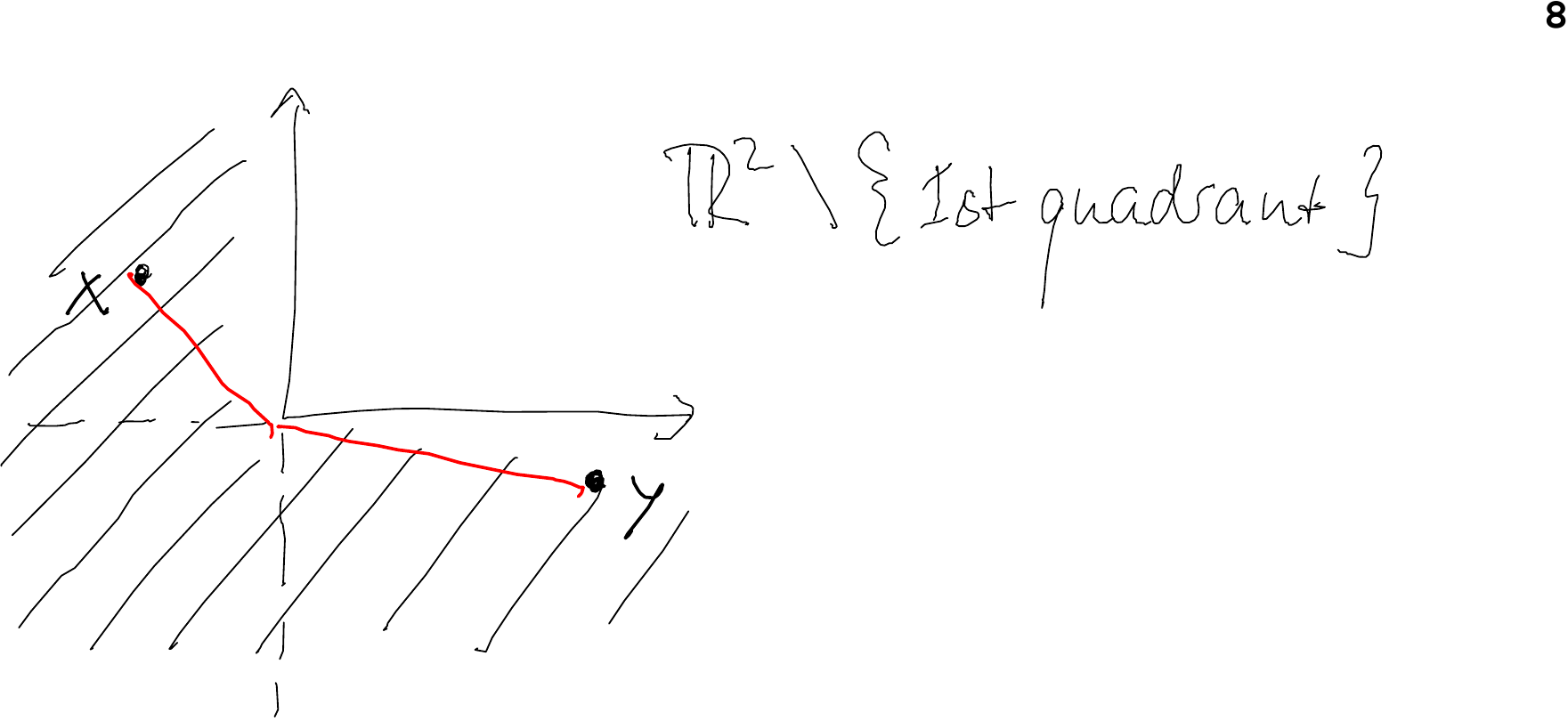}}
\caption[R2-Q1]{The space$X\define \R^2\setminus Q_1$. Compare Example~\ref{ex_2.2}.}
\label{fig_08}
\end{center}
\end{figure}

The following is an easy consequence of Definition~\ref{def:2.1}
.\begin{prop}\label{prop:2.3}
 \begin{itemize}
  \item\label{2.3.1} The restriction of the quotient map $\proj:\bigsqcup_{C\in \cubes}\to X$ to one cube $C\in \cubes$ is injective. 
  \item\label{2.3.2} The intersection of two cubes in $X$ is either empty or a face of both (here a face might be the whole cube).
 \end{itemize}
\end{prop}

Because of \ref{prop:2.3}.\ref{2.3.1} we may identify a cube $C\in\cubes$ with its image in $C$ and write $C\in X$. 

\begin{example}\label{ex:2.4}
The following examples of cubical complexes will be illustrated in Figure~\ref{fig_09} below.
\begin{enumerate}
 \item Graphs with the metric induced by putting all edges to be geodesic segments isometric to the unit interval $[0,1]$. Each edge is a cube glued to its neighboring edges. In particular trees are cube complexes. 
 \item $\R^n$ carries the structure of a cube complex (i.e.\ may be cubulated). The \emph{standard cubing} of $\R^n$ is such that each subset of the form 
	$$\{(x_1, x_2, \ldots, x_n) \;\vert\; a_i<x_i<a_i+1\},\, a_i\in \Z \text{ for all }i$$
 is the image of a cube.
 \item The torus can be presented as a cubical complex: Take the subset $\{(x,y)\in\R^2 \vert 0\leq x,y\leq 3\}$ pf $\R^2$ with the (restricted) standard cubing and identify pairwise the opposite sides of this square. 
\end{enumerate}
\end{example}

\begin{figure}[htbp]
\begin{center}
   	\resizebox{!}{0.6\textwidth}{\includegraphics{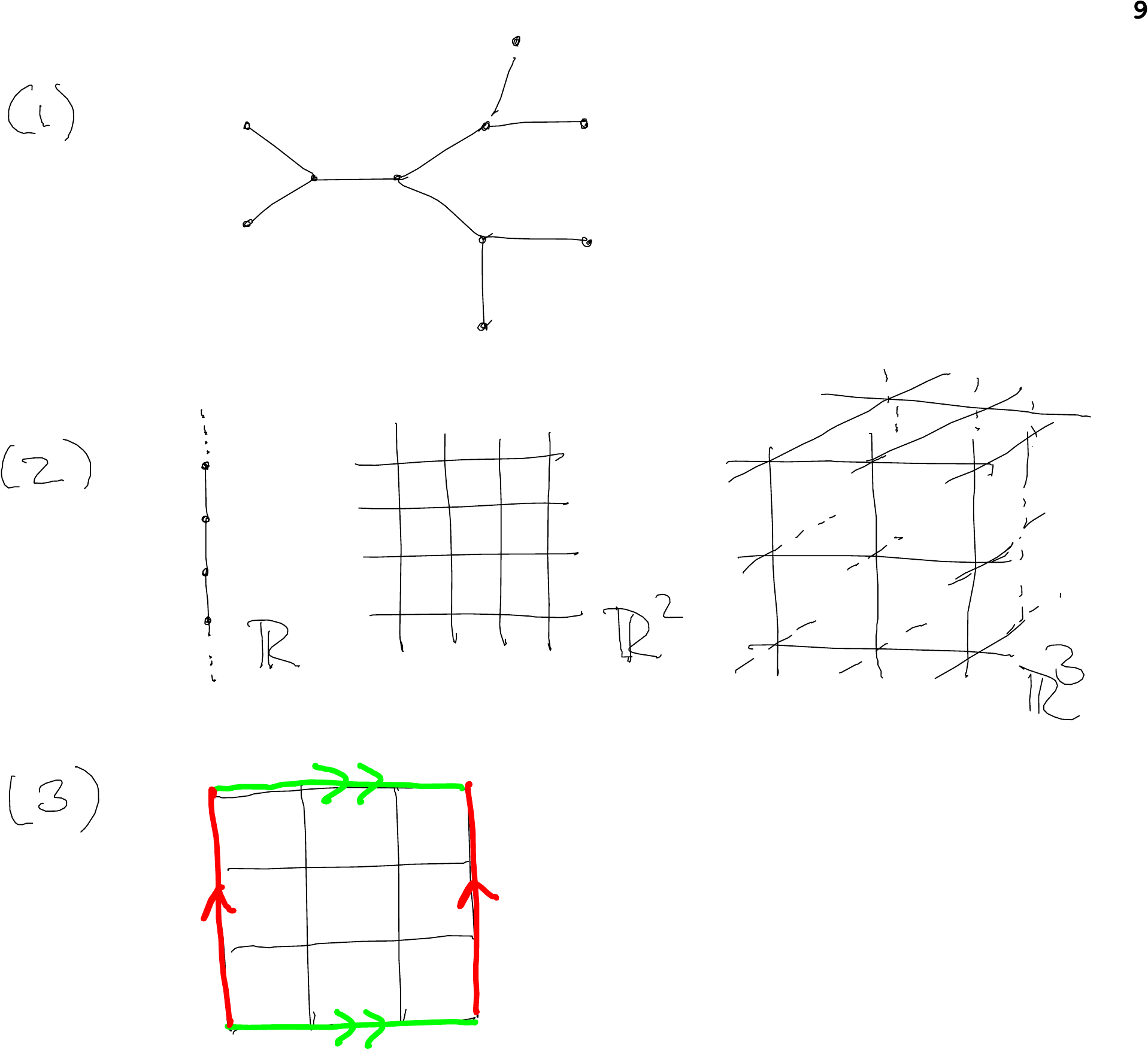}}
\caption[cube complexes]{Examples of cubical complexes.}
\label{fig_09}
\end{center}
\end{figure}

We will now define the polyhedric metric on a cubical complex $X$. 

\begin{definition}
Let $x$ and $y$ be two points in $X$. A \emph{string} $\Sigma$ from $x$ to $y$ is a sequence of points $xi, i=1,\ldots, m$ such that $x_0=x, x_m=y$ and $\forall i=0, \ldots, m$ there exists a cube $C_i$ containing $x_i$ and $x_{i+1}$. 
The \emph{length} of a string $\Sigma$ is given by $l(\Sigma)\define \sum_{i=0}^{m-1} d_{C_i}(x_i, x_{i+1})$ where $d_{C_i}$ is the Euclidean metric on $C_i$. 
\end{definition}

The length of a string is well defined by \ref{def:2.1}.\ref{2.1.2}

\begin{figure}[htbp]
\begin{center}
   	\resizebox{!}{0.3\textwidth}{\includegraphics{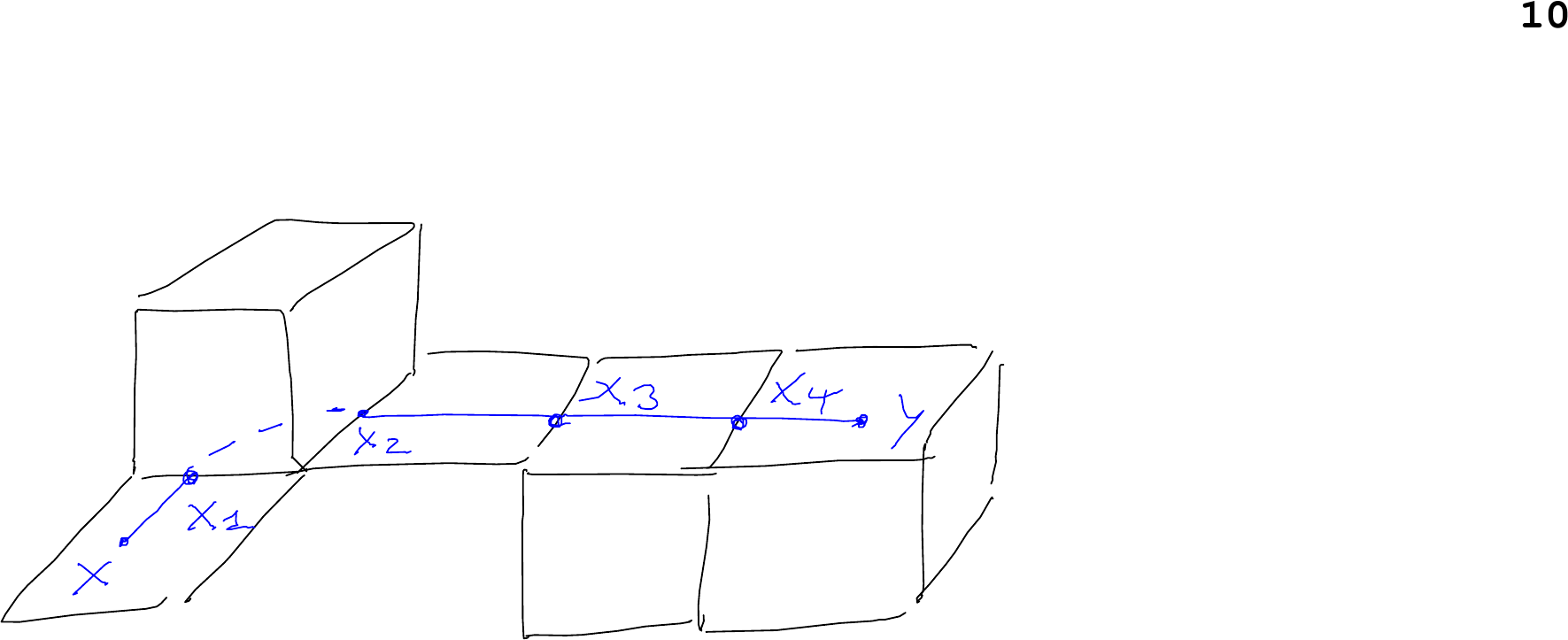}}   	
\caption[string]{A string in a cubical complex.}
\label{fig_10}
\end{center}
\end{figure}

\begin{prop}\label{prop:2.6}
Suppose a cube complex $X$ is string-connected, that is for all $x,y\in X$ there exists a tring from $x$ to $y$, and let $d:X\times X\to \R$ be defined by 
$$d(x,y)\define \inf\{L(\Sigma\vert \Sigma \text{ is a string from } x \text{ to } y\}. $$
Then $(X,d)$ is a metric space and 
$$d(x,y)=\inf\{l(\gamma) \vert \gamma\text{ is a rectifiable curve }x\rightsquigarrow y\},  \;\forall x,y\in X.$$
\end{prop}
In order to prove this proposition we need a technical lemma. 

\begin{lemma}\label{le:2.7}
Let $x$ be a point in a connected cubical complex $X\neq\{pt\}$. Let further $C$ be a cube containing $x$. Then $\epsilon(x,C)$, as defined below,  is independent of the choice of $C$.
$$\epsilon(x,C)\define \inf\{d_C(x,F) \vert F\subset C \text{ and } x\notin F\} .$$
\end{lemma}
\begin{proof}
If for a fixed $x$ there is a unique cube containing $x$ in its interior, then the assertion is clear. In case that $x$ is a corner of $C$ (i.e.\ all coordinates of $x$ are eithhttp://dict.leo.orghttp//www.opensuse.org/de/er $0$ or $1$), then $x$ is a corner in every cube containing it and $\epsilon(x,C)=1$.

Suppose now that $x\in C$ be neither an interior point of $C$  nor a corner. Then let $F_C\subset C$ be the face of minimal dimension such that $x\in F_C$. Such a face exists for every cube containing $x$ and its dimension is always the same. 
The (unique) point  $y\in C$ with $d(x,y)=\epsilon(x,C)$ is contained in a face of $F_C$. 
Thus 
$$\epsilon(x,C)=\inf\{d_{F_C}(x,G) \vert G \text{ is a face of } F_C\}.$$
Since $F_C$ and $F_{C'}$ need to be isometric for cubes $C, C'$ containing $x$ the claim follows.
\end{proof}

\begin{figure}[htbp]
\begin{center}
	\resizebox{!}{0.3\textwidth}{\includegraphics{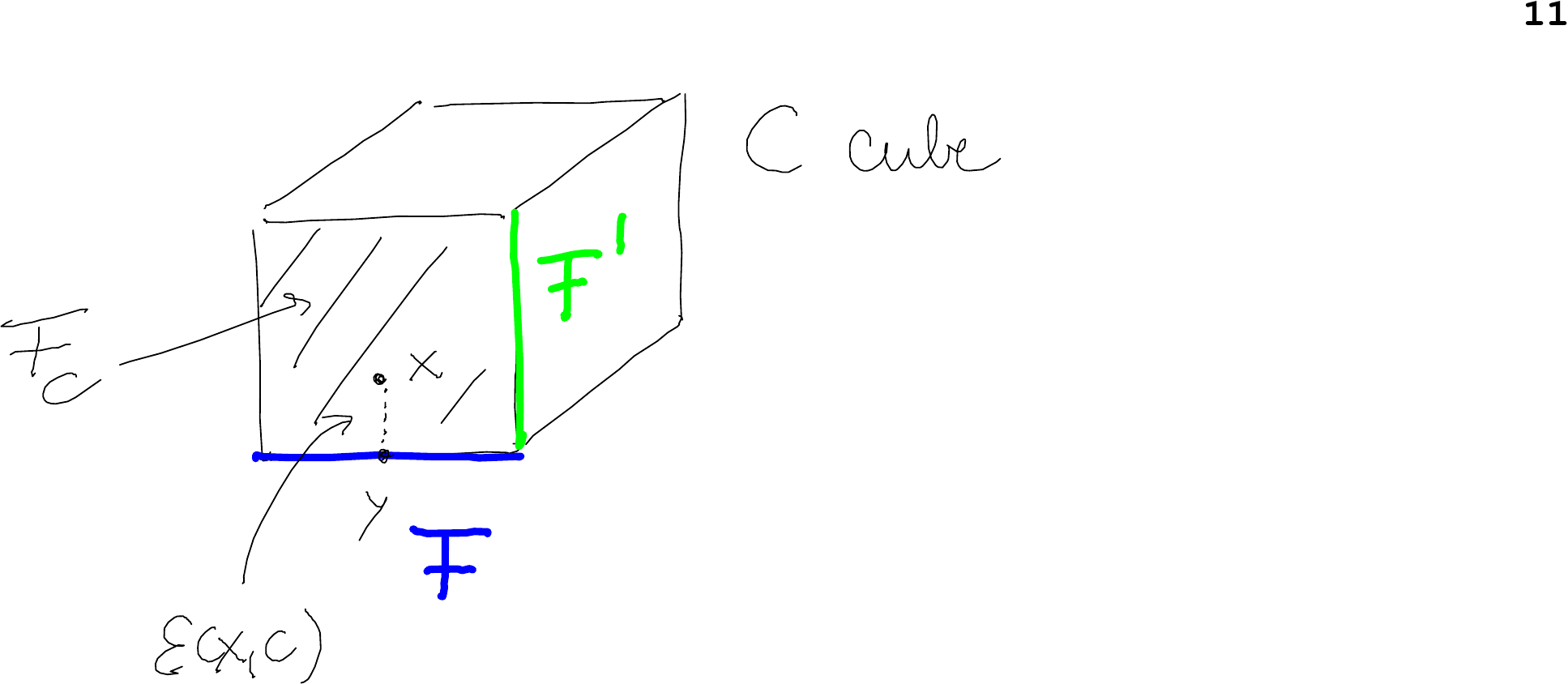}}
\caption[proof]{Illustration of the proof of Lemma~\ref{le:2.7}.}
\label{fig_11}
\end{center}
\end{figure}

\begin{proof}[Proof of Proposition~\ref{prop:2.6}]
We first proof that $d$ is a metric. 
Symmetry of $d$ follows directly from the definition by reading strings backwards.
Connecting strings $\Sigma:x\rightsquigarrow y$ and $\Sigma:y\rightsquigarrow z$ in series gives us strings from $x$ to $z$ and hence the triangle inequality holds. 
Positivity of $d$ comes from the fact that if a cube $C$ contains $x$ then it contains every $y$ with $d(x,y)<\epsilon(x,C)$. Hence in this case $d(x,y)=d_C(x,y)$ where $d_C$ is the (restricted) Euclidean metric on $C$. The fact that $d$ is indeed a length metric is clear from the definition of $d$. 
\end{proof}

\begin{definition}\label{def:2.8}
 We say that a cubical complex $(X,d)$ defined by the pair $(\cubes, \glue)$ is 
\emph{finite dimensional} if there is a global upper bound on the dimension of cubes in $\cubes$.
It is \emph{locally finite} if no point of $X$ is contained in infinitely many cubes. 
\end{definition}

We leave the proof of the following proposition as an exercise. 
\begin{prop}\label{prop:2.9}
 A cubical complex is a complete geodesic metric space if it is either finite dimensional or locally finite. 
\end{prop}
\begin{proof}
For a proof in case of finite dimensional $X$ see \cite[Thm. I.7.19]{BH}.  If $X$ is locally finite show first that  $X$ is proper (use theorem of Hopf-Rinow \cite[Thm I.3.7]{BH} to do so) and show then that $X$ is geodesic. 
\end{proof}

\begin{example}
Infinite dimensional non-complete cubical complexes exist. Take for example an ordered basis of an infinite dimensional Hilbert space $H$ and in $H$ the ascending union of the unit cubes spanned by the first $k$ basis elements. The resulting cubical complex does not have a bound on the dimension of its cube and will not be complete. 
\end{example}

\begin{remark}
There are a priori two topologies on a cube complex $X$. For one the quotient topologie $\mathcal T_p$ and on the other hand the topology $\mathcal T_d$ which is induced by the length metric. 
One has in general $\mathcal T_d\subset \mathcal T_p$ and equality iff $X$ is locally finite. 
\end{remark}

We will now prove a criterion which allows us to easily characterize the $\cat(0)$-cubical complexes via a simple property of their links. This chriterion was established by Gromov 
and is one of the main reasons why cube complexes are so popular in geometric group theory. In general curvature testing is hard and even for simplicial complexes no good characterizations are known. So it is remarkable how easy testing for $\cat(0)$ in the class of cubical complexes is. 

\begin{figure}[htbp]
\begin{center}
	\resizebox{!}{0.3\textwidth}{\includegraphics{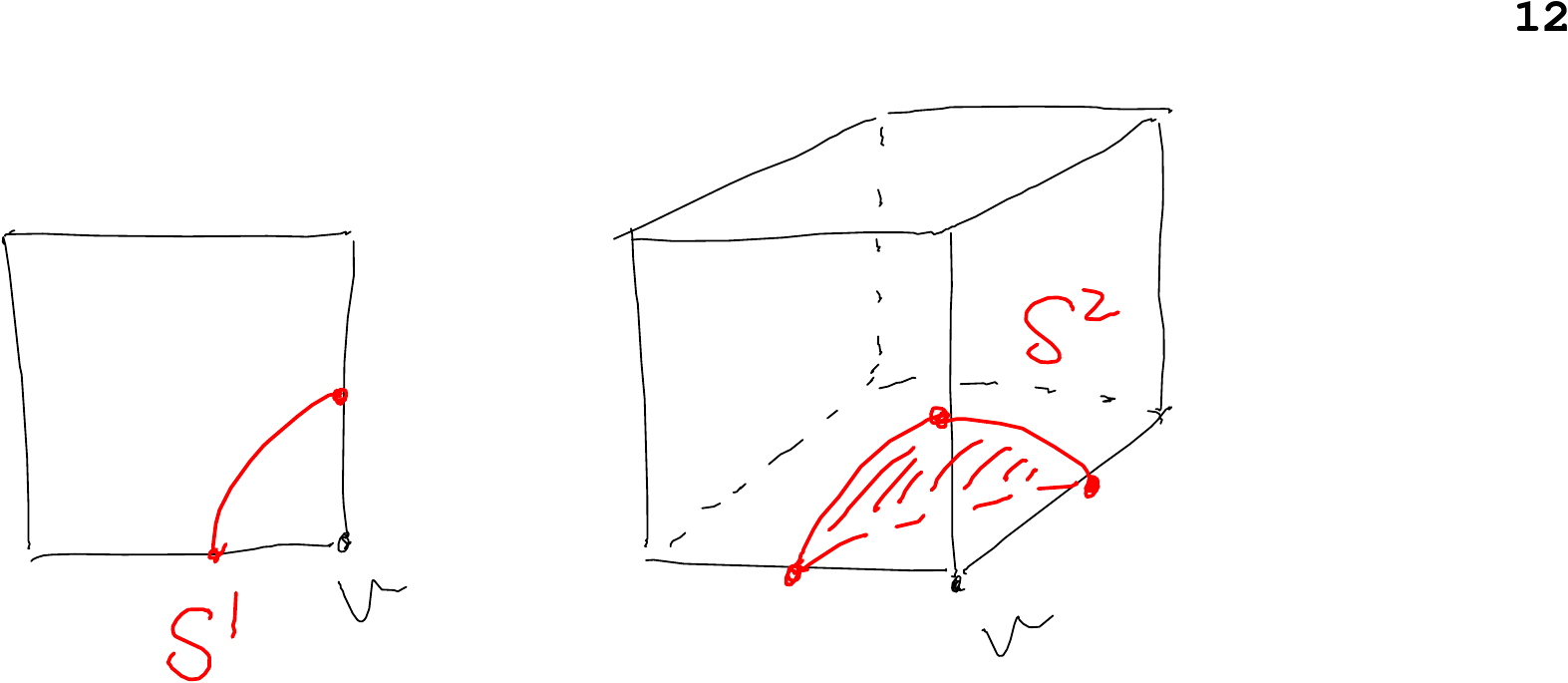}}
\caption[ARSC]{Simplices of an all right spherical complex.}
\label{fig_12}
\end{center}
\end{figure}

\begin{definition}\label{def:2.10}
We consider the following shapes $S^n$ which are  ''rightangles cutouts'' of spheres. That is for $n\geq 1$ fix a cube $C^{n+1}\cong [0,1]^{n+1}$ and some $\varepsilon > 0$. Let $v$ be a corner of $C$, then the \emph{all-right spherical shape} $S^n$ of dimension $n$ is given by 
$$S^n\define \{y\in C \;\vert\; d(v,y) = \varepsilon\}.$$
The \emph{faces} of $S^n$ are the intersections of $S^n$ with faces of $C$ (and are isomorphic to some $S^k$ with $k<n$. Distances of points $p,q$ in $S^n$ are measured in the angular metric, i.e.\ 
$d_{S^n}(p,q)\define \angle_v(\gseg{vp}, \gseg{vq})$, where $\angle_v$ stands for the Eucidean angle at $v$ in $C$.

\begin{figure}[htbp]
\begin{center}
	\resizebox{!}{0.3\textwidth}{\includegraphics{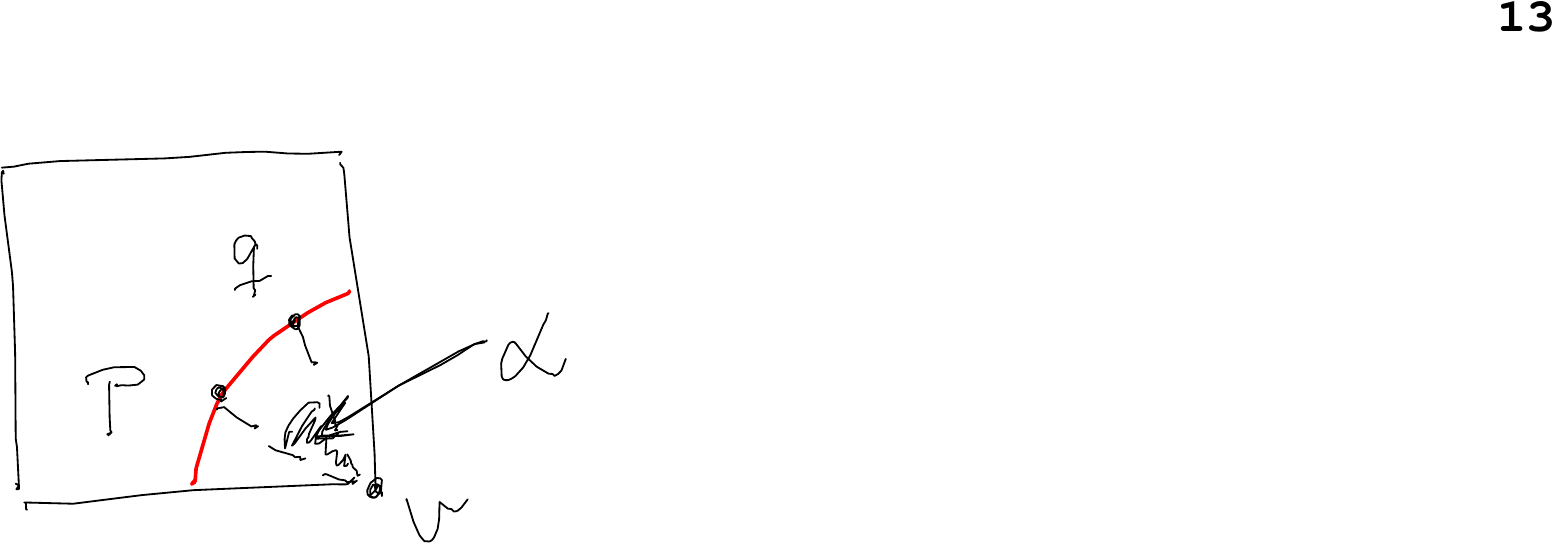}}
\caption[ARSC]{Measuring distances of points in a simplex of an all right spherical complex.}
\label{fig_13}
\end{center}
\end{figure}

An \emph{all-right spherical complex} is a polyhedral complex built out of all-right spherical shapes satisfying gluing rules analogous to the ones given in Definition~\ref{def:2.1} with $\cubes$ replaced by a familiy $\mathcal{S}$ of all-right spherical shapes. 
\end{definition}

Observe that links of all-right spherical complexes are again all-right spherical complexes. 

We will in the following write \emph{spherical complex} when we actually mean an all-right spherical complex, since these are the only ones appearing here. 

\begin{definition}\label{def:2.11}
Let $1>\varepsilon >0$.  The link $\lk(v,X)$ of a corner $v$ in a cubical complex $X$ is the spherical complex obtained by looking at the $\varepsilon$-sphere around $v$ in $X$ equipped with the from $X$ induced simplicial structure. 
\end{definition}

\begin{example}\label{ex:2.12}
An example of links in a cubical complex can be found in Figure~\ref{fig_14}
\end{example}

\begin{figure}[htbp]
\begin{center}
	\resizebox{!}{0.4\textwidth}{\includegraphics{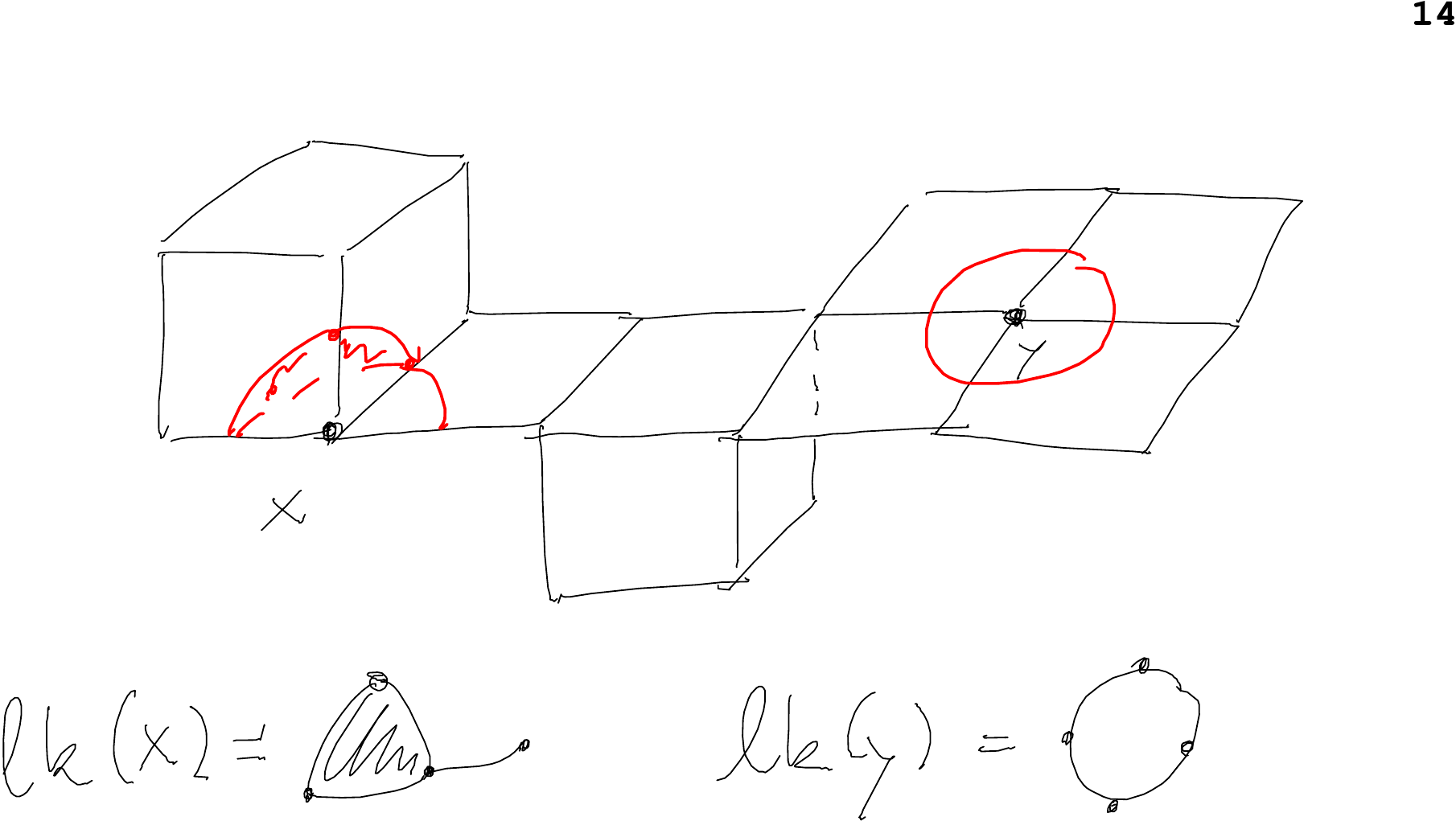}}
\caption[links]{Links in a cubical complex. }
\label{fig_14}
\end{center}
\end{figure}

Links carry in a natural way the structure of an abstract simplicil complex. (Forget the metric and the spherical nature of its simplices.)
\begin{definition}\label{def:2.13}
An abstract simplicial complex $\Delta$ is \emph{flag} iff every subset of vertices $V'\subset \vert(\Delta)$ of the vertices of $\Delta$ which are pairwise connected by an edge span a simplex.  
\end{definition}

In other words a simplicial complex is flag if there are no empty simplices. That is whenever the 1-skeleton of a simplex is there, then the simplex exists. 

\begin{thm}[Gromov's link condition]\label{thm:2.14}
A finite dimensional cubical complex is locally $\cat(0)$ iff all its vertex links are flag simplicial complexes.  
\end{thm}
\begin{proof}
Combining Propositions~\ref{prop:2.15} and \ref{prop:2.16} the assertion follows. 
\end{proof}

\begin{prop}\label{prop:2.15}
A finite dimensional $M_\kappa$-polyhedral (e.g.\ cubical or all-right spherical) complex  is locally $\cat(\kappa)$ iff all its vertex links are $\cat(1)$. 
\end{prop}
\begin{proof} See \cite[Thm II.5.2]{BH}. \end{proof}

\begin{prop}\label{prop:2.16}
A finite dimensional spherical complex $\Delta$ is $\cat(1)$ if and only if it is a flag simplicial complex.  
\end{prop}
\begin{proof}
Suppose first that $\Delta$ is $CAT(1)$. We will prove by induction on the size of an empty simplex that $\Delta$ is flag. 

The case $n=1$ is easy to verify. Let $n=2$. We then need to show that there are no empty triangles. Suppose for a contradiction that there is an empty triangle $x,y,z$. In this empty triangle the distance from $y$ to the midpoint $m$ of the side on $x$ and $z$ is $\frac{3}{4}\pi$. The distance between the comparison points $\bar y$ and $\bar m$ of $y$ and $m$ in the spherical comparison triangle is smaller than $\frac{3}{4}\pi$. But this contradicts the CAT(1) property of $\Delta$. 

We now look at the induction step from $n-1$ to $n$. Suppose the assertion is true for $n-1$. 
Let $\{v_0, \ldots, v_n\}$ be vertices in $\Delta$ that span the $n-1$-Skeleton $S$ of an $n$-Simplex in $\Delta$. Then, by Theorem~\ref{thm:2.14} one has: for all vertices $v=v_i, i=0, \ldots, n$ the link $\lk(v, \Delta)$ is CAT(1) and by induction a flag complex  as the dimension of the link is $n-1$. 
But then the intersection of the link with $S$ is the $n-2$-Skeleton of an $n-1$-Simplex in $\lk(v,\Delta)$ which is in fact the boundary of an $n-1$-Simplex. 
This implies that $S$ is the boundary of an $n$-Simplex and $\Delta$ is flag.

To show the reverse implication we proceed by induction on the dimension of $\Delta$.   
Suppose first $dim(\Delta)=1$. The triangles in a one-dimensional flag simplicial complexes with circumference less than $2\pi$ are automatically degenerate and hence CAT(1). 

For the induction step suppose that $\dim(\Delta)=k$ and that $\Delta$ is flag. Suppose further that the assertion holds for $n-1$-dimensional complexes. It is enough to show that there are no geodesic loops of length less than $2\pi$. Observe that since we are dealing with a right-angled spherical complex one has that 
\[\lk(v,\Delta)\cong S_{\frac{\pi}{2}}=\{x\in\Delta\vert \ d(x,v)=\frac{\pi}{2}\}. \]
Let $\gamma$ be a locally geodesic loop. Let $v$ be a vertex in $\Delta$ such that $c\define \mathrm{Im}(\gamma)\cap B_{\frac{\pi}{2}}(v)\neq\emptyset$.  
 
 We claim that $c$ is a curve of length $\pi$. In case that $v\in c$ the claim follows. In case $v\notin c$ we now construct a surface $S$ in $\overline{B_{\frac{\pi}{2}}(v)}=\{x\in\Delta\vert \ d(x,v)\leq\frac{\pi}{2}\}. $
 The surface $S$ is defined to be the union of all Geodesics of length $\frac{\pi}{2}$ starting in $v$ and containing at least one point of $\gamma$. 
 One then has that $S$ is the union of spherical triangles $D_i$ of hight $\frac{\pi}{2}$ sitting along $\gamma$. 
 See Figure~\ref{fig_15} for an example how this might look like. 
 
 We subdivide $\gamma$ by points $x_i$ defined by the faces of these triangles cutting $\gamma$. 
 The link $\lk(v,\Delta)$ is isomorphic to the CAT(1) cone over $S_{\frac{\pi}{2}}(v)$. This implies that there is an embedding $\iota$ from $S$ to $\mathbb{S}^2$ such that the restriction of $\iota$ onto each triangle $D_i$ is an embedding with the property that $c$ is mapped to a local Geodesic. 
 Assume without loss of generality that $\iota(v)$ is the north pole of the sphere. Then the image $\iota(c)$ connects two points on the equator and intersects the interior of the upper hemisphere non-trivially. Hence $\ell(\iota(c))=\pi$ and thus also $\ell(c)=\pi$ as local geodesics of length up to $\pi$ in $\mathbb{S}^2$ are also global geodesics.\footnote{For details compare Davis-Moussong~\cite{DM} Lemma 2.3.6.} 
 We have thus shown the claim that the length of $c$ is $\pi$. 
  
 For the last step of the proof we suppose that $\ell(\gamma)<\pi$. 
 But then $\gamma$ can not be the union of two curves of length equal to $\pi$ (obviously) and can hence not be contained in the union of  two balls $\bar B_{\frac{\pi}{2}}(v)$ and $\bar B_{\frac{\pi}{2}}(u)$ where $B_{\frac{\pi}{2}}(v)\cap B_{\frac{\pi}{2}}(u) =\emptyset$. 
  
 We hence conclude that the set 
 $B\define \{\text{vertices in }\Delta\vert\ B_{\frac{\pi}{2}}(v)\cap\mathrm{Im}(\gamma)\neq\emptyset\}$
contains at least three elements and the vertices in this set are pairwise connected by edges. But this contradicts the fact that $\gamma$ can not be a union of two geodesics of length $\pi$. 
Thus $\mathrm{Im}(\gamma)\subset \mathrm{span}(B)\subset\Delta$.  
But $ \mathrm{span}(B)$ is a simplex in $\Delta$ as $\Delta$ is flag. Therefore $\gamma$ is contained in a simplex. However, by assumption $\gamma$ is a closed local geodesic. We thus arrive at a contradiction and conclude that the statement must hold. 
\end{proof}

\begin{figure}[htbp]
	\begin{center}
		\resizebox{!}{0.4\textwidth}{\includegraphics{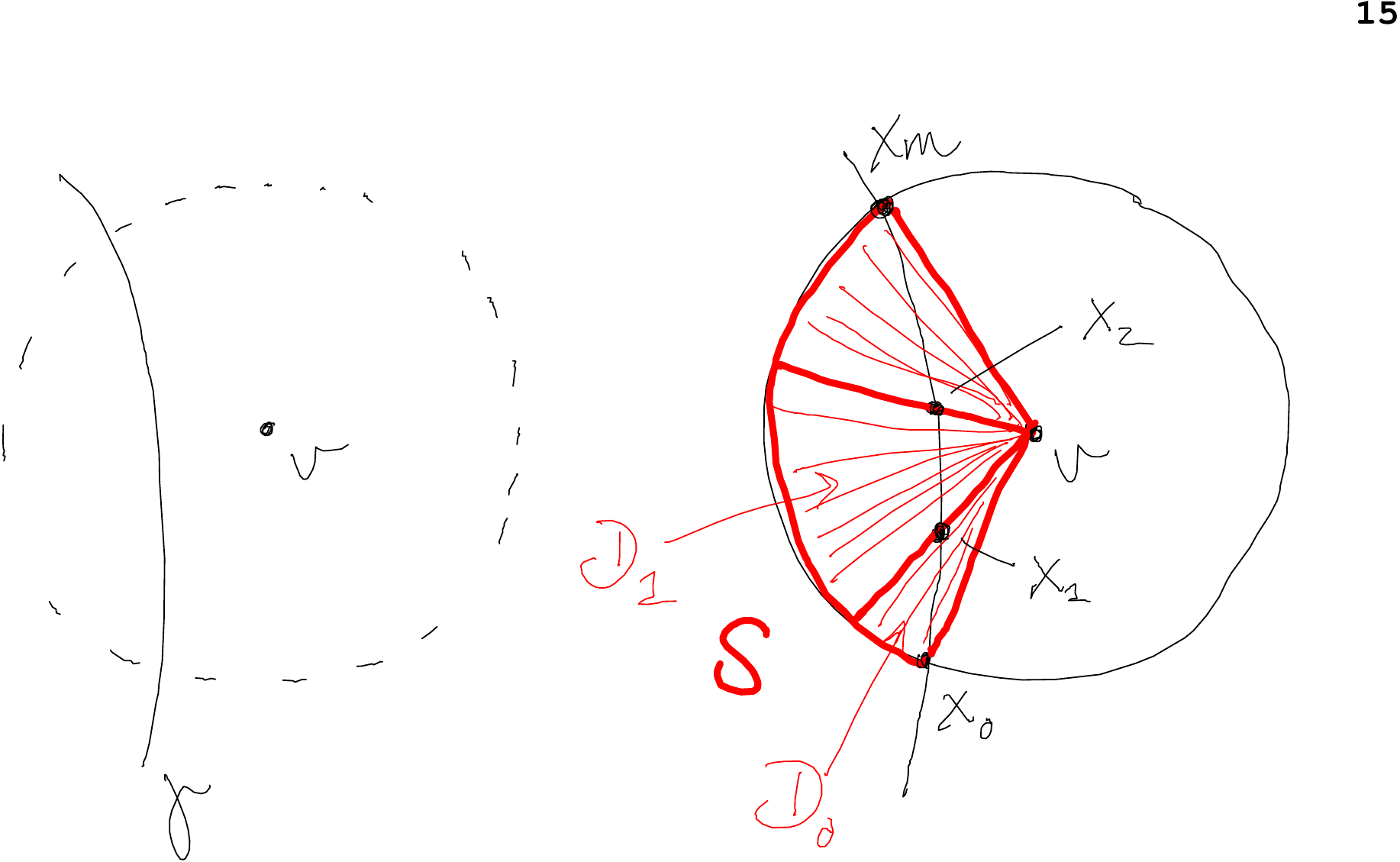}}
		\caption[links]{Compare proof of Proposition~\ref{prop:2.16}.}
		\label{fig_15}
	\end{center}
\end{figure}

In particular a cubical complex $X$ is $\cat(0)$ if and only if it is simply connected and all its vertex links are flag. 

\begin{prop}\label{prop:2.17}
A spherical complex $\Delta$ is $\cat(1)$ if and only if it is locally $\cat(1)$ and there are no locally geodesic circles of length $<2\pi$.  
\end{prop}
\begin{proof}
This proposition is true in a slightly more general setting. For a precise statement and  proof see \cite[Thm. 5.4]{BH}.
\end{proof}

We end this section with a series of examples.

\begin{example}\label{ex:2.18}
Compare Figure~\ref{fig_16} for an illustration of links of the standard cubulation of $\R^n$ as well as a cubulation of a surface of genus 2 (which is actually a VH-complex). 
\end{example}

\begin{figure}[htbp]
	\begin{center}
		\resizebox{!}{0.9\textwidth}{\includegraphics{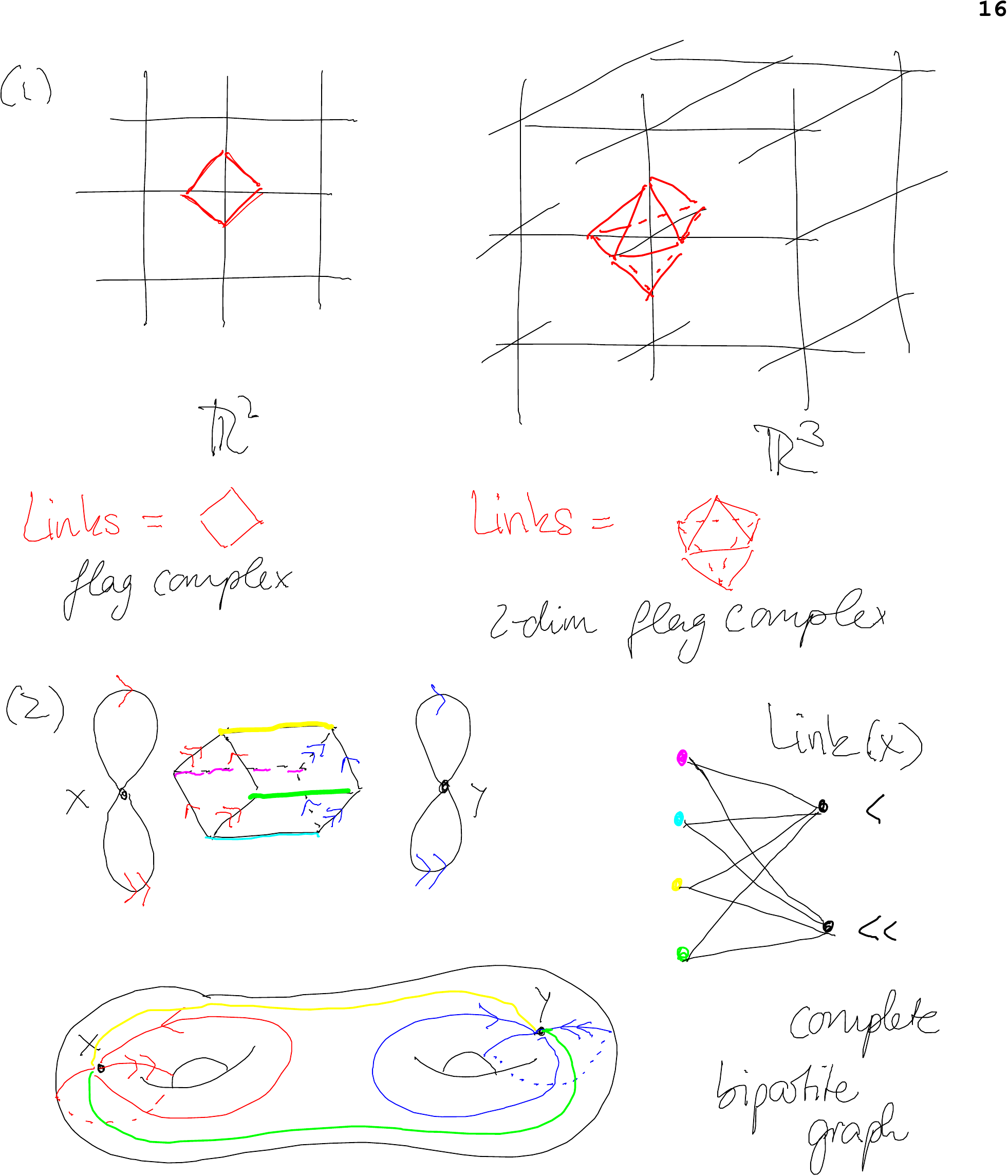}}
		\caption[links]{Links in some cube complexes.}
		\label{fig_16}
	\end{center}
\end{figure}

\begin{example}
The Salvetti complex of a right angled Artin group (RAAG):\newline 
A \emph{right angled Artin group} is defined as follows by means of a simplicial graph $\Gamma=(V,E)$, with $V$ the vertices and $E$ the edges of $\Gamma$.
$$G(\Gamma)\define \{ g_v, v\in V \;\vert\;  [g_u, g_v]=1 \text{ for all edges } (u,v) \in E\}.$$
Is the circumference of $\Gamma$ greater or equal than 4 then the presentation 2-complex of $\Gamma$ is locally $\cat(0)$. However in general higher dimensional cubes are needed in order to obtain a locally $\cat(0)$ complex from a RAAG. 

Define $R(\Gamma)$ to be  the cube complex obtained from the presentation 2-complex by attaching an n-cube for every n-clique\footnote{An \emph{n-clique} is a set of pairwise connected vertices in a graph.} in $\Gamma$. The resulting complex is the \emph{Salvetti complex}. 
\end{example}

One can show the following more general theorem. 

\begin{thm*}
For every finite simplicial graph $\Gamma$ there is a locally $\cat(0)$ cube complex $R(\Gamma)$ such that $G(\Gamma)=\Pi_1(R(\Gamma))$ and the 2-skeleton of $R(\Gamma)$ is the standard presentation 2-complex. 
\end{thm*}


It is not known whether in general an Artin group acts proper and freely on a  $\cat(0)$ cubical complex.

\begin{example}
Compare Figure~\ref{fig_17} for another interesting example of a CAT(0) cubical complex due to Dani Wise.  
\end{example}

\begin{figure}[h]
	\begin{center}
		\resizebox{!}{\textwidth}{\includegraphics{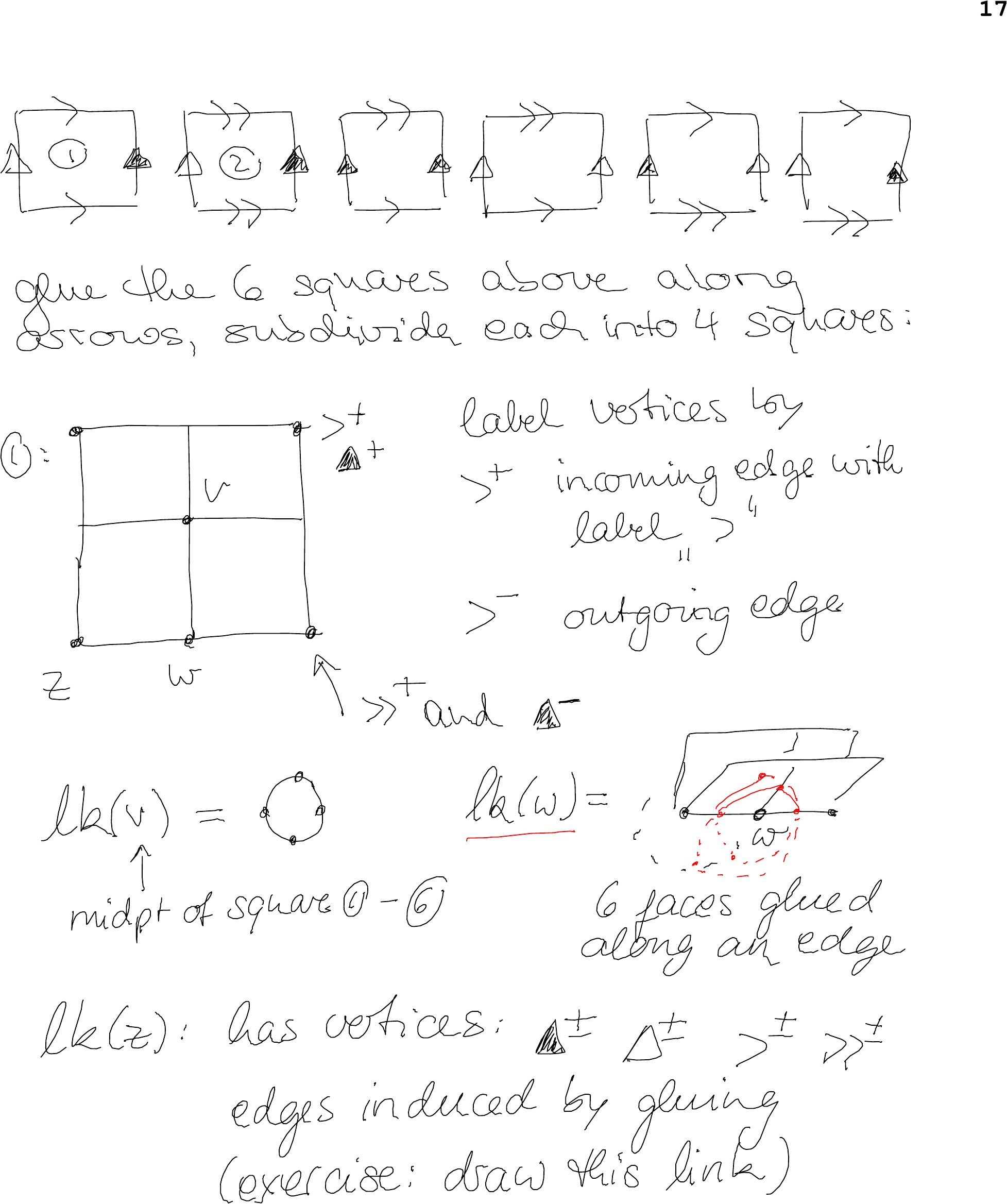}}
		\caption[links]{A nice CAT(0) cube complex.}
		\label{fig_17}
	\end{center}
\end{figure}


%
%

\newpage
\section{Hyperplanes}

One of the nice properties of cubical complexes is that they have hyperplanes. We introduce those in this section and study some of their properties. 

\begin{definition}\label{def:3.1}
A \emph{midcube} (or \emph{hypercube}) $M_i$ of a cube $C= [0,1]^n$ is given by  
$$M_i\define\{x\in C\,\vert\, x_i=\frac{1}{2}\}, i=1,\ldots,n.$$ 

Two edges(=1-cubes) $e, e'$ in a cubical complex $X$ are \emph{square equivalent}, denoted by  $e\sim e'$ if $e$ is opposite $e'$ in some 2-cube in $X$. 
We extend $\sim$ to an equivalence relation on all edges in $X$.\footnote{That is two edges are equivalent if there is a sequence of 2-cubes such that the edges are eventually equivalent along this sequence of cubes.}

A midcube $M$ is \emph{transversal} to an equivalence class $[e]_\sim$ of edges (or simply to $e$)  if the intersection of  $M$ with the 1-skeleton of $X$ consists of midpoints of edges in $[e]_\sim$. We write $M \pitchfork e$.

A \emph{hyperplane} $H$ in $X$ is the union of all midcubes $M$ transversal to a fixed edge $e$. 
$$ H(e)\define \bigcup_{M\pitchfork\ e} M.$$

The \emph{support} $N(H)$. of a hyperplane $H$ is the union of all cubes $C$ in $X$ intersecting $H$ in a midcube. 
\end{definition}

In a $\cat(0)$ cubical complex a hyperplane is a subcomplex intersecting each cube not at all or in precisely one midcube. We will learn more about this later, but first some examples of hyperplanes.  

\begin{example}\label{ex:3.2}
$\cat(0)$ examples:
\begin{enumerate}
 \item Trees
	\begin{figure}[h]
		\begin{center}
		\resizebox{!}{0.15\textwidth}{\includegraphics{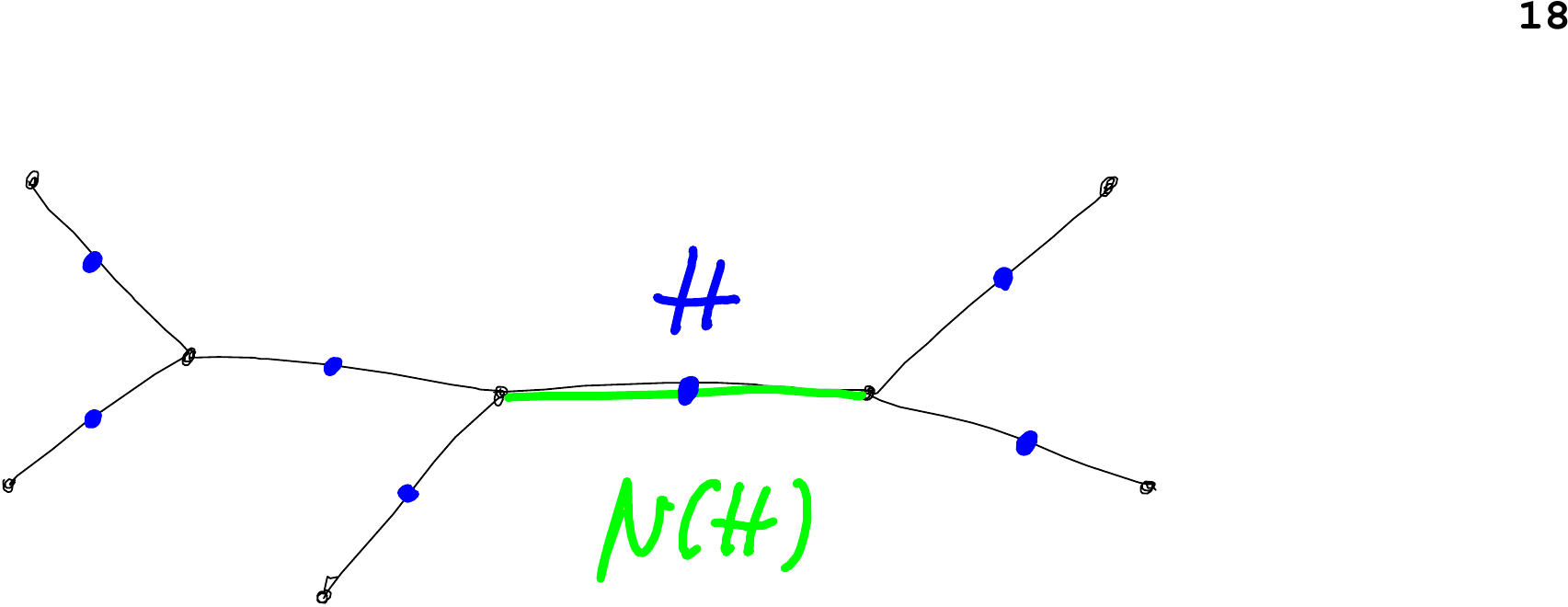}}
		\end{center}
	\end{figure}

 \item $\R^n$ with standard cubing
 \begin{figure}[h]
 	\begin{center}
 		\resizebox{!}{0.175\textwidth}{\includegraphics{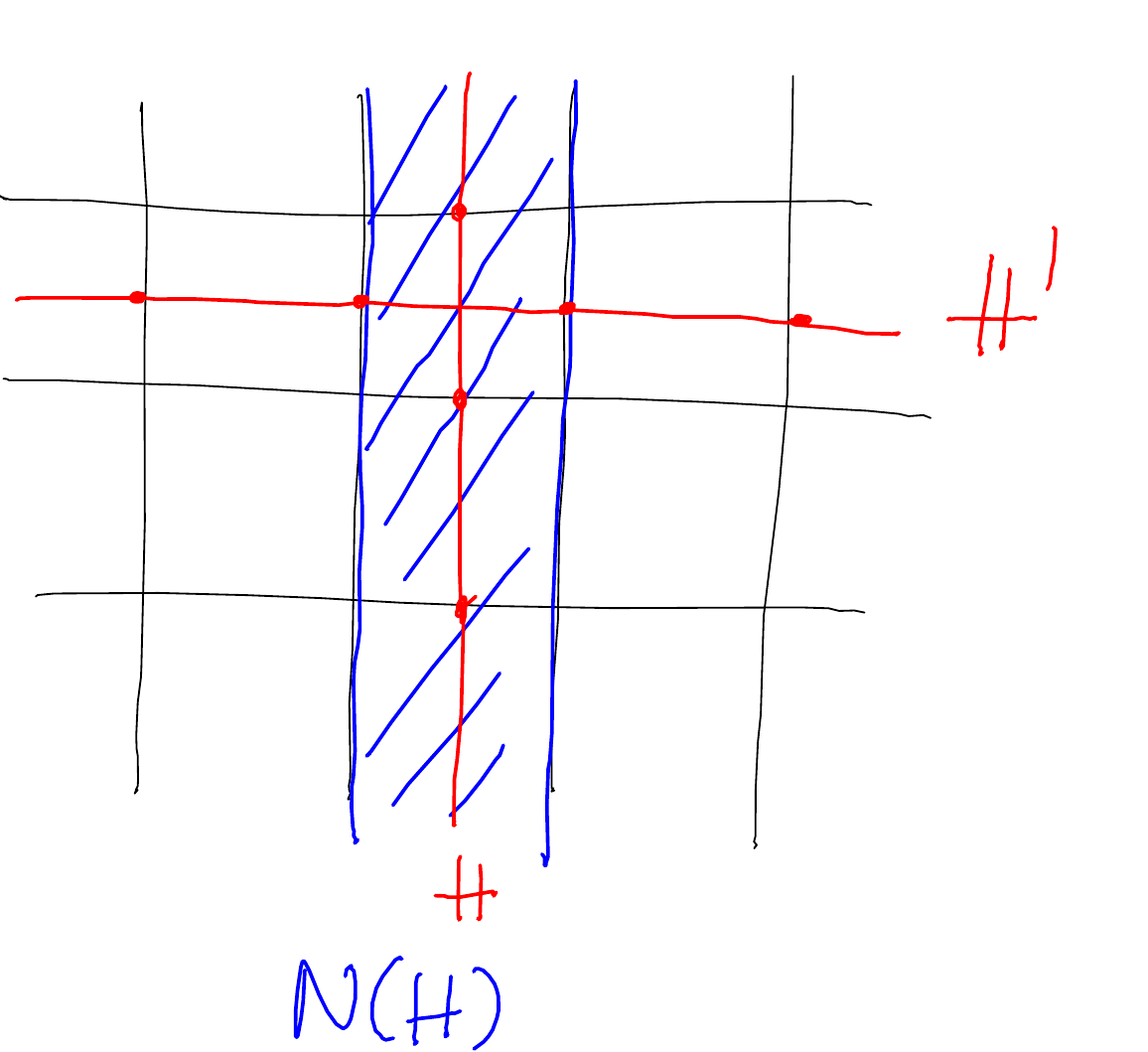}}
 	\end{center}
 \end{figure}

 \item $PSL_2(\Z)$ 
 
  \begin{figure}[h]
	\begin{center}
		\resizebox{!}{0.2\textwidth}{\includegraphics{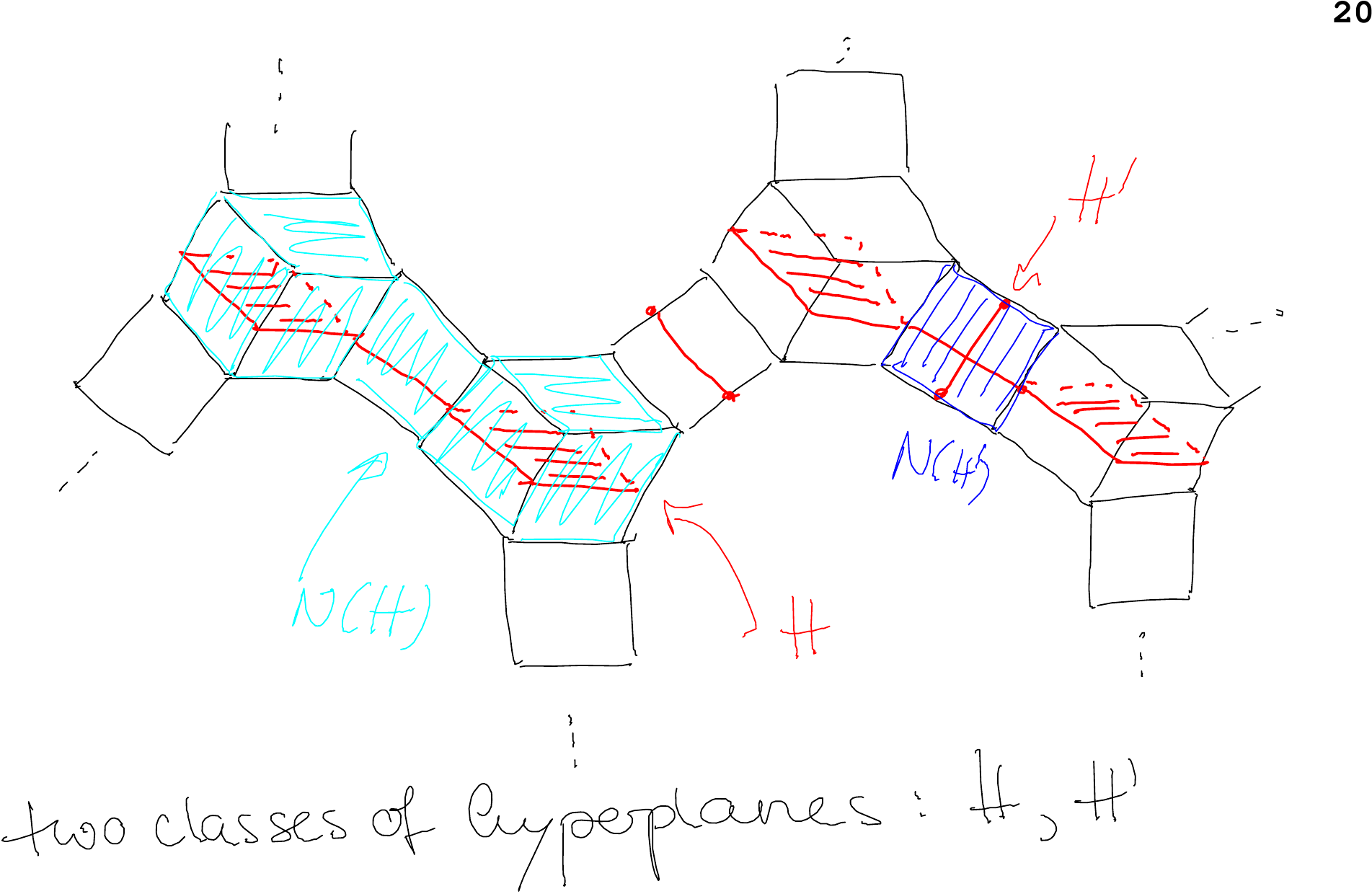}}
	\end{center}
 \end{figure}
 
\end{enumerate}

Non-$\cat(0)$ examples:
\begin{enumerate}
 \item self-intersecting hyperplanes
  \begin{figure}[h]
	\begin{center}
		\resizebox{!}{0.2\textwidth}{\includegraphics{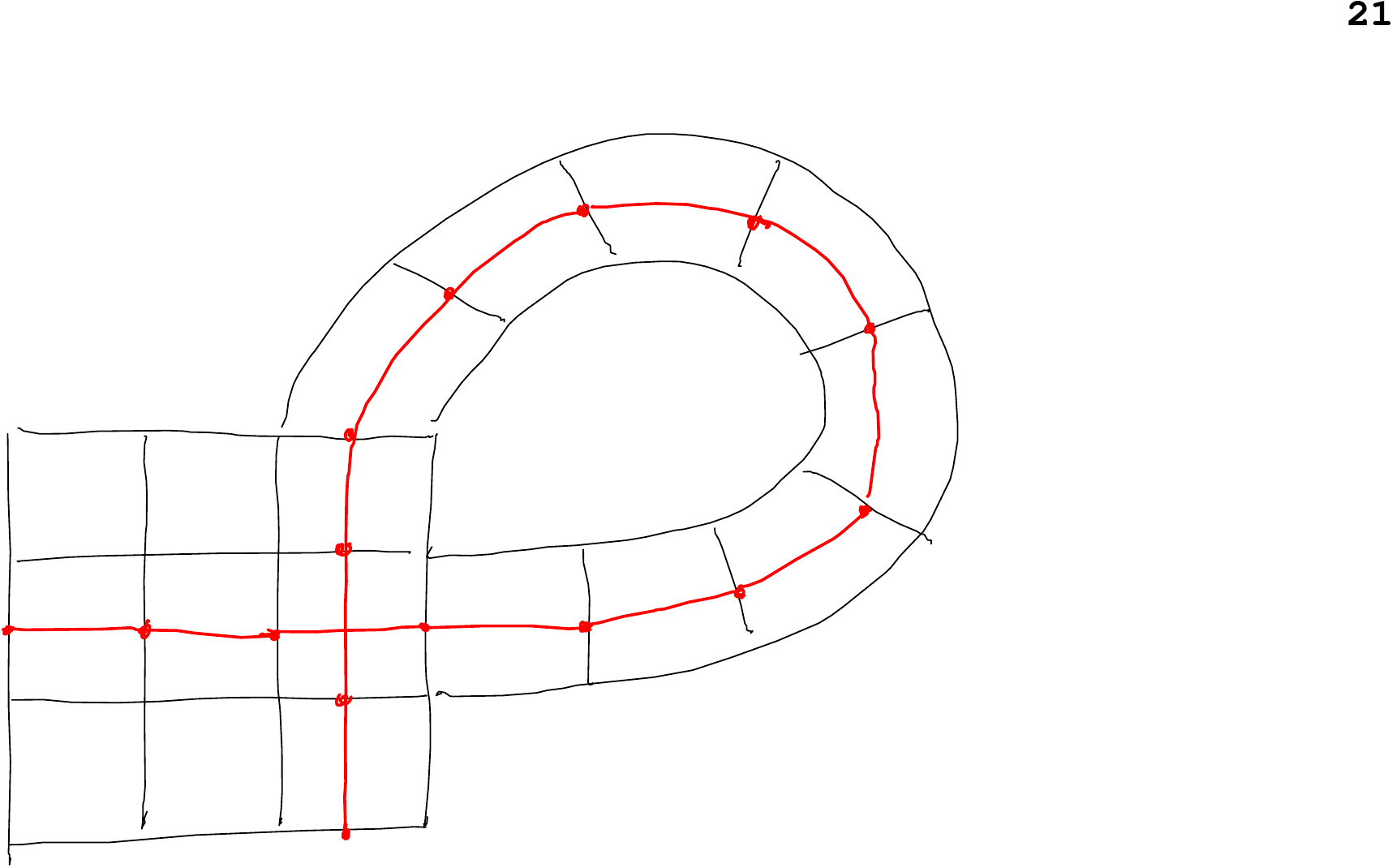}}
	\end{center}
\end{figure}

 \item self-parallel hyperplanes
  \begin{figure}[h]
	\begin{center}
		\resizebox{!}{0.2\textwidth}{\includegraphics{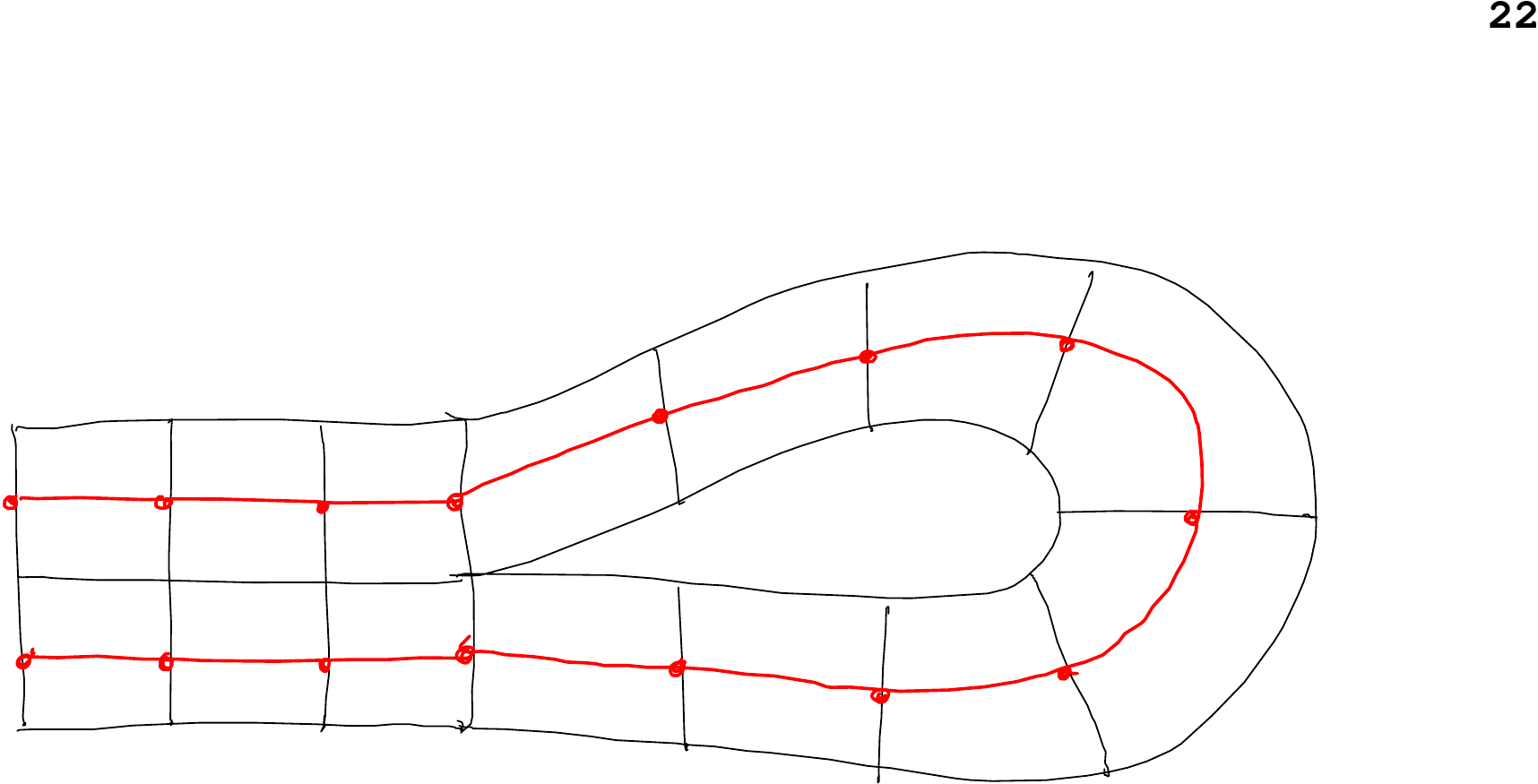}}
	\end{center}
\end{figure}

\end{enumerate}
\end{example}


\begin{lemma}\label{le:3.4}
 Every local isometry $g:X\to Y$ from a geodesic metric space $X$ into a $\cat(0)$-space $Y$ is an isometric embedding. 
\end{lemma}
\begin{proof}
Let $g:X\to Y$ be a local isometry  and $\gamma:[a,b]\to X$ a geodesic from $x$ to $y$.Then $g\circ \gamma$ is a local geodesic in $Y$. Since $Y$ is uniquely geodesic \ref{prop:1.10}.\ref{1.10.3} implies that $g\circ\gamma$ is a geodesic and
$$ d_Y(g(x),g(y))=l(g\circ\gamma)=l(\gamma)=d_X(x,y)$$
where the second to last equality holds since the lengths of a curve is invariant under local iometries.
\end{proof} 

\begin{prop}\label{prop:3.3}
For a finite dimensional  $\cat(0)$ cubical complex $X$ the following are true.
\begin{enumerate}
 \item\label{3.3.1} A hyperplane $H$ in $X$ is itself a $\cat(0)$ cubical complex and for each cube $C$ we have that  $H\cap C=\emptyset$ or that $H\cap C$ is a midcube of $C$. 
 \item\label{3.3.2} The support $N(H)$ is isometric to $H\times [0,1]$ and is a convex subcomplex of $X$.  
 \item\label{3.3.3} $X\setminus H$ has precisely two connected components. 
 \item\label{3.3.4} Every hyperplane is closed and convex in $X$. 
\end{enumerate}
\end{prop}
\begin{proof}
We start with the proof of \ref{3.3.4} which is due to Charney and which we've taken from \cite{Roller}. 
Let $H$ be a hyperplane defined by $[e]_\sim$ and let $X_0\define N(H)$ with $\tilde X_0$ being the universal cover of $X_0$ with the induced cubical structure. 
The class of $e$ lifts to an equivalence class $[\tilde e]_\sim$ in $X_0$. Observe that in $\tilde X_0$ each cube contains exactly one midcube transversal to $\tilde e$ since otherwise $\tilde X_0$ would not be simply connected. 
Hence there exists an isometry $f:\tilde X_0 \to \tilde X_0$ which is uniquely defined by the property that it maps a cube to its reflected image along the midcube transversal to $\tilde e$. 
The fixed point set $\fix(f)=\tilde H$, where $\tilde H$ is the hyperplane defined by $\tilde e$. 
Proposition~\ref{prop:1.14} now implies that $\tilde H$ is closed an convex in $\tilde X_0$. 
>From \ref{le:3.4} we obtain that $g:\tilde X_0\to X_0\hookrightarrow X$,  with $g$ being the concatenation of the covering map from $\tilde X_0$ to $X_0$ followed by the local isometric embedding of $X_0$ into $X$, is in fact an isometric embedding itself. Since $g(\tilde H)=H$ we may deduce that $H$ is closed and convex. 

Assertion~\ref{3.3.2} can easily be deduced from \ref{3.3.4} and its proof using that for a midcube $M$ in $C$ one has $C\cong M\times [0,1]$.  

To verify \ref{3.3.1} observe first that the intersection of $H$ with a cube is a unique midcube by \ref{3.3.4}. Hence to show that $H$ is a cubical complex it is enough to convince oneself that $[e]_\sim$ induces the following equivalence relation on midcubes:
two midcubes $M_i\subset C_i$, i=1,2, are \emph{equivalent}, denoted by $M_1\sim M_2$, if there exists a common face $F$ of $C_1$ and $C_2$ such that $M_1\cap F=M_2\cap F$. We write $[M]_\sim$ for the equivalence class of a midcube $M$ under the equivalence relation induced by $\sim$. The $\cat(0)$-property follows then from \ref{3.3.4} as convex subsets of $\cat(0)$-spaces are again $\cat(0)$. 

We now prove \ref{3.3.3}. Let $H$ be again a  hyperplane defined by $[e]_\sim$ and $m$ the midpoint of the edge $e$. An edge $e$ is closed and convex in $X$ hence there is by Proposition~\ref{prop:1.13} a  projection $\pi_e:X\to X$ such that $d(x, \pi_e(x))=\inf_{y\in e} d(x,y)$. Compare Figure~\ref{fig_24} and \ref{fig_25}. 

  \begin{figure}[h]
	\begin{center}
		\resizebox{!}{0.4\textwidth}{\includegraphics{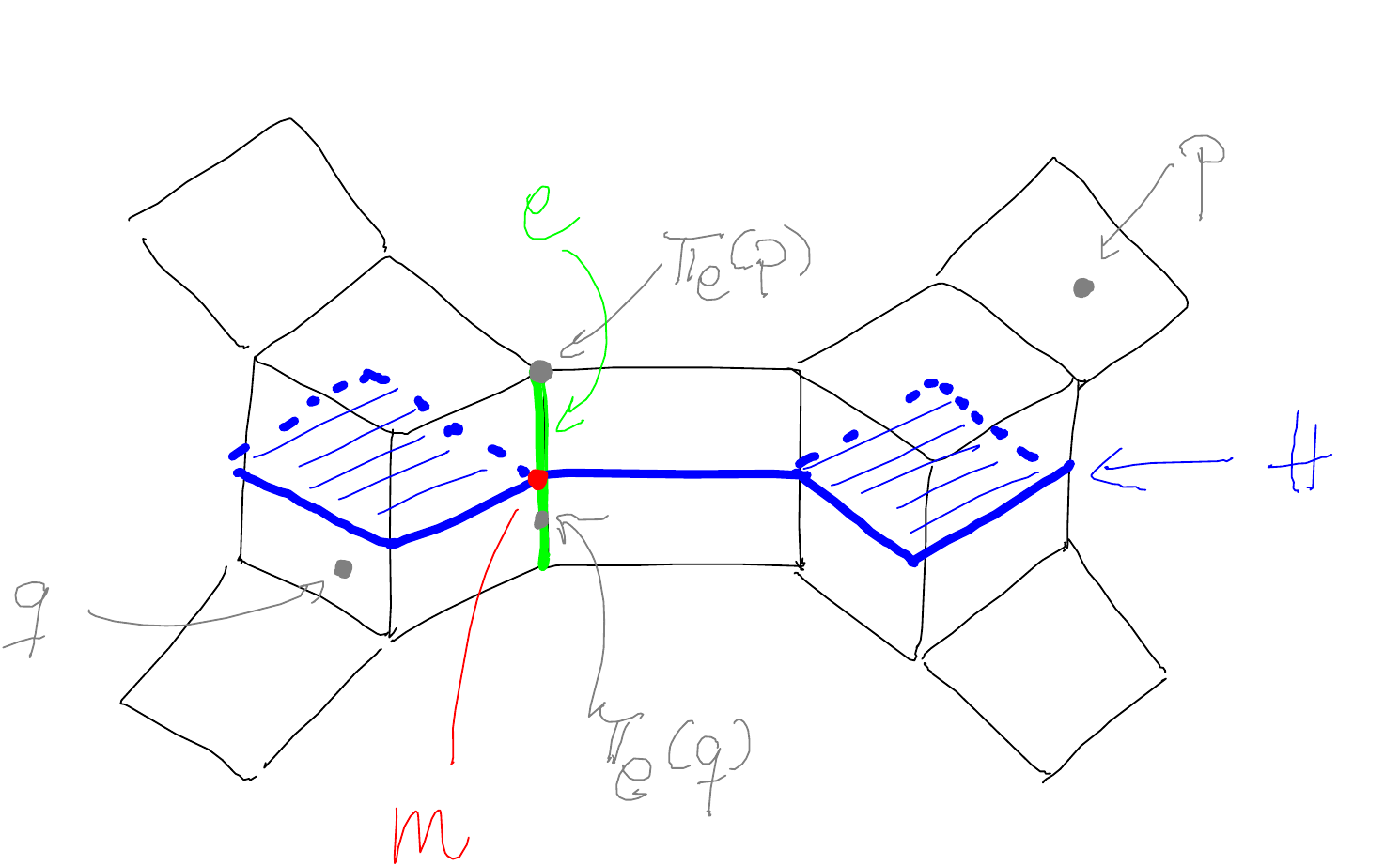}}
		\caption[links]{Illustration of the proof of \ref{prop:3.3}}
\end{center}
		\label{fig_24}
\end{figure}

We claim: $\pi_e^{-1}(\{m\})=H$ and $\pi_e(H)=\{m\}$.\newline
Choose $x\in H$, then $\pi_H(\pi_e(x))=m$ and 
$$d(x,m)=d(\pi_H(x), \pi_H(\pi_e(x)))\leq d(x,\pi_e(x)),$$
where $\pi_H$ is the projection onto $H$ defined in \ref{prop:1.13}. 
The point $\pi_e(x)$ is the unique point in $e$ having minimal distance to $x$. Hence the claim. 

Suppose that $\pi_e(x)=m$ for some $x\in X$. We prove by contradiction that $x\in H$. Assume the contrary. Then the geodesic $\gamma$ connecting $x$ and $m$ shares only $m$ with $H$. Moreover $\gamma$ and $e$ do not form a right angle at $m$ (otherwise $\gamma\cap H$ would contain more than one point). But then $\gamma$ may be shortened to a curve $\tilde\gamma$  perpendicular  to $e$ and connecting $x$ with $e$ which contradicts \ref{prop:1.13}.

 \begin{figure}[h]
	\begin{center}
		\resizebox{!}{0.4\textwidth}{\includegraphics{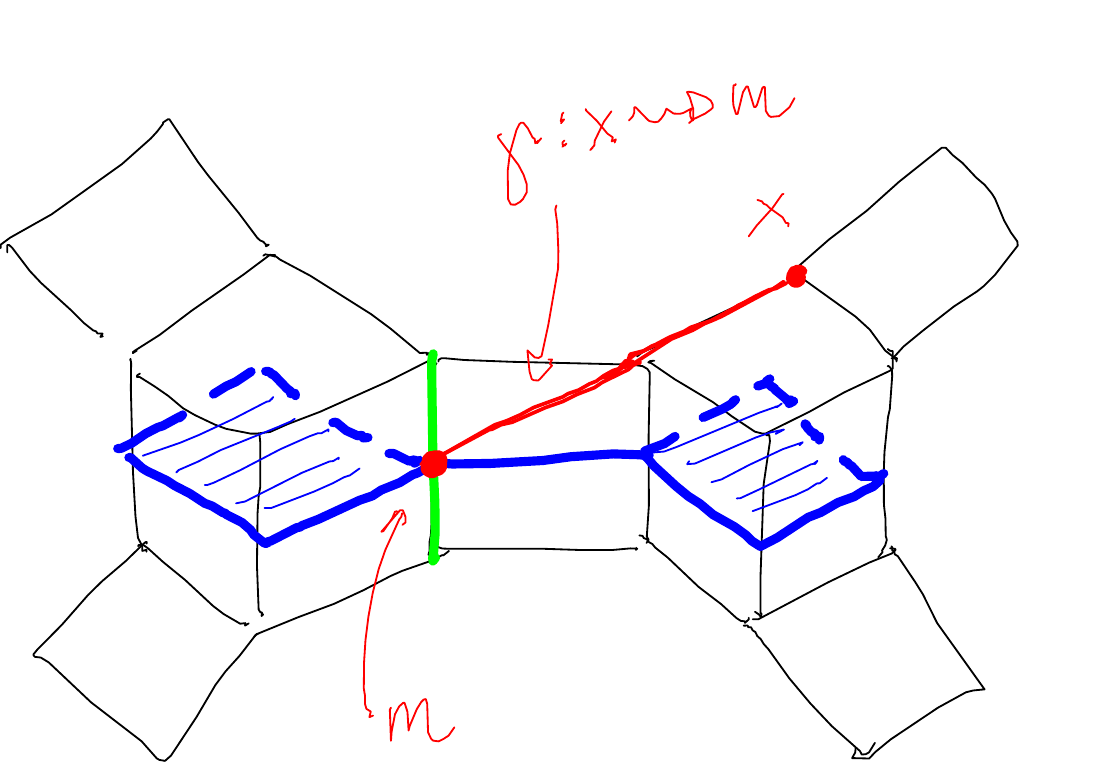}}
	\caption[links]{Illustration of the proof of \ref{prop:3.3}}.
	\end{center}
		\label{fig_25}
\end{figure}

We thus obtain that $\pi_e(X\setminus H) = e\setminus\{m\}$. Recall that $X$ is a geodesic space and hence each point in $X\setminus H$ can be connected by a geodesic to precisely one of the two pieces of $ e\setminus\{m\}$. 
\end{proof}

Sageev has shown  \ref{prop:3.3}.\ref{3.3.1} with a more direct argument. Studying links of vertices in $H$ it is easy to see that $H$ is locally $\cat(0)$ and simply connectedness can be obtained using disk-diagram techniques. Compare~\cite{Sageev}. 


\begin{definition}\label{def:3.4}
Given a hyperplane $H$ the connected components of $X\setminus H$ are called  \emph{(open) halfspaces} of $X$. We write $h$ and $h^c$ for the two halfspaces determined by $H$.
\end{definition}

Observe  that by definition $X=h\cup H\cup h^c$. 

\begin{prop}(Intersections of hyperplanes)\label{prop:3.5}
Let $X$ be a $\cat(0)$ cubical complex and $H_1, \ldots, H_m$ hyperplanes in $X$ such that $H_i\cap H_j\neq \emptyset$ for all $i\neq j$. Then 
$$ \bigcap_{i=1}^m H_i\neq \emptyset.$$
If $\dim X =n< \infty$ each family $\{H_i\}_{i\in I}$ with $H_i\cap H_j\neq \emptyset$ for all $i\neq j\in I$ contains at most $n$ elements. 
\end{prop}
\begin{proof}
The proof is by induction on $m$. 
We consider $m=3$ as a smallest case. Let $H_1, H_2$ and $H_3$ be hyperplanes. Suppose further that $H_1\cap H_2\cap H_3=\emptyset$. 
Let $\gamma$ be a geodesic connecting $H_3\cap H_1$ with $H_2\cap H_3$ inside the hyperplane $H_3$ and let $x$ and $y$ denote its endpoints. Compare Figure~\ref{fig_27}.  

\begin{figure}[h]
	\begin{center}
		\resizebox{!}{0.6\textwidth}{\includegraphics{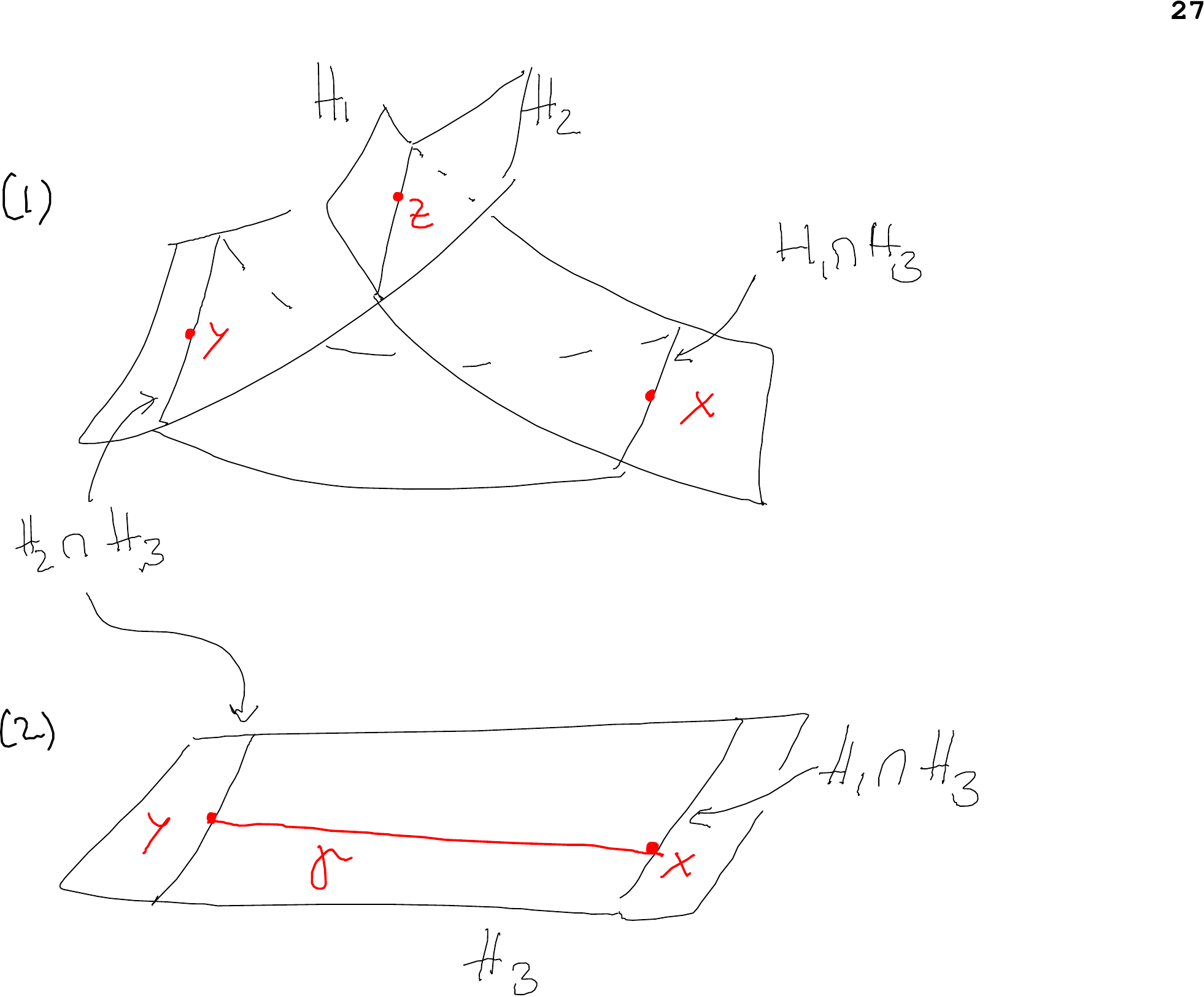}}
		\caption[links]{Illustration of the proof of \ref{prop:3.5}}.
	\label{fig_27}
	\end{center}
\end{figure}

Let now $z$ be a point in $H_1\cap H_2$ and consider geodesics $\overline{zx}\subset H_1$ and $\overline{zy}\subset H_2$. Then consider the triangle $\Delta$ spanned by these geodesics in the vertices $x,y,z$. 
The three hyperplanes intersect transversally therefore the angles in $\Delta$ at $x$ and $y$ are $\frac{\pi}{2}$. This is a contradiction as such a triangle does not exist in $\R^2$ but we have assumed the space to be CAT(0). 

For the induction step suppose $m>3$ and that the assertion is true for $m-1$. 
Let $H_1, H_2, \ldots, H_m$ be a collection of hyperplanes. Hyperplane $H_m$ is in particular a cubical complex itself and intersection $H_i\cap H_m$ are hyperplanes in $H_m$. By induction hypothesis the collection $\{H_i\cap H_m\}_{i=1, \ldots, m-1}$ has a non-empty intersection. Thus $\cap_{i=1}^m H_i\neq\emptyset$. 

It remains to prove that if the dimension of $X$ is $n$, then there exist only finitely many elements in such a collection. Let $\{H_i\}_{i\in I}$ be a collection of hyperplanes with pairwise nonempty intersection. Then $\cap_{i\in I}H_i\neq\emptyset$, i.e. there exist $\vert I\vert$-many midcubes which intersect nontrivially. Each $n$-cube has only $n$ such midcubes. Thuis $\vert I\vert \leq n=dim(X)$. 
\end{proof}

 \begin{figure}[h]
	\begin{center}
		\resizebox{!}{0.4\textwidth}{\includegraphics{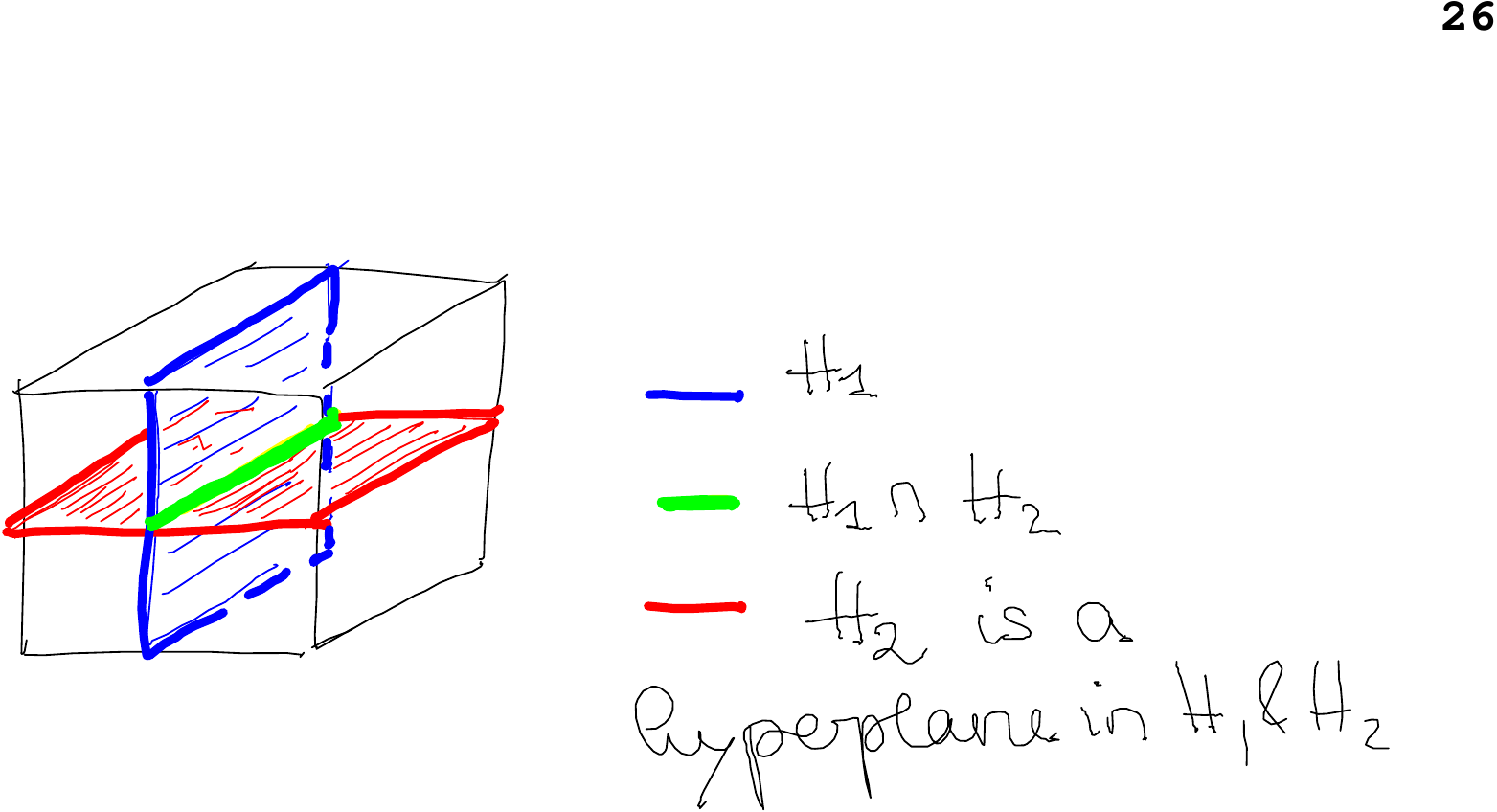}}
		\caption[links]{Illustration of the proof of \ref{prop:3.5}}.
	\end{center}
	\label{fig_26}
\end{figure}

%
%

\newpage

\section{Halfspace systems vs.\  cube complexes}

The main goal of this section is to define halfspace systems and construct from these examples of $\cat(0)$ cube complexes. 

\begin{definition}\label{def:4.1}
Consider a tripel $(H, \less, \star)$ with $H$ a set, $\less$ a partial order on $H$ and $\star$ a order reversing involution on $H$, that is a map $\star:H\to H:h\mapsto h^\star$ such that $h^\star\less g^\star$ if $g\less h$. Assume further that the following two conditions are satisfied:
\begin{enumerate}
 \item\label{4.1.1}(finite interval condition)
		$\forall h_1, h_2\in H$ there exist only finitely many $k\in H$ with $h_1\less k\less h_2$. 
 \item\label{4.1.2}(nesting condition)
		for $h,k\in H$ at most one of the following is true:
		$h\less k$, $h\less k^\star$, $h^\star\less k$ and $h^\star\less k^\star$.
\end{enumerate}
Then $(H, \less, \star)$ is a \emph{halfspace system}, elements $h\in H$ are called \emph{half-spaces}. If none of the conditions in \ref{4.1.2} is satified by a fixed pair $h,k$ we say that $h$ and $k$ are \emph{transversal}.

Defining an equivalence relation $h\sim h^\star$ we obtain a set $\bar H\define H\diagup_\sim$ of \emph{hyperplanes} $\bar h\in \bar H$, where we may identify an equivalence class $\bar h$ with the defining set $\{h, h^\star \}$. The \emph{boundary map} $\delta:H\to\bar H$ assigns to a halfspace $h$ its hyperplane $\bar h =\{h, h^\star\}$ 
\end{definition}

The following example illustrates why we may think of $\bar h$ as separating the halfspaces $h$ and $h^\star$ in some kind of ambient space. 

%
%
%

\begin{example}\label{ex:4.2}
The halfspaces  $h$ and $k$ are transversal, and $h, h'$ are nested. 
\end{example}

\begin{figure}[h]
	\begin{center}
		\resizebox{!}{0.5\textwidth}{\includegraphics{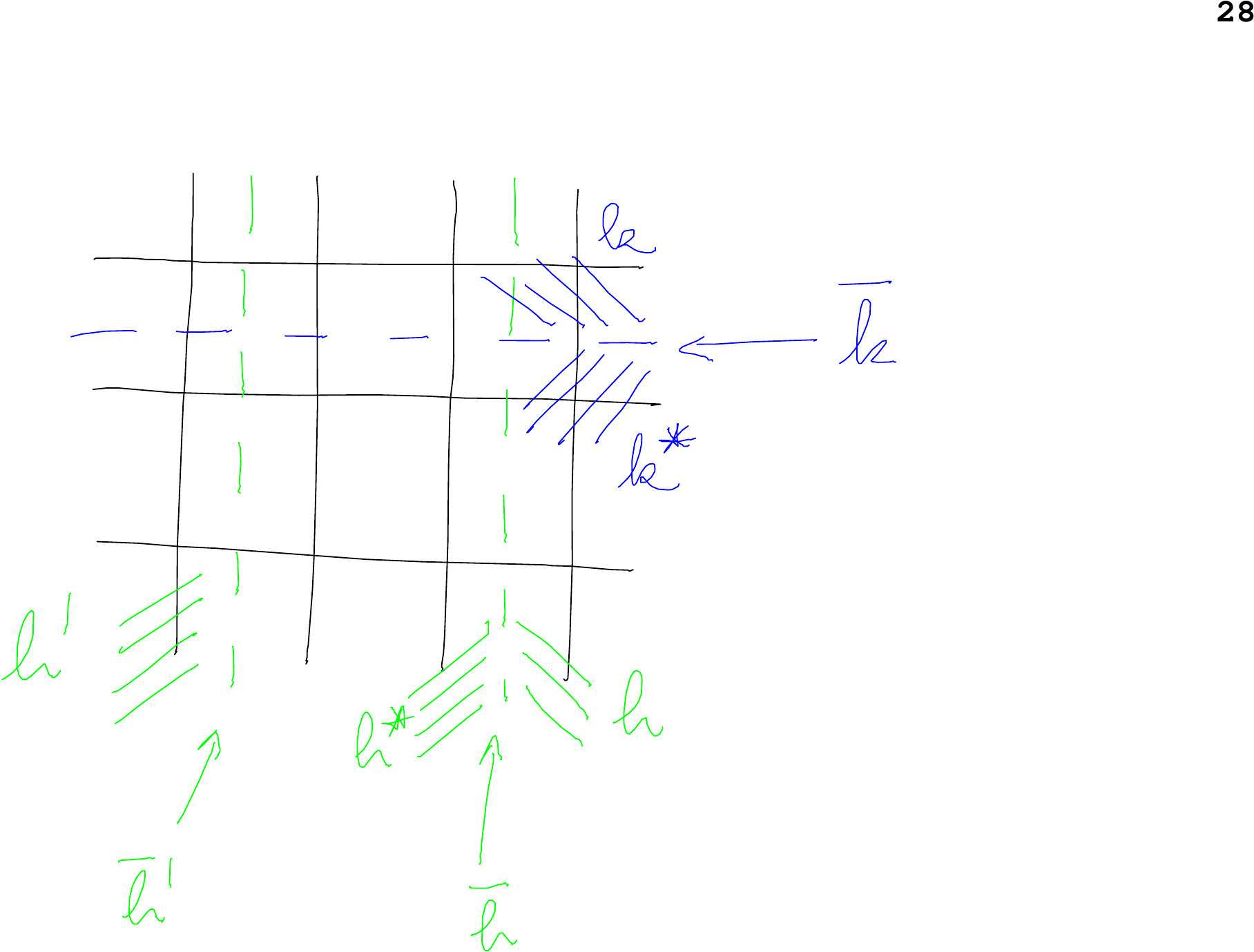}}
		\caption[links]{An example of nested halfspaces is the pair $h,h'$ and $h,k$ are transversal ones. }.
	\end{center}
	\label{fig_28}
\end{figure}

Or next main goal is to prove the following theorem.
The construction of a cubical complex from a half-space system is due to Sageev \cite{Sageev}. But other authors have also contributed to the current formulation. See also Chatterji-Niblo \cite{CN} or Nica \cite{Nica}. 
 
\begin{thm}\label{thm:4.3}
A halfspace system defines in a natural way a cubical complex $X(H)$. Every connected component of $X(H)$ is $\cat(0)$ and maximal cubes are in a one-to-one correspondence with maximal  families of pairwise transversal hyperplanes. The dimension of a maximal cube equals the number of hyperplanes in such a family. 
\end{thm}

Before we are able to give a proof we need to see how one constructs a cube complex from a halfspace system. Even though vertices will be abstract maps in the definition of $X(H)$ we may first use a more heuristic approach and have a look at a concrete cubulation in order to observe which properties should be satisfied by  the vertices of the desired complex:

{\bf Main idea concerning vertices:}

When looking at the standard cubulation of the plane for example and consider halfspaces defined by vertical or horizontal lines (i.e. cubulated copies of $\R$) in this complex a vertex $v$ is either contained in a halfspace $h$ or in its complement $h^*$.  It may never happen that $v\in h\cap h^*$. Moreover a vertex is uniquely determined by the set of halfspaces containing it. This should be guaranteed by the construction as well. 
Compare Figure~\ref{fig_29} part 1). 

\begin{figure}[h]
	\begin{center}
		\resizebox{!}{0.4\textwidth}{\includegraphics{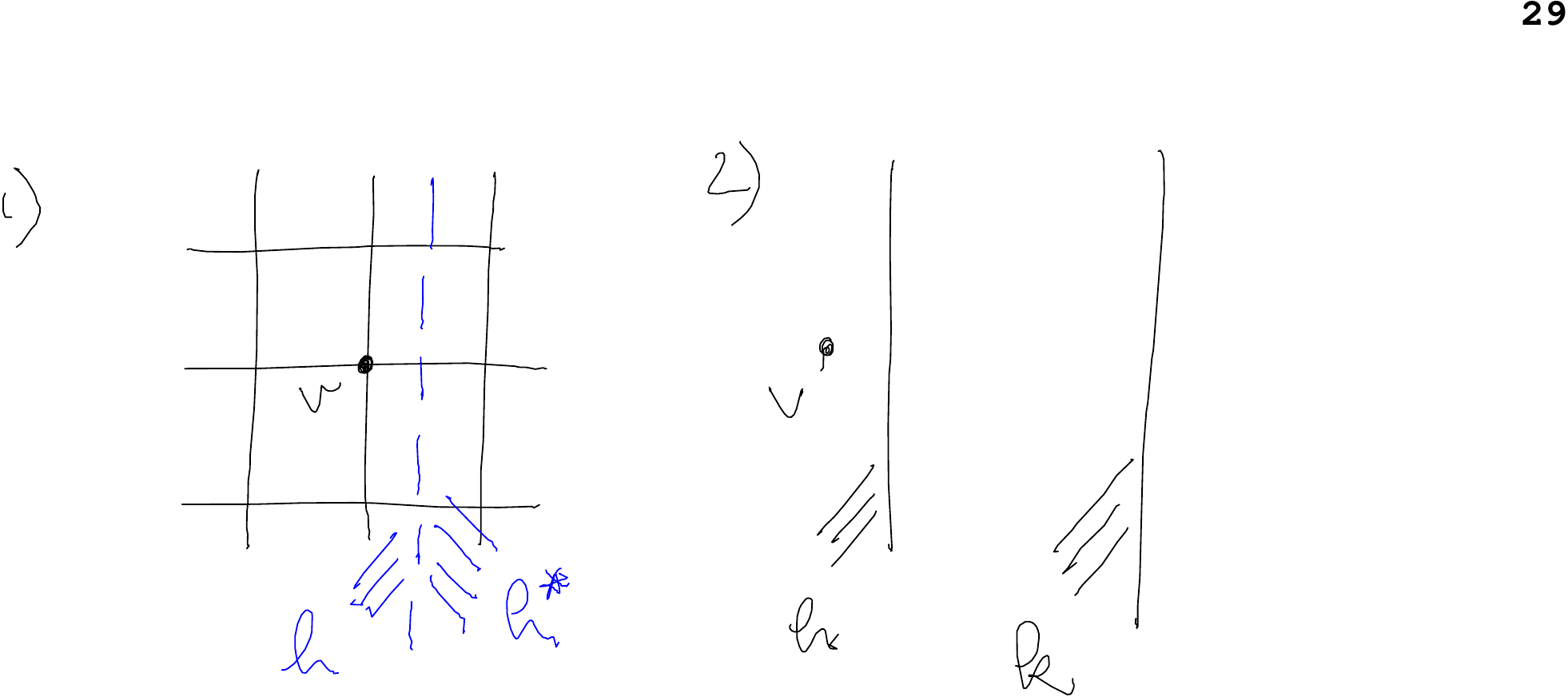}}
		\caption[links]{How half-spaces relate to vertices.}.
	\end{center}
	\label{fig_29}
\end{figure}

We also would like to have the following property satisfied by the constructed cube complex: Whenever $v\in h$ and we have $h\subset k$ then $v$ should be also contained in $k$. Compare Figure~\ref{fig_29} part 2).

\begin{definition}\label{def:4.4}
For a halfspace system $(H, \less, \star)$ the \emph{cube complex $X(H)$} associated to this triple is defined as follows:
\begin{enumerate}
 \item The set of \emph{vertices} $V(H)$ is defined by
		\begin{equation}\label{def:vertices}
		V(H)=\{v:\bar H \to H \,\vert\, v(\bar h)\neq (v(\bar k))^\star \text{ for all }\bar k, \bar h\in \bar H \}.
		\end{equation}
		We will write $v\in h$ if $v(\bar h)=h$ and say that \emph{$v$ is contained in $h$}.
 \item The set of \emph{edges} $E\subset V\times V$, i.e. the one-skeleton of $X$,  consists of all pairs of 				vertices $(v,w)$ such that 
 \[\vert \{\bar h \in \bar H \,\vert\, v(\bar h)\neq w(\bar h) \} \vert =1.\]
\end{enumerate}
\end{definition}

Condition \ref{def:vertices} basically means that either $v(\bar h)\leq v(\bar k)$ or the two values are not comparable. Also the map $v:\bar H\to H$ can be thought of as assigning to each hyperplane $\bar h$ the halfspace in $\{h, h^\star\}$  which contains $v$.

\begin{figure}[h]
	\begin{center}
		\resizebox{!}{0.4\textwidth}{\includegraphics{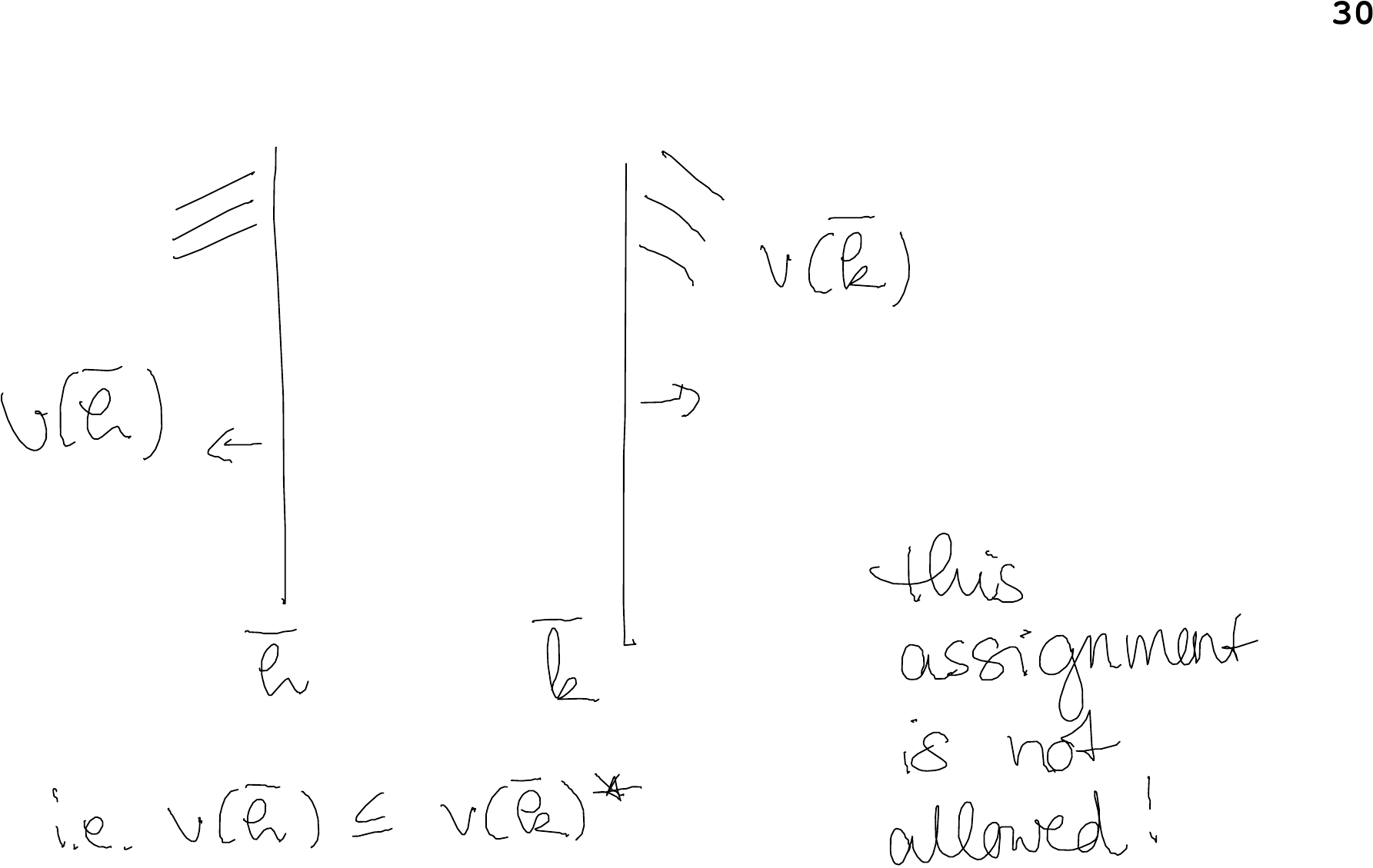}}
		\caption[links]{Illustration of a condition in Definition~\ref{def:4.4}.}.
	\end{center}
	\label{fig_30}
\end{figure}

We will prolong Definition~\ref{def:4.4} in a minute to define the higher dimensional cubes. But first let us  introduce additional notions and properties of the vertices and halfspaces. 

\begin{definition}\label{def:4.5}
 A halfspace $h$ is \emph{minimal} with respect to $v\in V(H)$ if there does not exist $k\in H$ with $v\in k$ and $k\less\h$.
\end{definition}

For examples of (non-)minimal halfspaces see Figure~\ref{fig_31}. 
  
\begin{figure}[h]
	\begin{center}
		\resizebox{!}{0.4\textwidth}{\includegraphics{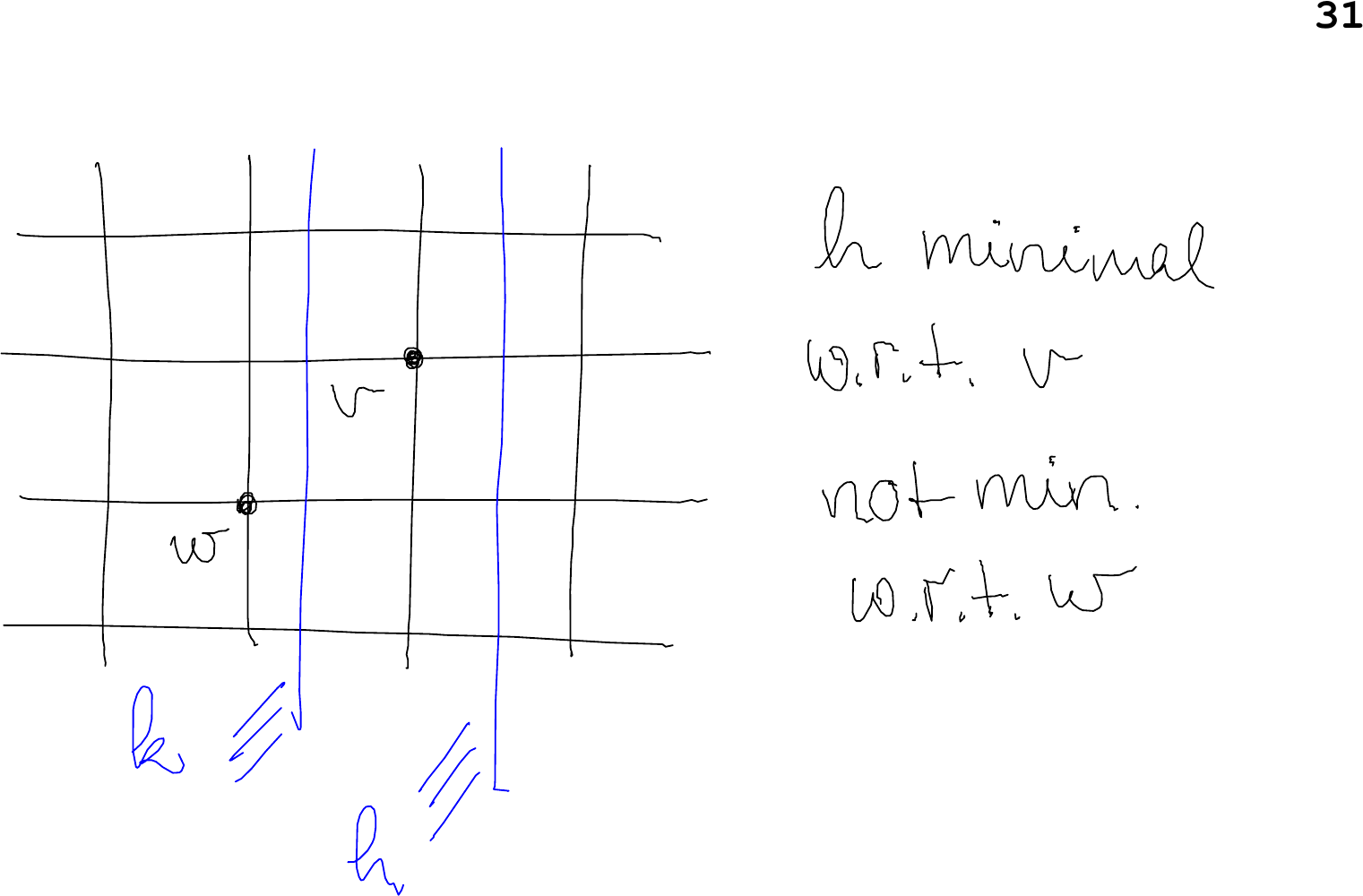}}
		\caption[links]{Minimal and non-minimal halfspaces}.
	\end{center}
	\label{fig_31}
\end{figure}

\begin{definition}\label{def:4.6}
Let $v$ be a vertex, $\bar h\in \bar H$ and define $v_{\bar h}: \bar H\to H$ by
\begin{equation}
v_{\bar h}(\bar k) = \left\{
\begin{array}{c l}     
   v(\bar k)^\star &  \text{ if } \bar k=\bar h\\
   v(\bar k) 	& \text{ else. }
\end{array}\right. 
\end{equation}
\end{definition}

For the map $v_{\bar h}$ the halfspace  $h$ is replaced by by the opposite halfspace $h^\star$ in comparison with the vertex $v$. Thus if $v_{\bar h}$ is a vertex, then $v$ and $v_{\bar h}$ are adjacent vertices in $X(H)$.

\begin{lemma}\label{le:4.7}
Let $v\in V(H), \bar h\in \bar H$ and let  $v_{\bar h}$ be as in \ref{def:4.6}. Then 
\begin{enumerate}
 \item\label{4.7.1} $v_{\bar h}$ is a vertex iff $h$ is minimal with respect to $v$ and
 \item\label{4.7.2} the neighbours of a vertex $v$ are in one-to one correspondence with the halfspaces $h$ that are minimal with respect to $v$. 
\end{enumerate}
\end{lemma}
\begin{proof}
We first prove \ref{4.7.1}. Let $v_{\bar h}$ be a vertex. Then 
$$v_{\bar h}(\bar k)\not\less (v_{\bar h}(\bar h))^\star \text{ for all }\bar k\bar h\in \bar H.$$
Suppose that $h$ is not minimal with respect to $v$: Then $\exists k\neq h$ with $v\in k$ and $k\less h$. By Definition~\ref{def:4.6} then $v_{\bar h}(\bar k) =k$ and 
$$v_{\bar h}(\bar k)=k\leq h = (v_{\bar h}(\bar h))^\star. $$
Which contradicts the fact that $v_{\bar h}$ is a vertex. 

To prove the converse let $h$ be minimal with respect to $v$ and suppose that $v_{\bar h}$ is not a vertex. 
Then one of the two following alternatives is satisfied:
\begin{itemize}
 \item[(i)] $v_{\bar h}(\bar h)\less (v_{\bar h}(\bar k))^\star \text{ for some }\bar k\in \bar H$
 \item[(ii)] $v_{\bar h}(\bar k)\less (v_{\bar h}(\bar h))^\star \text{ for some }\bar k\in \bar H$.
\end{itemize}
Note that it can't happen that $h$ is not involved since $v$ was a vertex which only has been altered in the image of $h$. 

Suppose (i). Then
$$h^\star = (v(\bar h))^\star = v_{\bar h}(\bar h)\less (v_{\bar h}(\bar k))^\star = (v(\bar k))^\star.$$
Using the fact that $^\star$ is orientation reversing we may conclude that $v\in v(\bar k)\less h$, which contradict minimality of $h$. 
\end{proof}

{\bf Heuristic for definition of $n$-cubes:}
Let $h,k$ be transversal hyperplanes which are minimal with respect to a vertex $v$. Changing $v$ on both $h$ and $k$ at the same time we should get again a vertex in the cube complex. 
See Figure~\ref{fig_32} for an illustration of this idea. 
In other words a pair of transversal hyperplanes should correspond to a 2-cube in the complex. 

\begin{figure}[h]
	\begin{center}
		\resizebox{!}{0.4\textwidth}{\includegraphics{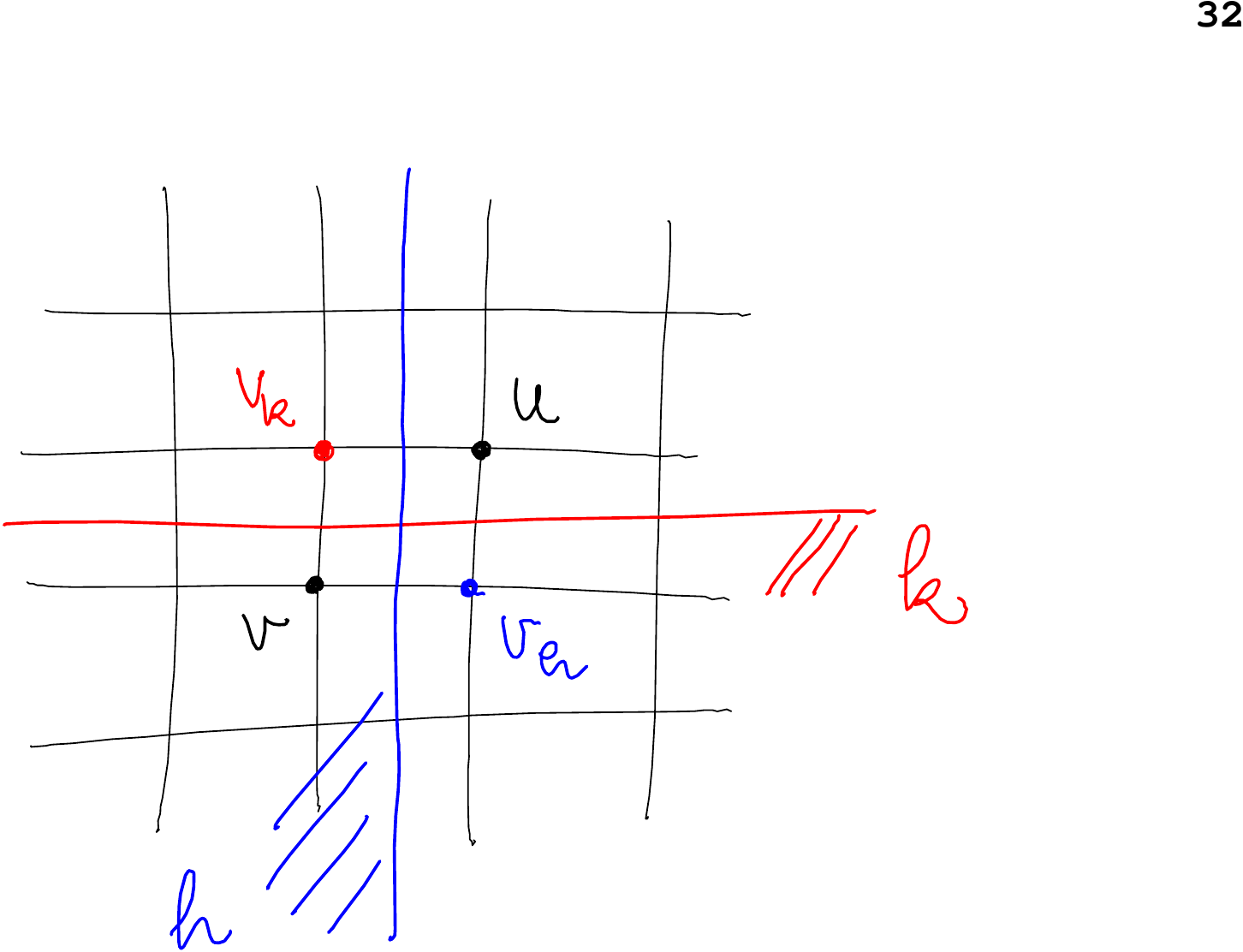}}
		\caption[links]{A 2-cube corresponds to a pair of transversal hyperplanes.}.
	\end{center}
	\label{fig_32}
\end{figure}

\begin{definition} (Continuation of \ref{def:4.4}) \label{def:4.4continued}
\begin{itemize}
 \item[(3)]\label{4.4.3} The $n$-skeleton $X^n(H)$ of $X(H)$  is defined inductively as follows: We glue an $n$-cube to a  subcomplex $Y$ of $X^n(H)$ if $Y\cong (C^n)^{n-1}$ where $(C^n)^{n-1}$ is the n-1-skeleton of the n-cube $(C^n)$.
\end{itemize}
\end{definition}

So whenever there is an empty k-cube in $X(H)$ the interior of this cube gets added. Keep adding cubes as long as there are skeleta of higher-dimensional cube with no interior. 

\begin{lemma}\label{le:4.8}
Let $C$ be an n-cube in $X(H)$ and  $v$ a vertex of $C$. The neighbours of $v$ in $C$ are of the form $v_{\bar h_i}$ for hyperplanes $\bar h_i, i=1\ldots, n$ where $v_{\bar h_i}$ is as defined in \ref{def:4.6}. 
Let $u$ be the vertex in $C$ diagonally opposite $v$, then $u$ may be obtained from $v$ by simultaneously switching out $h_i$ by $h_i^\star$ in the definition of $v$. That is 
\begin{equation}
v_{\bar h}(\bar k) = \left\{
\begin{array}{c l}     
   v(\bar k)^\star &  \text{ if } \bar k=\bar h_i \text{ for some } i=1, \ldots, n\\
   v(\bar k) 	& \text{ else. }
\end{array}\right. 
\end{equation}
 \end{lemma}

\begin{figure}[h]
	\begin{center}
		\resizebox{!}{0.4\textwidth}{\includegraphics{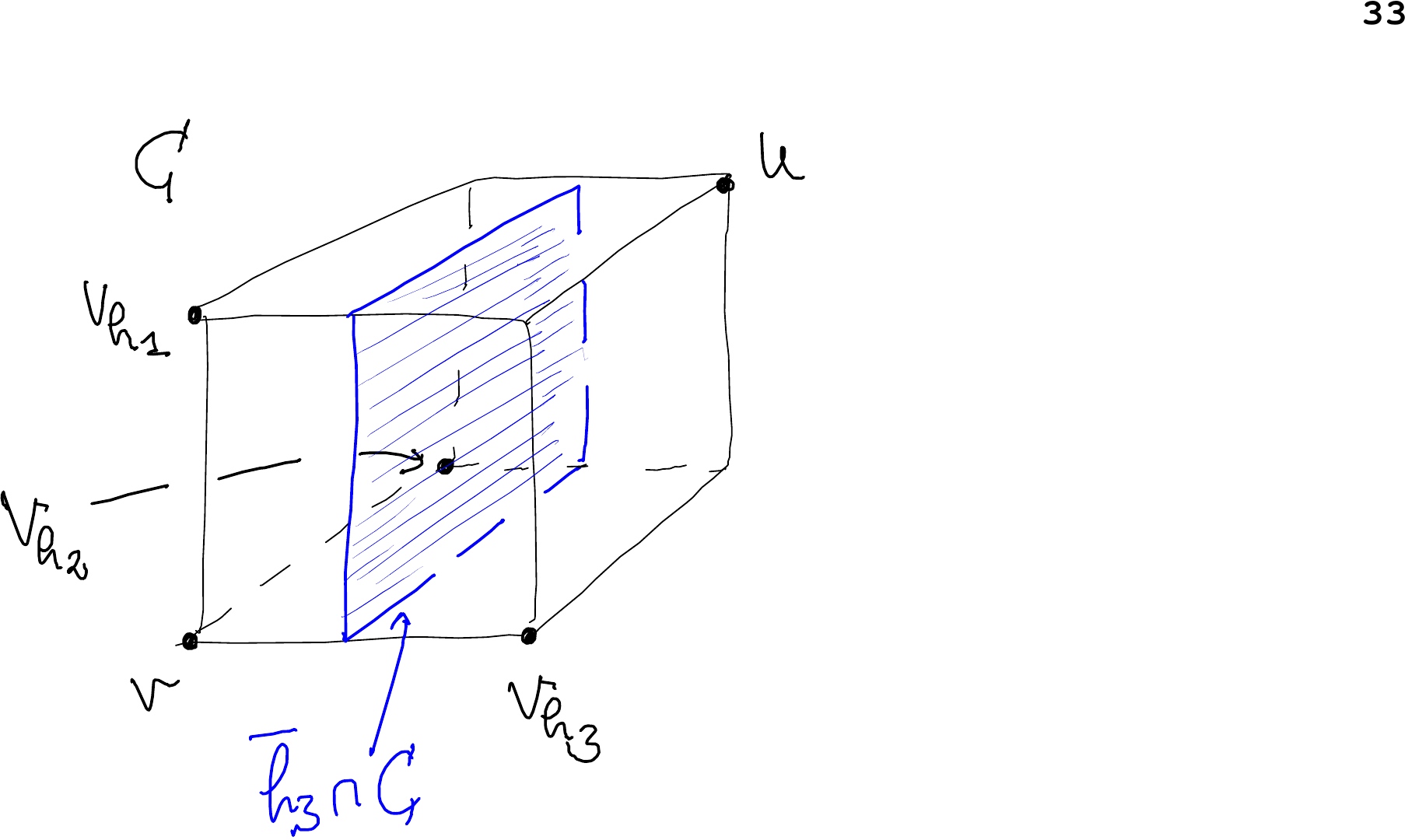}}
		\caption[links]{Illustration of Lemma~\ref{le:4.8}}.
	\end{center}
	\label{fig_33}
\end{figure}

\begin{proof}
The proof is by induction on $n$. First note that for $n=1$ there is nothing to prove. 
For $n=2$ there are four vertices in total. The vertex $v$, its neighbours $v_1, v_2$ and the diagonally opposite one $u$. We know that if $v$ label the half-spaces in $H$ in such a way that $v$ corresponds to the map $\bar h\mapsto h$ for all $\bar h\in \bar H$. By slight abuse of notation ($\bar H$ might not be countable) we may think of this map as listing all the halfspaces ``chosen by '' $v$, that is for a suitable naming of the elements of $\bar H$ and $H$ we have that   
\begin{equation*}
v \hat = (h_1, h_2, h_3,\ldots ),\, 
v_1\hat = (h_1^\star, h_2, h_3,\ldots ) \text{ and } 
v_2\hat = (h_1, h_2^\star, h_3,\ldots ).
\end{equation*}
The map that repesents vertex four needs to differ from $v_1$ and $v_2$ in precisely one position each. Hence there is no other choice than picking 
$$u \hat = (h_1^\star, h_2^\star, h_3,\ldots ).$$

\begin{figure}[h]
	\begin{center}
		\resizebox{!}{0.4\textwidth}{\includegraphics{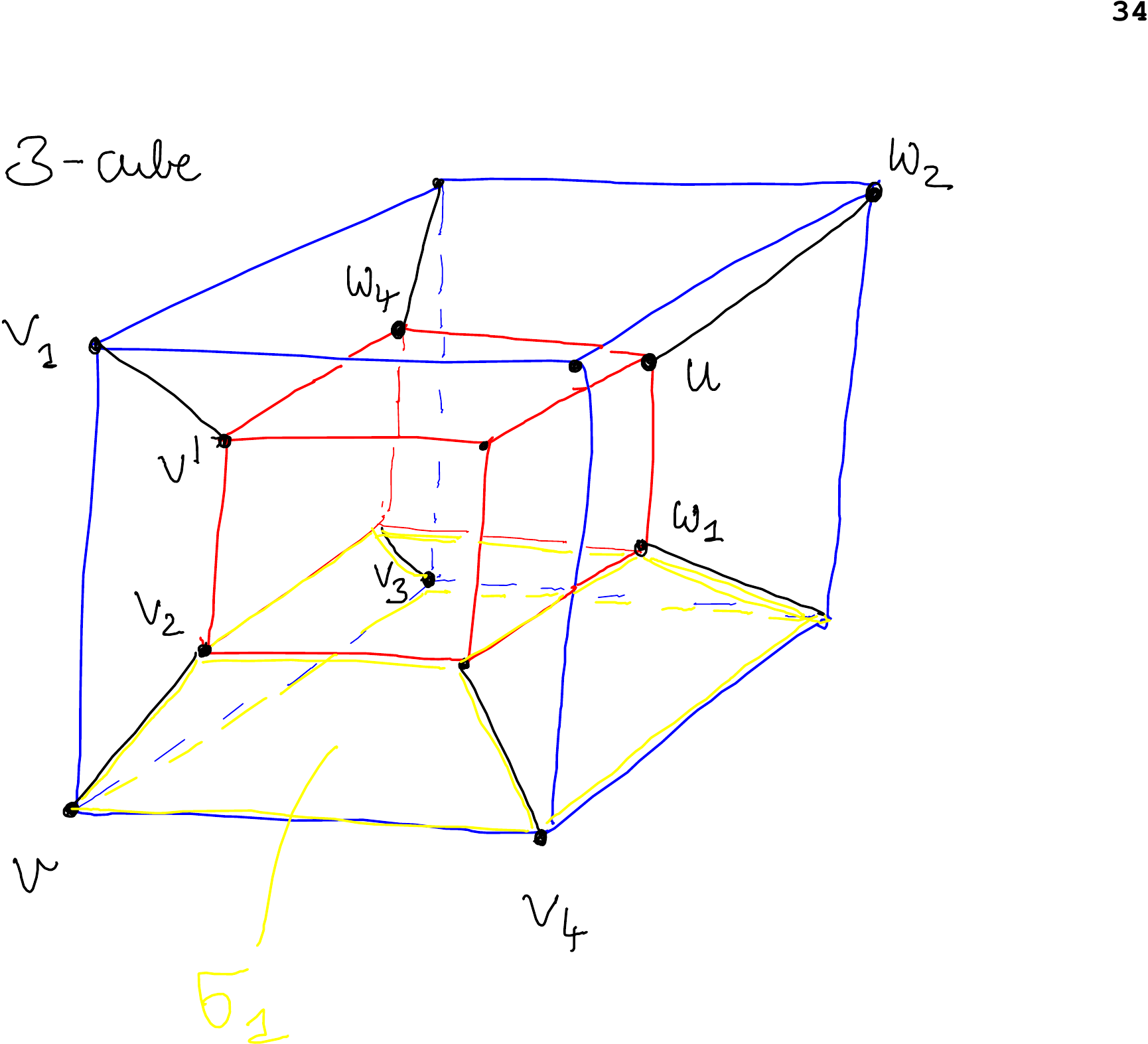}}
		\caption[links]{Illustration of the induction step in the proof of Lemma~\ref{le:4.8}.}.
	\end{center}
	\label{fig_34}
\end{figure}

Suppose now that we have shown the assertion for $n-1$ and $n-2$ and let $v$ be a vertex in an $n$-cube $C$ having neighbours $v_i=v_{\bar h_i}$ for $i=1,\ldots, n$. The vertex diagonally opposite of $v$ in $C$ is again denoted by $u$. We will write $\sigma_i$ for the codimension one face of $C$ containing all vertices $v_j, j=1,\ldots, n$ except for vertex $v_i$ and denote by $w_i$ the vertex diagonally opposite $v$ in $\sigma_i$. By induction hypothesis the vertex $w_i$ was obtained from $v$ by simultanously switching out the halfspaces $h_j$ for all $j\neq i$. That is 
\begin{equation*}
 w_i(\bar k)=\left\{
\begin{array}{c l}     
   v(\bar k)^\star &  \text{ if } \bar k=\bar h_j \text{ for some } j\neq i \\
   v(\bar k) 	& \text{ else. }
\end{array}\right. 
\end{equation*}
Consider the vertex $v'$ obtained from $v$ by simultaneously switching out the hyperplanes $h_1, \dots, h_{n-2}$. This vertex is diagonally opposite $v$ in the (n-2)-subcube of $C$ containing the vertices $v_1, \dots, v_{n-2}$. 
Now $v', w_n$ and $w_{n-1}$ are three vertices of a 2-dimensional subcube of $C$. 
Since $u$ is a common neighbour of $w_n$ and $w_{n-1}$ it has to be the fourth vertex in this cube and is obtained from $v'$ by simultaneously switching out the halfspaces $h_{n-1}$ and $h_n$. 
\end{proof}

\begin{lemma}\label{le:4.9}
Let $v$ be a vertex in $X(H)$ and $S\subset \bar H$ with $\vert S\vert =n$.
Then the vertices $w_T$, where $T$ runs through all subsets of $S$, and wich are defined by 
\begin{equation*}
 w_T(\bar k)=\left\{
\begin{array}{c l}     
   v(\bar k)^\star &  \text{ if } \bar k \in T \\
   v(\bar k) 	& \text{ if } \bar k \notin T. 
\end{array}\right. 
\end{equation*}
span an n-cube iff the following two properties hold:
\begin{enumerate}
 \item\label{4.9.1} $h\define v(\bar h) $ is minimal with respect to $v$ for all $\bar h\in S$ and
 \item\label{4.9.2} $\forall \bar h, \bar k \in S$ the image $h\define v(\bar h)$ is transversal to $k\define v(\bar k)$. 
\end{enumerate}
\end{lemma}
\begin{proof}
Suppose first that the given vertices span an n-cube. 
Minimality follows from Lemma~\ref{le:4.7}.\ref{4.7.1} as $v(\bar h)$ is minimal with respect to $v$ for all $\bar h\in S$ (here $\vert T\vert =1$).  In $C$ now $v$ is adjacent to all vertices $w_T$ where $\vert T\vert =1$. 

We will now prove transversality, that is \ref{4.9.2}. Choose $T=\{\bar h, \bar k\}\subset S$ and write $w_{\bar k}$, resp. $w_{\bar h}$,  instead of $W_{T'}$ for the vertices defined by the subssets $T'$ of $T$ containing a single hyperplane.

Lemma~\ref{le:4.8} implies that $v, w_{\bar h}, w_{\bar k}, w_T$ span a 2-cube. Further 
$h$ is minimal w.r.t.\ $v$ and $ w_{\bar k}$ and similarly  $k$ is minimal w.r.t.\ $v$ and $ w_{\bar h}$.
From the definition of minimality \ref{def:4.5}  we obtain that 
\begin{equation}\label{tec1}
\not\exists\, h'\in H \text{ such that } v\in h'\text{ or  } w_{\bar k}\in h' \text{ and } h'\less h.
\end{equation}
As well as the same statement with the roles of $h$ and $k$ switched.
Putting $h'=k$ in \ref{tec1} we get 
$$v(\bar h)=h\not\less k  \Longleftrightarrow h^\star \not\less k^\star$$
and 
$$w_{\bar k}(\bar k)=v(\bar k)^\star =k^\star \not\less h \Longleftrightarrow h^\star \not\less k.$$ 
and again similar statements with roles of $h$ and $k$ switched which rule out the remaining cases of \ref{def:4.1}. Therefore $h$ and $k$ are transversal. 

To prove the converse suppose now that \ref{4.9.1} and \ref{4.9.2} are satisfied.
We will first prove by induction on $n\define \vert T\vert$ that $w_T$ is a vertex for all $T\subset S$. 

If $n=1$ the claim follows from Lemma~\ref{le:4.7} since $v(\bar h)=h$ is minimal with respect to $v$ for all $\bar h\in S$. 

Suppose the statement is true for $\vert T\vert=n-1$. For a vertex $w_T$ and $\bar h\in S\setminus T$ we need to show that $\h\define v(\bar h)$ is minimal with respect to $w_T$ and then apply Lemma~\ref{le:4.7} to obtain the assertion for $=\vert T\vert=n$.

We know that $h$ is minimal with respect to $v$, i.e.\ there is no $\bar k\in H$ such that $v(\bar k)=k$ and $k\less h$. 
For all $k\in H\setminus T$  we have that  $v(\bar k)=w_T(\bar k)$ and hence 
\begin{equation}\label{*}
\not\exists\, \bar k\in H\setminus T \text{ such that }w_T(\bar k) =k \text{ and }k\less h.
\end{equation}
Further by the definition of $w_T$ we know for all $k\in T$ that $(v(\bar k))^\star =w_T(\bar k)$. 

But $h$ is transversal to $k\define v(\bar k)$ for all $\bar k\in T$ and hence $h\not\less k$,  $h\not\less k^\star$, $h^\star\not\less k$ and $h^\star\not\less k^\star$. From this we deduce
\begin{equation}\label{**}
\not\exists\, \bar k\in T \text{ such that } w_T(\bar k)=k^\star \text{ and } k^\star\less h.
\end{equation}
Combining equations (\ref{*}) and (\ref{**}) we may deduce that $h$ is minimal with respect to $w_T$ and hence, by Lemma~\ref{le:4.7} the map $w_T $ is in fact a vertex.

It is easy to see that the vertices defined like this do span the one-skeleton of an $n$-cube. And by constrcution of $X(H)$ the $n$-cube itself has to exist.
\end{proof}

\begin{figure}[h]
	\begin{center}
		\resizebox{!}{0.4\textwidth}{\includegraphics{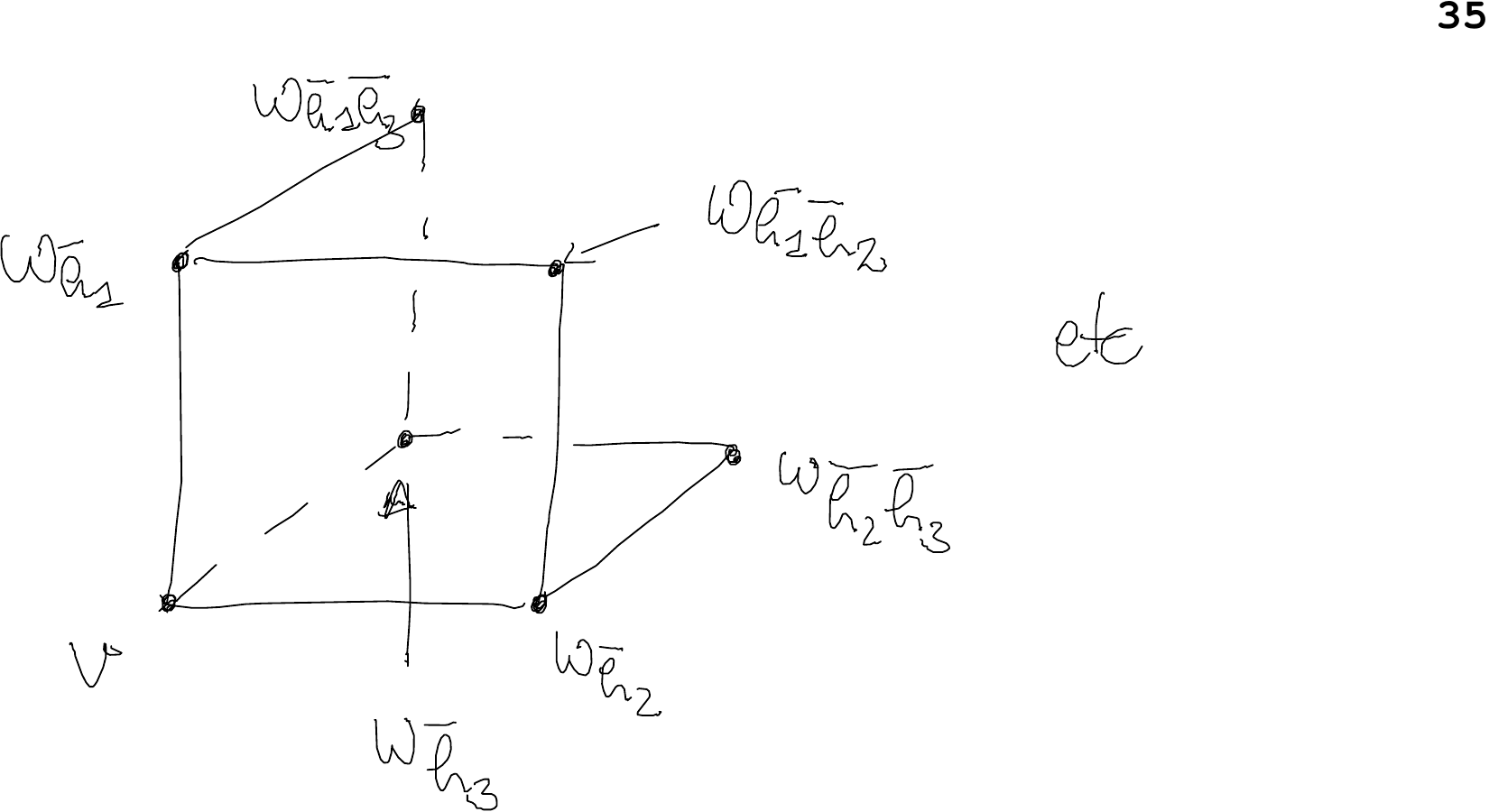}}
		\caption[links]{A set of $n$ pairwise transitive hyperplanes spans an $n$-cube. }
	\end{center}
	\label{fig_35}
\end{figure}

We are now ready to prove the main theorem of this section. 

\begin{proof}[Proof of Theorem~\ref{thm:4.3}]
From Lemma~\ref{le:4.9} we obtain that the $n$-cubes in $X(H)$ are in one-to-one correspondence with the families of $n$ pairwise transversal hyperplanes in $H$. It hence remains to prove that $X(H)$ is locally $\cat(0)$ and each connected component is simply connected. 

Consider a vertex $v:\bar H\to H$ in $X\define X(H)$.
In the link $\lk_X(v)$ there are no bigons, since otherwise there would exist pockets in $X$, compare Figure 28 below. 
But pockets can't exist since the two hyperplanes which define neighbours $p,q$ of $v$ uniquely determine a third vertex $u$ if they are transversal. 

\begin{figure}[h]	\label{fig_36}
	\begin{center}
		\resizebox{!}{0.4\textwidth}{\includegraphics{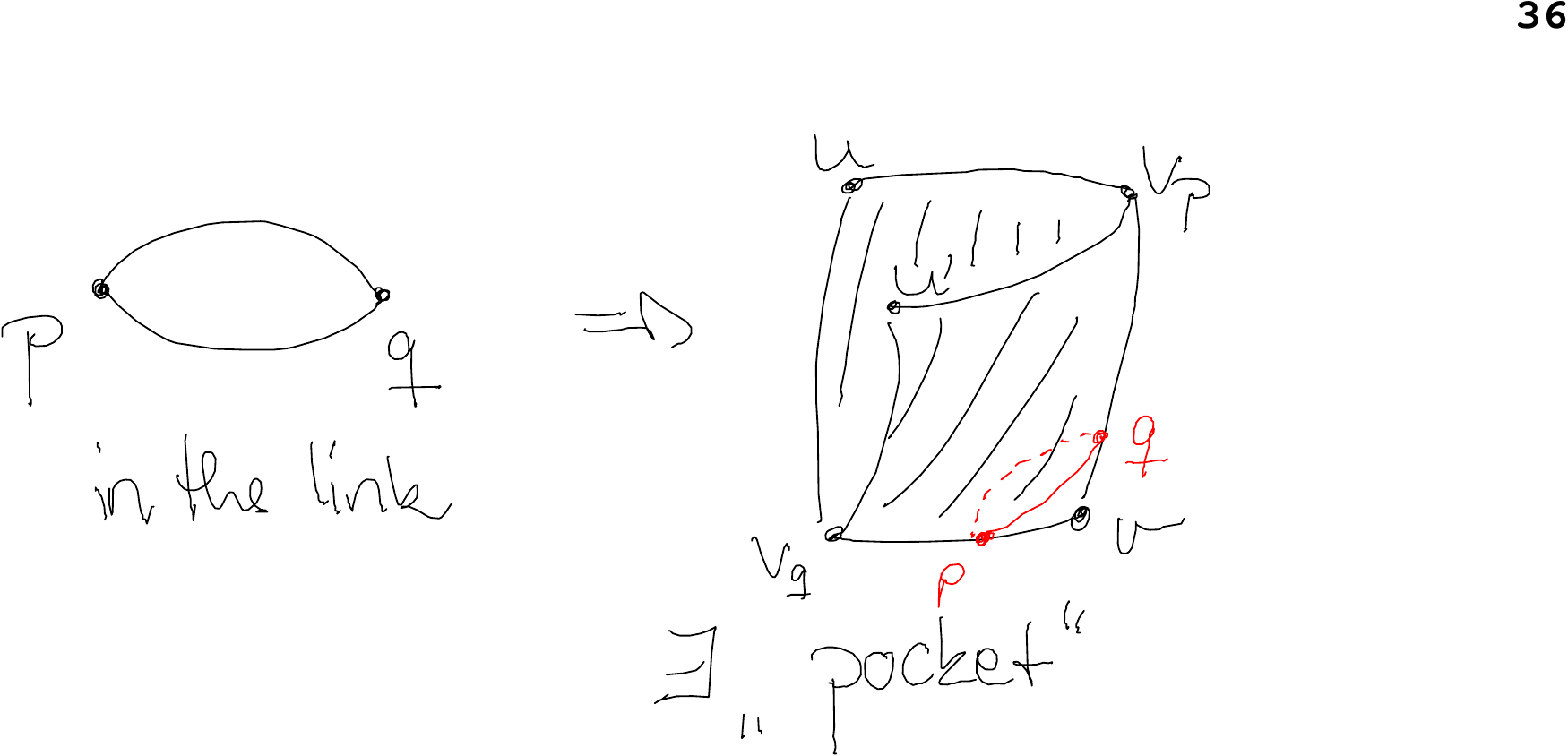}}
		\caption[links]{The existence of bigons in links violates the uniqueness of a square determined by transversal hyperplanes.}.
	\end{center}
\end{figure}

Suppose there is an $n-1$-skeleton of an $n$-simplex in $\lk_X(v)$. Then there are $n$ hyperplanes $\bar h_1, \dots, \bar h_n$. such that each $\bar h_i$ defines a new vertex $v_i$ adjacent to $v$. Since each pair $\bar h_i, \bar h_j$ spans a square $Q_{ij}$ (the corresponding edge exists in the link, hende $Q_{ij}$ exists) we conclude by Lemma~\ref{le:4.9} applied to $Q_{ij}$ that $h_i$ is transversal to $h_j$ for al $i\neq j$. 
Applying Lemma~\ref{le:4.9} to $h_1, \dots, h_n$ we see that there exists an $n$-cubehaving $v$ as a vertex and the $Q_{ij}$ as faces. Hence the $n-1$-skeleton is the boundary of an $n$-simplex and thus $X$ is locally $\cat(0)$ by Proposition~\ref{prop:2.16}. 

Let now $Y$ be a conected component of $X$. We need to prove that $Y$ is simply connected.

Let us make the following observations (the proof of which we leave as an exercise):
\begin{itemize}
 \item Closed paths in the 1-skeleton $Y^{(1)}$ of $Y$ without backtracking do contain an even number of edges.
 \item Write $d_1(v,y)$ for the number of hyperplanes separating $v$ and $w$. This defines a metric on the set of vertices of $X$ and induces a metric on $Y^{(1)}$.
\end{itemize}
Let $v_0$ be a vertex in $Y$ and $l$ a closed path in $Y^{(1)}$ containing $v_0$. Let further $v$ be a vertex on $l$ having maximal distance (in $Y^{(1)}$ with respect to $d_1$)  to $v_0$. There are possibly more than one such vertex. We pick any of them.
Choose neighbours $a,b$ of $v$ such that $v$ differs from $a$, respectively $b$, in precisely one hyperplane $\bar h_a$, respectively $\bar h_b$, and such that $v(\bar h_x)=h_x$ for both $x=a$ and $x=b$. Compare Figure 29. 

\begin{figure}[h]
	\begin{center}
		\resizebox{!}{0.4\textwidth}{\includegraphics{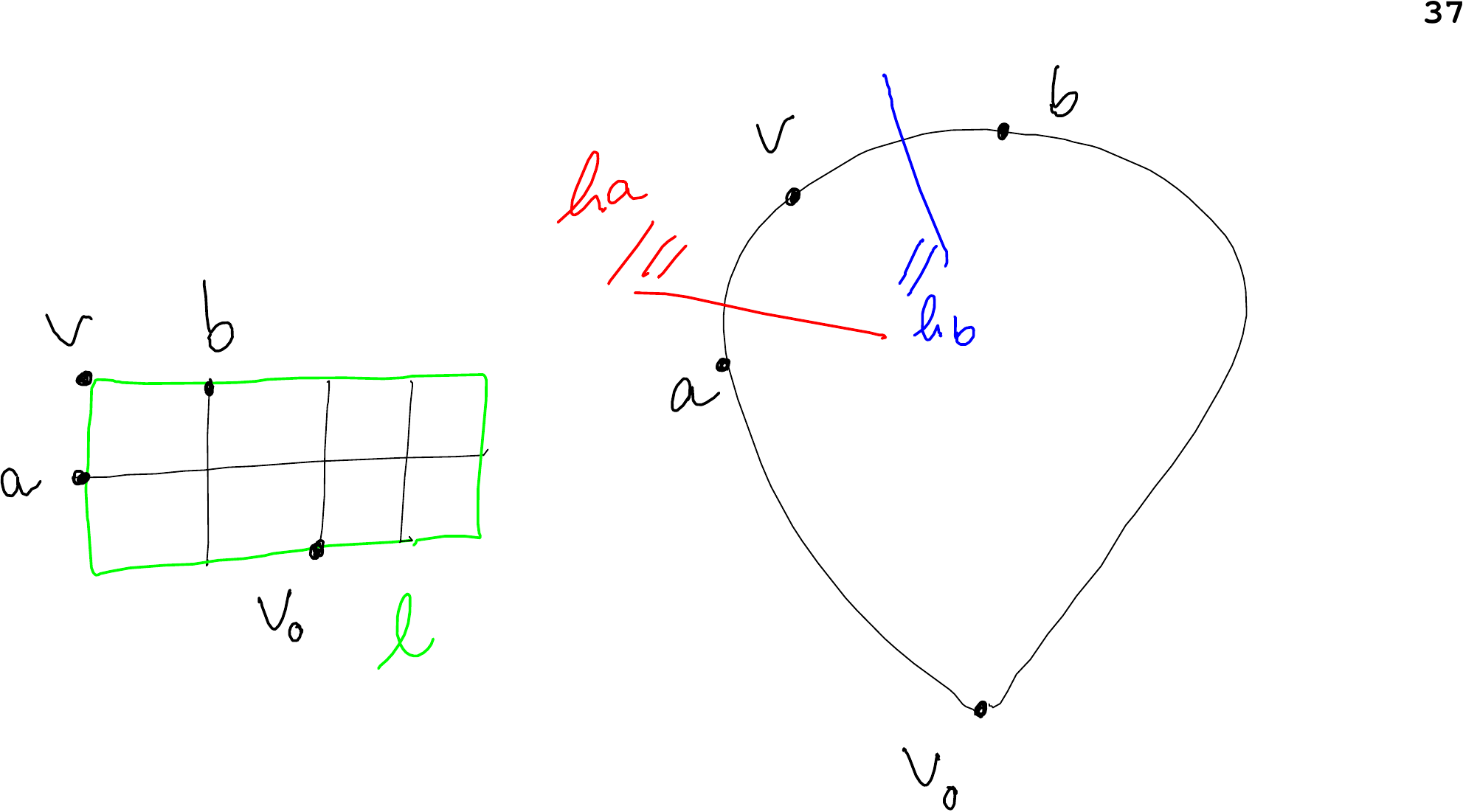}}
		\caption[links]{Closed paths in the proof of Theorem~\ref{thm:4.3}.}.
	\end{center}
	\label{fig_37}

\end{figure}

Claim 1: $\bar h_a$ and $\bar h_b$ are transversal. 
In order to show this claim we have to verify three inequalities: First $h_a\not\less h_b\not\less h_a$, then $ h_a\not\less h_b^\star$ and finally $h_a^\star \not\less h_b$. 

By construction $h_a$ and $h_b$ are minimal w.r.t.\ $v$. Hence there does not exist any $k$ with $v\in k$ and $k\less h_a$ respectively $k\less h_b$. In particular we have that $v(\bar h_a)=h_a\not\less h_b$ and $v(\bar h_b)=h_b\not\less h_a$. 

The halfspace $h_b^\star$ is minimal with respect to $b$. Hence there does not exist a hyperplane $\bar k$ with $b(\bar k)=k$ and $k\less h_b^\star$ and we arrive at the fact that $b(\bar h_a)=h_a \not\less h_b^\star$.

To prove the last inequality $h_a^\star \not\less h_b$ we first show a second claim.

Claim 2: $v_0(\bar h_a) =h_a^\star$ and $v_0(\bar h_b) =h_b^\star$.
We only prove $v_0(\bar h_a) =h_a^\star$ (by contradiction) as the proof of the second assertion is analogous.

Suppose $v_0(\bar h_a) =h_a$. The vertices $v$ and $v_0$ then agree on $\bar h_a$ and we have that $\bar h_a\in \{\bar h\,\vert\, v(\bar h)\neq v_0(\bar h) \}\ddefine S$. The distance between $v$ and $v_0$ is $d_1(v,v_0)=\vert S\vert$ and 
\begin{equation*}
a(\bar h)=\left\{
\begin{array}{c l}     
   v(\bar h) &  \text{ if } \bar h\neq \bar h_a \\
   v(\bar h)^\star& \text{ if }\bar h=\bar h_a. 
\end{array}\right. 
\end{equation*}
Therefore $\{\bar h\,\vert\,v_0(\bar h)\neq a(\bar h)\} = S\cup\{\bar h_a\}$ and 
$d_1(v_0, a)=d_1(v_0, v)+1$ which contradicts the choice of $v$. Hence Claim 2. 

From Claim 2 we may deduce that 
$h_a^\star=v_0(\bar h_a) \not\less v_0(\bar h_b )^\star ={h_b^\star}^\star = h_b$ and therefore that Claim 1 is true.

We have thus shown that $h_a$ and $h_b$ are transversal and define a 2-cube in $Y$ having vertices $v, a, b$ and $v_{ab}$, where $v_{ab}$ is given by 
\begin{equation*}
v_{ab}(\bar h)=\left\{
\begin{array}{c l}     
   v(\bar h)^\star &  \text{ if } \bar h\neq \bar h_a \text{ or }  \bar h= \bar h_b \\
   v(\bar h) & \text{ else }
\end{array}\right. 
\end{equation*}
By construction and Claim 2 $d(v_0, v_{ab})= d(v_0, v)-2$. 

\begin{figure}[h]
	\begin{center}
		\resizebox{!}{0.4\textwidth}{\includegraphics{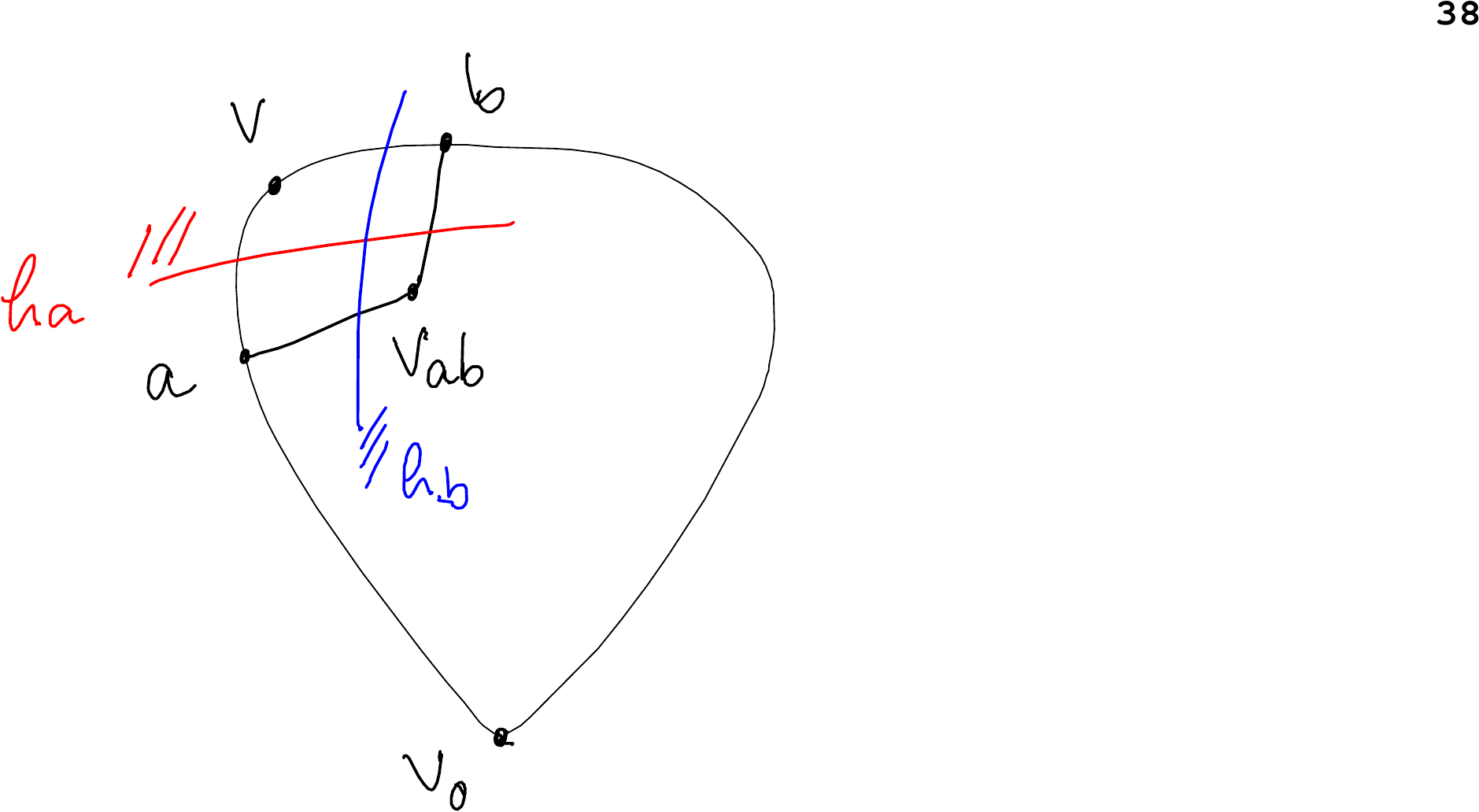}}
		\caption{Defining a replacement $l'$ of the original loop.}.
	\end{center}
	\label{fig_38}
\end{figure}

Define (as illustrated in Figure 30)  
a new loop $l'$ by replacing the two edges connecting $a,v, b$ by the path $a, v_{ab}, b$ and possibly deleting any newly obtained backtracks. Then, since  $d(v_0, v_{ab}) < d(v_0, v)$ we obtain by this procedure a homotopy contracting $l$ to $v_0$. Hence the assertion.
\end{proof}


\begin{example}\label{ex:4.10}
Ad hoc examples of $X(H)$:
\begin{enumerate}
 \item Standard cubulation of $\R^2$. See Figures~\ref{fig_39} and~\ref{fig_40}.
 \item Cubulation of $\tilde A_2$. See Figures~\ref{fig_43} and~\ref{fig_45}.
\end{enumerate}
\end{example}

\begin{figure}[h]
	\resizebox{!}{0.4\textwidth}{\includegraphics{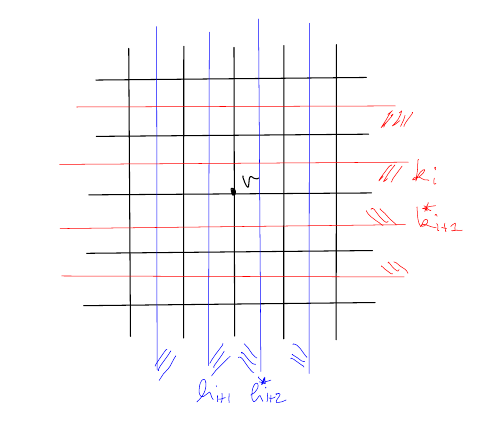}}
	\caption{Hyperplanes for the standard cubulation of $\R^2$.}.
	\label{fig_39}
\end{figure}

\begin{figure}[h]
	\resizebox{!}{0.3\textwidth}{\includegraphics{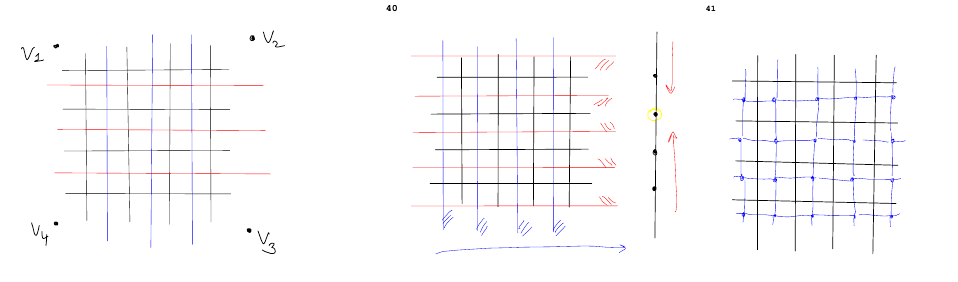}}
	\caption{The construction not only describes the (dual of the) original cubes (right hand side) but also points on the "boundary" (left side and middle).}.
	\label{fig_40}
\end{figure}

\begin{figure}[h]
	\resizebox{!}{0.4\textwidth}{\includegraphics{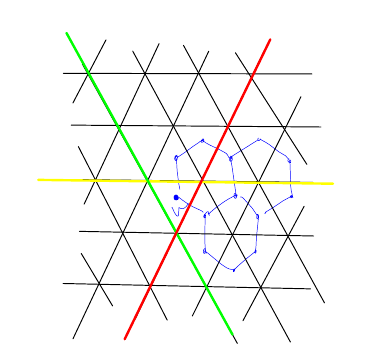}}
	\caption{Hyperplanes for the equilateral triangulation of the plane are all the ones parallel to the green, yellow and red one. In blue we see some of the vertices in the resulting cubical complex.}.
	\label{fig_43}
\end{figure}

\begin{figure}[h]
\begin{center}
	\resizebox{!}{0.3\textwidth}{\includegraphics{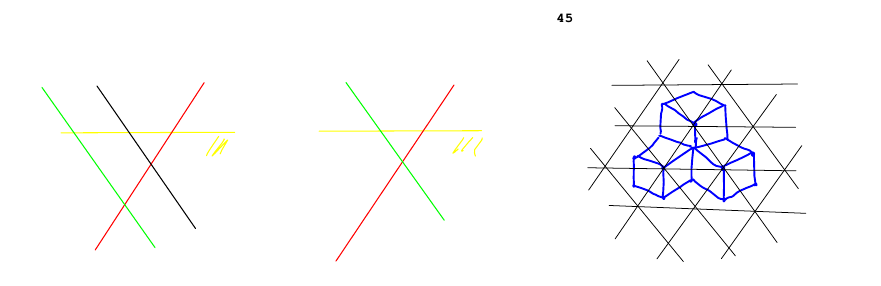}}
	\caption[links]{Any triple of pairwise transversal hyperplanes induces a 3-cube.}.
	\label{fig_45}
\end{center}
\end{figure}


\begin{remark}
One can prove 
\begin{itemize}
 \item Constructing $X(H)$ with $H$ the halfspace system of a cubical complex $C$ yields a compactification of $C$ with the Roller boundary. For details see \cite{Roller} 
 \item One can prove that an action of a group $G$ on a cubical complex $C$ can be extended to a certain nice subset of the Roller boundary. See \cite{NevoSageev}.
\end{itemize}
\end{remark}

Example~\ref{ex:4.10} suggest that one should be able to cubulate Coxeter groups. We will show that this is in fact possible. But let us first recall some basic facts about geometric group theory and Coxeter groups  and -complexes as well as their presentations.

%
%

\newpage
\section{Elementary notions in geometric group theory}

Given a set of generators $S$ and a set $R$ of relations in $S$, i.e. words in $S\cup S^{-1}$,  we may consider the following universal construction of a group 
\begin{equation}\label{def:5.0}
\langle S\vert R\rangle\define F(S)\diagup \langle R^F\rangle
\end{equation}
where $F(S)$ is the free group generated by $S$ and $R^F$ is the image of $R$ under conjugation by $F(S)$, that is the subgroup of $F(S)$ generated by $R$.
Note that there is a homomorphism $\phi: F(S)\to \langle S\vert R\rangle$. 

A group $G$ is \emph{finitely generated}, resp. \emph{finitely presented},  if there exists $S$ and $R$ such that $G$ is isomorphic to $\langle S\vert R\rangle$ and $S$ is finite, respectively $S$ and $R$ are finite.

\begin{definition}\label{def:5.1}
The \emph{Cayley graph} $\Cay=\Cay(G,S)$ of a finitely generated group $G$ with generating set $S$ is defined as follows: 
The vertices of $\Cay$ are $G$ and there is an edge $(g,h)\in G\times G$ if $gs=h$ for some $s\in S$. The edge $(g, gs)$ is then labeled by $s$. If $s$ is of order two in $G$ we consider $(g,gs)$ not as a double edge but  as an oriented edge from $g$ to $gs$. 
\end{definition}

\begin{example}\label{ex:5.2}
\begin{enumerate}
 \item $\Cay(\Z, \{1\})$, see Figure~\ref{fig_47}. 
 \item $\Cay(\Z, \{2,3\})$, see Figure~\ref{fig_47}. 
 \item $\Cay(F_2, \{a, b\})$, see Figure~\ref{fig_49}. 
 \item $\Cay(\Z^2, \{a, b\})$, see Figure~\ref{fig_50}.
 \item $\Cay(\Z\diagup{\Z_6}, \{1\})$, see Figure~\ref{fig_50}.
\end{enumerate}
\end{example}

\begin{figure}[h]
	\begin{center}
		\resizebox{!}{0.3\textwidth}{\includegraphics{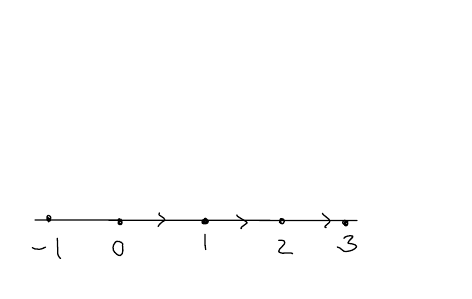}}
		\resizebox{!}{0.2\textwidth}{\includegraphics{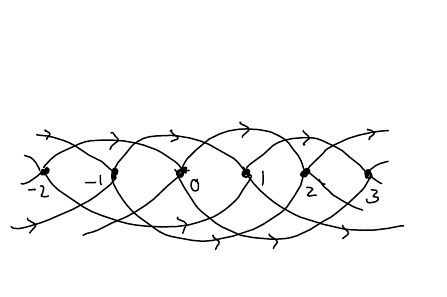}}
		\caption{On the left $\Gamma(\Z, \{1\})$, on the right $\Gamma(\Z, \{2,3\})$.}
	\label{fig_47}
	\end{center}
\end{figure}

\begin{figure}[h]
	\begin{center}
		\resizebox{!}{0.3\textwidth}{\includegraphics{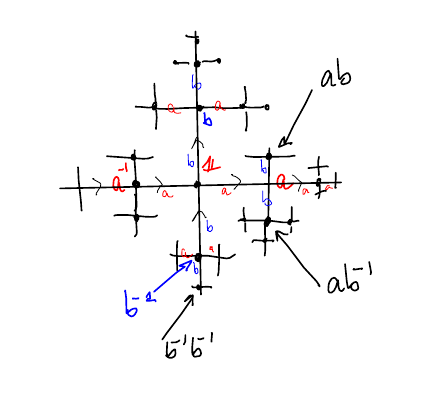}}
		\caption{The Cayley graph $\Gamma(F_2, \{a,b\})$.}
	\end{center}
	\label{fig_49}
\end{figure}

\begin{figure}[h]
	\begin{center}
		\resizebox{!}{0.3\textwidth}{\includegraphics{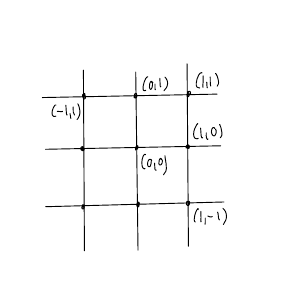}}
		\resizebox{!}{0.3\textwidth}{\includegraphics{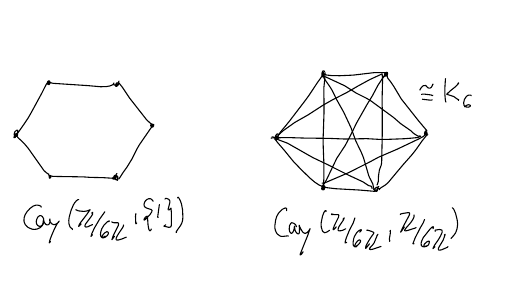}}
		\caption{On the left $\Gamma(\Z^2, \{a,b\})$, middle and right the Cayley graphs of $\Z\diagup \Z_6$ with two different generating sets.}
	\label{fig_50}
	\end{center}
\end{figure}

\begin{prop}\label{prop:5.3}
The following hold for any group $G$ generated by $S$. 	
\begin{itemize}
 \item $\Cay(G,S)$ is connected
 \item $\Cay(G,S)$ is regular and each vertex is contained in $\vert S\cup S^{-1}\vert$ edges. 
 \item $\Cay(G,S)$ is locally finite if and only if $S$ is finite.
 \item $G$ acts on $\Cay(G,S)$ via left multiplication, that is $g\dot v=g v$. 
\end{itemize}
\end{prop}

One can consider the Cayley graph of a group as a metric space and define the so called \emph{word metric} on $\Cay$. 

\begin{definition}\label{def:5.4}
Let $G$ be a finitely generated group and $\Cay=¸\Cay(G,S)$ its Cayley graph with respect to $S$. 
The distance $d_S(g,h)$ of two vertices $g$ and  $h$ is defined to be the minimal number of edges of a path connecting $g$ and $h$ in $\Cay$. We call $d_S$ the \emph{word metric} on $G$ with respect to $S$. The \emph{length} of an element $g\in G$, denoted by $l(g)$, is the distance between $g$ and the identity, i.e. $l(g)\define d_S(g, 1)$.
\end{definition}

It is easy to see (exercise!) that $d_S:G\times G \to \N$ is in fact a metric on the set of vertices of $\Cay$. 

\begin{definition}\label{def:5.5}
A map $f:X\to Y$, with $X,Y$ metric spaces, is a \emph{quasi-isometric embedding} if there exist constants $C;D\in \R _{> 0}$ such that for all $x,y\in X$
\begin{equation*}
 \frac{1}{C}d_X(x,y) -D \leq d_Y(f(x),f(y)) \leq Cd_X(x,y) +D.
\end{equation*}
We call $f$ a \emph{quasi-isometry} if the image is $R$-dense in $Y$, i.e.\ 
\begin{equation*}
\exists R<\infty \text{ such that } \forall y\in Y \;\exists x\in X \text{ with } d_Y(f(x),y) <R.
\end{equation*}
\end{definition}

\begin{prop}\label{prop:5.6}
Any pair of Cayley graphs $\Cay(G,S), \Cay(G,T)$ for two generating systems $S$ and $T$ of a group $G$ are quasi-isometric. Further $(G,d_S)$ is quasi-isometric to the Cayley graph $\Cay(G,S)$.  
\end{prop}
\begin{proof}
 We leave the proof as an exercise. See also \cite[p. 138ff]{BH}.
\end{proof}

\begin{example}\label{ex:5.7}
The following maps are quasi-isometries:
\begin{itemize}
 \item $f:\Z^2\hookrightarrow\R^2$ where we equip $\Z^2$ with the restricted Euclidean metric.
 \item $f:\R\to\Z : x\mapsto \lfloor x\rfloor$
 \item $f:X\to \{pt\}$ for any bounded metric space $X$.
\end{itemize}
The following maps are \underline{not} quasi-isometries:
\begin{itemize}
 \item $f:\R\to\R:n\mapsto n^2$
 \item any map from $(R^2, d_{eucl})$  to  the hyperbolic plane $\h^2$ with its standad metric. 
\end{itemize}
 The following groups are quasi-isometric:
\begin{itemize}
 \item $\Z$ and $\Z\star\Z$
 \item any pair of finite groups
\end{itemize}
\end{example}

\begin{definition}\label{def:5.8}
Let $G$ be a group, $\mathcal{C}$ a category and $X\in \mathrm{Obj}(\mathcal{C})$. A \emph{group action} of $G$ on $X$ is a homomorphism $G\to \Aut_\mathcal{C}(X)$, that is a family of elements $\left \{ f_g\in\Aut(X)\right\}_{ g\in G}$ such that 
\begin{equation*}
f_g\circ f_h =f_{gh} \;\;\forall g,h\in G. 
\end{equation*}
We write $g.x$ for $f_g(x)$ and shorthand $G\gact X$ for the existance of a group action of $G$ on $X$.
\end{definition}

Take for example $\mathcal{C}$ to be the category of topological spaces and $\Aut_\mathcal{C}(X)=Homeo(X)$ or let $\mathcal{C}$ be the category of cubical complexes and $\Aut_\mathcal{C}(X)$ the automorphisms respecting the cubical structure. 

\begin{definition}\label{def:5.9}
We say that a given action $G\gact X$ is \emph{geometric} if 
\begin{enumerate}
 \item the action is properly discontinuous, i.e.\ for all compact sets $K\subset X$ the set $\left \{ g\in G \mvert  g.K\cap K \neq \emptyset \right\}$ is finite
 \item $X\diagup G$ is compact w.r.t.\ the quotient topology and
 \item $G$ acts by isometries.
\end{enumerate}
\end{definition}

\begin{remark}\label{rem:5.10}
Recall that a proper map is one for which all preimages of compact sets are compact. With this notion for an action to be properly discontinuous in the sense mentioned above is equivalent to the map $G\times X\to X\times X : (g,x)\mapsto (x, g.x)$ being proper while $G$ carries the discrete topology. 
\end{remark}

\begin{thm}(\v{S}varc-Milnor Lemma)\label{thm:5.11}
Suppose $G$ acts geometrically on a proper\footnote{Recall that a space is \emph{proper} if every closed ball $B_r(x)=\{y\in X \vert d(x,y)\leq r\}$ is compact.}  geodesic metric space $X$ then $G$ is finitely generated and quasi-isometric to $X$.
\end{thm}
\begin{proof} For a proof of this fact see \cite[I.8.19]{BH} \end{proof}

\begin{remark}
For the \v{S}varc-Milnor Lemma we may alterntively assume that $X$ is a length space which is slightly stronger than being proper and geodesic. Compare \cite[I.8.4]{BH}.
\end{remark}

\begin{example}\label{ex:5.12}
A group $G$ always acts by left-multiplication on its Cayley graph $\Gamma\define \Cay(G,S)$: 
\begin{equation*}
G\to \Aut(\Gamma) : g\mapsto l_g\define(h\mapsto gh).
\end{equation*}
Prove as an exercise that this map is well defined and a homomorphism. Moreover this action is geometric. 
\end{example}

\begin{remark}\label{rem:5.13}
\begin{enumerate}
 \item For any basepoint $x_0$ the following set $\mathcal{A}$ is a generating set for $G$ as in the \v{S}varc-Milnor Lemma:
\begin{equation*}
\mathcal{A}\define \left\{ g\in G\mvert g.B_r(x_0) \cap B_r(x_0)\neq\emptyset \right\}
\end{equation*}
 where $r>0$ such that $C\subset B_{r/3}(x_0)$ for a (fixed) compact set $C$ such that $G.C=X$. 
\item The action in \ref{ex:5.12} is free iff $S$ does not contain elements of order two. 
\end{enumerate}
\end{remark}

%
%

\newpage
\section{Cubulating Coxeter groups}

In this section we will apply Sageev's cubulation method to Coxeter groups and illustrate how it can be carried out for this interesting and important class of groups. The goal is to find a good candidate for a set of hyperplanes and show that it does indeed satisfy the requirements of Definition~\ref{def:4.1}. This chapter is based in \petra{\cite{???}} but we give a new proof of the main result which is more straightforward and much shorter. 

Let me first remind you about Coxeter groups and their simplicial complexes. 

\begin{definition}\label{def:6.1}
A \emph{Coxeter matrix} $M=(m_{ij})_{i,j}$ is a symmetric matrix with $m_{ij}=1$ if $i=j$ and such that $2\leq m_{ij}\in\N_\cup\{\infty\}$ for all other $i,j$.
\end{definition}

\begin{definition}\label{def:6.2}
A \emph{Coxeter group} $(W,S)$ is a group $W$ with presentation
\begin{equation*}
 \langle s_1, \ldots, s_n \vert\; (s_is_j)^{m_{ij}}=1\rangle
\end{equation*}
where $(m_{ij})_{i,j}$ is a Coxeter matrix. Here we mean by $(s_is_j)^{\infty}$ that there is no relation for the elements $s_i, s_j\in S$.
\end{definition}

\begin{definition}\label{def:6.3}
A \emph{Coxeter diagram} is a graph $T$ associated to a Coxeter group $(W,S)$ with Coxeter matrix $M=(m_{ij})_{i,j}$ as follows:
There is one vertex $i$ in $T$ for each element $s_i$ in $S$.
There is an edge between $i$ and $j$ in $T$ iff $m_{ij}\geq 3$ and these edges are labeled with $m_{ij}$ in case $m_{ij}\geq 4$.
\end{definition}
%

\begin{remark}\label{rem:6.5}
 Coxeter groups may be seen as abstract versions of reflection groups, that is subgroups of a group of automorphisms of say a Euclidean  space, a sphere or the hyperbolic plane,  which are generated by reflections.
 Moreover reflection groups may be seen as linear presentations of a Coxeter group. Here the relation $s_i$ amnd $s_j$ with order $m_{ij}$ of their product may be seen as reflections along hyperplanes which meet at an angle of $\frac{\pi}{m_{ij}}$. There exists a classification of finite reflection groups and one of Euclidean reflection groups. Details can be found in \cite{Humphreys} for example.
\end{remark}

\begin{example}\label{ex:6.6}
Here are a few examples of Coxeter groups:
\begin{enumerate}
 \item Type $\tilde A_2$ is an example of a Euclidean reflection group wit presentation $W+\langle s_1,s_2,s_3 \vert; s_i2=(s_is_j)3=1 \forall i\neq j\rangle$.
 \item Type $I_2(p)$ is an example of a spherical reflection group with presentation $W=\langle s_1.s_2 \; (s_1s_2)^p=s_i2=1, \text{ for } i=1,2 \rangle$.
 \item Any triangle with angles $\frac{\pi}{k}, \frac{\pi}{l}, \frac{\pi}{m}$ with $\frac{1}{k}\frac{1}{l}\frac{1}{m}<1$ gives rise to a reflection subgroup $W$ of the automorphisms group of the hyperbolic plane with presentation
 $W=\langle s_1, s_2, s_3\vert\; s_i2=(s_1s_2)^k=(s_1s_2)^l=(s_1s_2)^m=1, \text{ for all } i\rangle$.
 \item Symmetry groupf of regular polytopes are Coxeter groups as well. The symmetric group $S_n$ for example is the symmetry group of a regular n-simplex. An n-cube has symmetry group $BC_n$ and an Ikosaeder has $H_3$ as its symmetry group. 
\end{enumerate}

In Figure~\ref{fig_52} you can find visualizations of the first three examples as reflection groups.
\end{example}

\begin{figure}[h]
	{\center	
		\resizebox{!}{0.65\textwidth}{\includegraphics{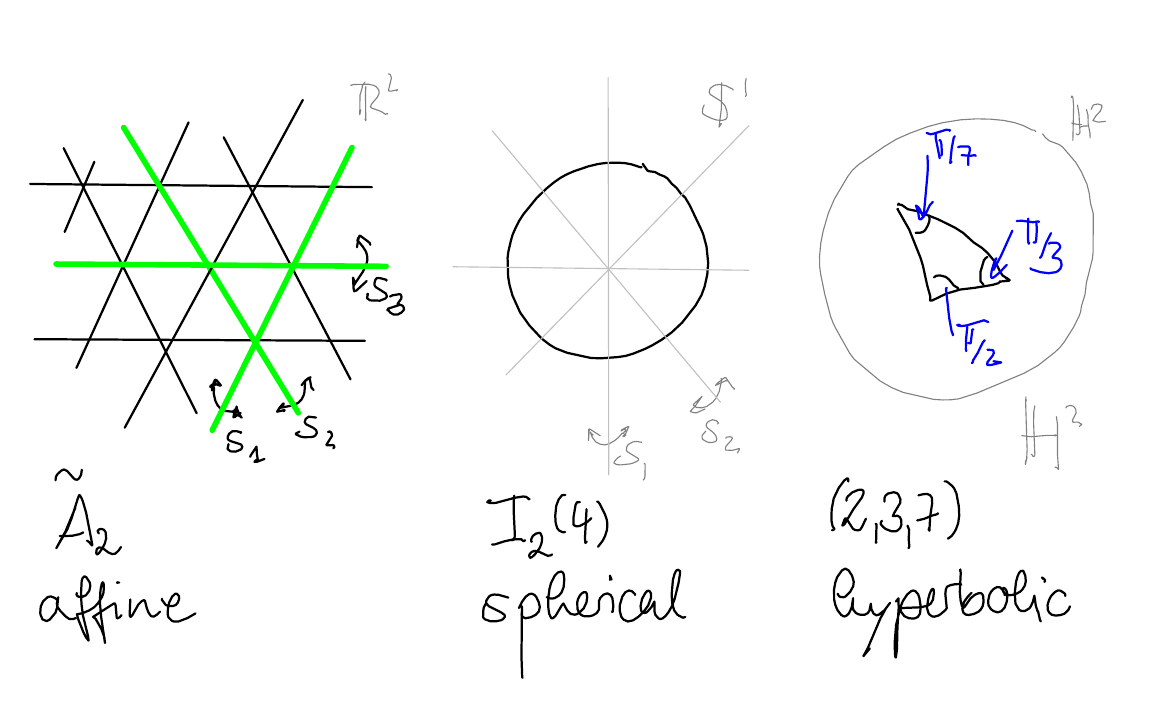}}
		\caption{Examples of Coxeter groups as reflection groups.}
		\label{fig_52}
	}
\end{figure}

\begin{definition}[\bf Deletion Condition]\label{def:6.7}
The following crucial condition a Coxeter group $(W,S)$ satisfies is known as the \emph{deletion condition}:
\begin{itemize}
 \item[(D)] $w=s_1\ldots s_m$ with $m>l_S(w)$ then there exists $i<j$ such that $w=s_1 \ldots \hat{s_i} \ldots \hat{s_i}\ldots s_m$.
\end{itemize}
Equivalent formulations of the deletion condition are the \emph{exchange condition}
\begin{itemize}
 \item[(E)] $w\in W$, $s\in S$ where $w$ has the reduced presentation $w=s_1\ldots s_d$. Then either $l_S(sw)=d+1$ or there exists an $i$ such that $w=s s_1 \ldots \hat{s_i} \ldots s_d$.
\end{itemize}
and the \emph{folding condition}
\begin{itemize}
 \item[(F)] $w\in W$, $s, t\in S$ with $l_S(sw)=l_S(w)+1$ and $l_S(wt)=l_S(w)+1$ then either $l_S(swt)=l(w)+2$  or $swt=w$.
\end{itemize}
\end{definition}

For a proof that these conditions are equivalent and satisfied by Coxeter groups see \cite{AB} or \cite[Thm 3.3.4 p.38]{Davis}.

\begin{definition}\label{def:6.8}
A \emph{reflection} $r$ in a Coxeter group $(W,S)$ is a conjugate $wsw^{-1}$ of a generator $s\in S$. The \emph{wall} $M_r$ is the set of edges $(u,v)$ in the Cayley graph of $W$ such that $u=rv$.
\end{definition}

Note that the edge $(w,ws)$ is contained in $M_r$ for $r=wsw^{-1}$. Also, each edge in $\Gamma_W$ is contained in a unique wall and we have a one-to-one correspondence between walls $M_r$ and reflections $wsw^{-1}$ in $W$.

\begin{lemma}\cite[2.4]{Ronan}\label{le:6.9}
Let $(W,S)$ be a Coxeter group and  $u,v$ adjacent vertices in $\Gamma_W$. Then $d_S(x,u)=d_S(x,v)\pm 1$ for all $x\in W$.
\end{lemma}
\begin{proof}
A minimal presentation of $x^{-1}u=s_1\dots s_d$ corresponds to a geodesic path $\gamma$ in $\Gamma_W$ from $x$ to $u$ consisting of edges which are labeled with $s_1, s_2$, etc. And the vertices $x, xs_1, xs_1s_2, \ldots ,u=xs_1\cdots s_d$ are appear along $\gamma$ in this order. 

Elongate $\gamma$ by one edge with label $s$ so that it goes to $v=us$. This corresponds to the product $x^{-1}us$. Now either $l(x^{-1}v)=l(x^{-1}u)+1$ or $l(x^{-1}v) < l(x^{-1}u)$. From (D) we can conclude that there exist two letters in the presentation which may be deleted and $l(x^{-1}v)=l(x^{-1}u)-1$ which finishes the proof. 
\end{proof}

\begin{definition}\label{def:6.10}
Let $\gamma=(c_0, c_1, \ldots, c_k)$ be a path in $\Gamma_W$. We say that $\gamma$ \emph{crosses} the wall $M_r$ if there exists $1\leq i\leq k$ such that the reflection $r$ maps $c_{i-1} $ to $c_i$ and vice versa.  
\end{definition}

\begin{lemma}\label{le:6.11}
Let $(W,S)$ be a Coxeter group with Cayley graph $\Gamma$. 
\begin{enumerate}
 \item\label{6.11.1} A minimal path in $\Gamma$ crosses each wall at most once. 
 \item\label{6.11.2} Given $x,y$ vertices in $\Gamma$ then the number of crossings of $M_r$ modulo two is indepenent of the choice of a path $\gamma:x\rightsquigarrow y$. 
\end{enumerate}
\end{lemma}

In other words the second assertion of the above lemma states that any path connecting $x$ and $y$ either crosses $M_r$ an even number of times or an odd number of times. So $M_r$ either does ``separate'' $x$ and $y$ or does not. 

\begin{proof}
To prove the first assertion let $\gamma=(c_0, c_1, \ldots, c_k)$ be a minimal path in $\Gamma$ and suppose there exists $M_r$ which is crossed by $\gamma$ twice, say at positions $i$ and $j$. Then the minmal path $(c_i, c_{i+1}, \dots, c_{j-1})$ is mapped onto a path of the same length by  $r$ but now connecting $r(c_i)=c_{i-1}$ and $r(c_{j-1})=c_j$. But then we may shorten $\gamma$ to the path $(c_1, \dots, c_{i-1}=r(c_i), r(c_{i+1})\dots, , r(c_{j-1})=c_j, c_{j+1}, \dots , c_k$ which contradicts minimality of $\gamma$. 

\begin{figure}[h]
	\begin{center}
		\resizebox{!}{0.3\textwidth}{\includegraphics{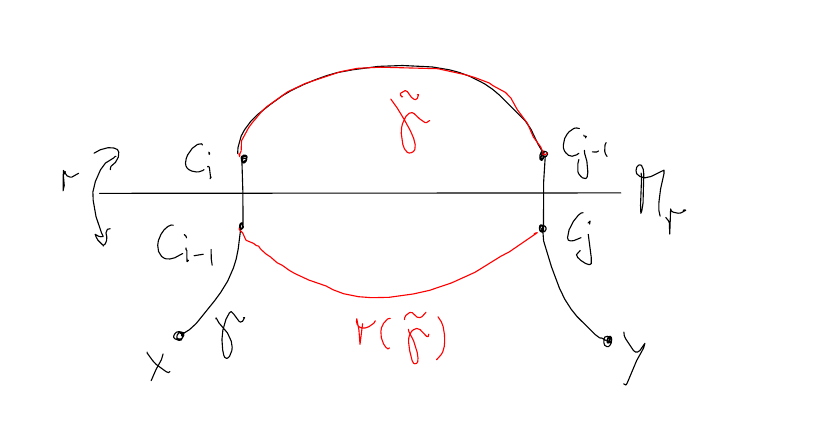}}
		\caption{A path in the proof of Lemma~\ref{le:6.11}.}
	\label{fig_53}
	\end{center}
\end{figure}

We now verify \ref{6.11.2}. Fix $x,y$ in $W$ and a (not neccessarily minimal) presentation $s_1\cdots s_d$ of $x^{-1}y$ and denote by $\gamma:x\rightsquigarrow y$ the associated path in $\Gamma$. Further let $M_r$ be a wall in $\Gamma$. 
We write  $n(\gamma)$ for the number of crossings of $M_r$ by $\gamma$. 
Suppose that $t_1\cdots t_k$ is a different presentation of  $x^{-1}y$. 

Each presentation of the element  $x^{-1}y$ in $W$ may be transformed into another presentation  by the following moves: 
\begin{enumerate}
 \item\label{6a} delete double appearance $ss$ of generators $s\in S$
 \item\label{6b} insert double appearance $ss$ of generators $s\in S$
 \item\label{6c} replace a string $s_is_js_i\cdots $ of $m_{ij}$ letters with the string $s_js_is_j\cdots$ of $m_{ij}$ letters. 
\end{enumerate}

Apply these moves to the two presentation of the element  $x^{-1}y$. It is easy to see that move \ref{6a} changes the value of $n(\gamma)$ by $-2$, move \ref{6b} by $+2$ and that  move \ref{6c} does not change the value of $n(\gamma)$. Hence the assertion. 
\end{proof}

\begin{definition}\label{def:6.12}
For a fixed wall $M_r$ in $\Gamma$ the Cayley graph of a Coxeter group $(W,S)$ we denote by  $n(\gamma)$, as in the previous lemma,  the number of crossings of $M_r$ by $\gamma$. The reflection $r$ defines two subsets  $\pm \alpha$ of the vertices of $\Gamma$, called \emph{roots},  by
\begin{equation*}
+\alpha = \{g\in W \vert\; n(g) \text{ is even} \} \text{ and } -\alpha = \{g\in W \vert\; n(g) \text{ is odd} \}.
\end{equation*}
We call $\alpha$ and $-\alpha$ \emph{opposite} in $\Gamma$. 
\end{definition}

This definition is (up to permuting the signs of the roots) independent of the choice of  connecting paths by Lemma~\ref{le:6.11}. Note also that the vertices of $\Gamma$ are precisely the union of $\alpha$ and $-\alpha$. 

For the proof of the following lemma compare also \cite[Prop. 2.6]{Ronan}.

\begin{lemma}\label{le:6.13}
Let  $u,v$ be adjacent vertices in $\Gamma$.  Define the set $H(u,v) \define\{c\in \Gamma \vert\; d_S(c,u) < d_S(c,v) \}$. For any root $\alpha$  with $u\in\alpha$ and $v\in -\alpha$ we then have  $\alpha = H(u,v)$.  
In particular $H(u,v)=H(x,y)$ for any pair of adjacent vertices $x,y$ with $x\in H(u,v)$ and $y\notin H(u,v)$. 
\end{lemma}
\begin{proof}
We first prove that $\alpha$ is contained in $H(u,v)$. Let $c\neq u$ be a vertex in $\alpha$. By the defintion of roots the number of crossings of $M_r$ a minimal path from $c$ to $u$ is even and  by Lemma~\ref{le:6.11}-\ref{6.11.1} that number is at most one, hence a minimal path does not cross the wall $M_r$. 
In particular such a path can not contain $v$. But then Lemma~\ref{le:6.9} implies that $d(c,v)=d(c,u)+1>d(c,u)$ and hence $c\in H(u,v)$. 

To verify the opposite containment suppose now that $c\in H(u,v)$. We need to prove: there exists a path connecting $v$ and $c$ which goes via $u$.
Such a path then by construction crosses $M_r$ and has crossing number at most 1 by Lemma~\ref{le:6.11}.\ref{6.11.1}. This implies $c\notin -\alpha$.

It remains to prove the claim that there exists a path connecting $v$ and $c$ which goes via $u$. Suppose $\gamma:c\rightsquigarrow u$ is minimal. Since $u$ and $v$ are adjacent we may deduce from Lemma~\ref{le:6.9}  that 
$$l(c^{-1}us)=l(c^{-1}v)=d(c,v)=d(c,u)+1=l(c^{-1}u).$$
The path $\gamma$ corresponds to a minimal presentation of $c^{-1}u$, say $s_1\cdots s_d$. Then $s_1\cdots s_ds$ is a minimal presentation of $c^{-1}v$ and the elongation of $\gamma$ by the edge $(u,v)$ is a minimal path which connects $u$ and $v$ and crosses $M_r$. 
\end{proof}

\begin{figure}[h]
	\begin{center}
		\resizebox{!}{0.3\textwidth}{\includegraphics{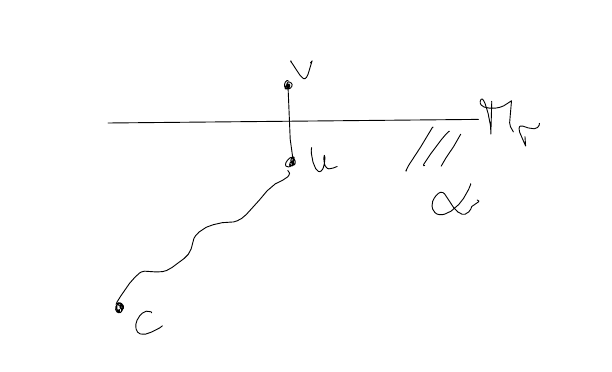}}
		\caption{A path in the proof of Lemma~\ref{le:6.13}.}
	\label{fig_54}
	\end{center}
\end{figure}

\begin{definition}(Halfspaces for $\Gamma_W$)\label{def:6.14}
Let $(W,S)$ be a Coxeter group  and let $u,v$ be adjacent vertices in the  Cayley graph $\Gamma_W$. That is $v=us$ for some $s\in S$. We write
\begin{equation*}
 H(u,v)=\left\{w\in W \mvert d_S(w,u) < d_S(w,v)\right\}.
\end{equation*}
and define 
\begin{align*}
\cH &\define\left\{ H(u,v) \mvert u,v\in W, d_S(u,v)=1\right\} \\
\star &:\cH\to \cH : H(u,v)\mapsto H(v,u)
\end{align*}
with $\cH$ being ordered by inclusion.
\end{definition}

\begin{figure}[h]
	\begin{center}
		\resizebox{!}{0.6\textwidth}{\includegraphics{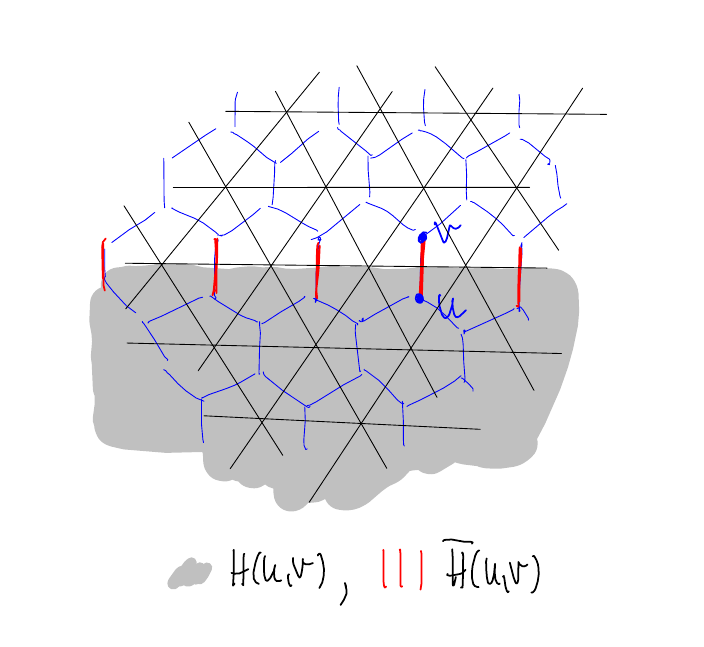}}
		\caption{The red edges are the wall $\bar H(u,v)$ of the grey shaded half space $H(u,v)$.}
	\label{fig_55}
	\end{center}
\end{figure}

\begin{prop}[(Niblo-Reeves)]\label{prop:halfspace}
	The triple $(\cH, \leq, \star)$ is a halfspace system in the sense of Definition~\ref{def:4.1}. 
\end{prop}

The following proof of Proposition~\ref{prop:halfspace} is different from (and shorter than) the one given in the literature.  

\begin{proof}
It is clear that $\subseteq$ forms a partial order on the set of halfspaces and that $\star$ is an order reversing involution of the same order as $W=H(u,v) \cup H(v,u)$. It is hence left to prove the following two things: 
\begin{enumerate}
	\item The finite interval condition is satisfied, that is for all $h_1, h_2$ there are only finitely many $k\in\cH$ suchh that $h_1\leq k\leq h_2$. 
	\item The nesting condition is satisfied, that is for all halfspaces $h$ we never have $h^\star = h$ (which is equivalent to the fact that at most one of the four conditions listed in Definition~\ref{def:4.1} holds). 
\end{enumerate}  
First observe that the nesting condition (2) is easily deduced from the fact that $H(U,v)^\star= H(v,u)$ and that the intersection of $H(u,v)$ with $H(v,u)$ is empty.

\begin{figure}[h]
	\begin{center}
		\resizebox{!}{0.35\textwidth}{\includegraphics{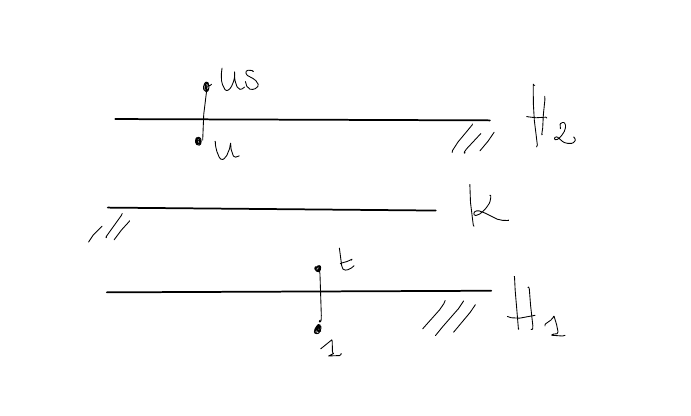}}
		\caption{Illustration of the proof of Proposition~\ref{prop:halfspace}.}
	\label{fig_56}
	\end{center}
\end{figure}

We now verify (1): Let $H_1\subsetneq H_2$ be two halfspaces and suppose without loss of generality that $H_1=H(1,t)$ for some $t\in S$ and $H_2=H(u,v)$ with $v=us, s\in S$ as shown in Figure~\ref{fig_56}. Otherwise move the $H_i$ in such a position using the left-translation action of $W$ on $\Gamma_W$. 
Observe that $t\in H_1$ and $u\notin H_2$ since otherwise $H_1=H_2$ by Lemma~\ref{le:6.13}. 

For any halfspace $K$ such that $H_1\subsetneq K\subsetneq H_2$ we may conclude that $t\in K$ and $u\notin K$ using again Lemma~\ref{le:6.13} and the fact that $K\neq H_i, i=1,2$. 

Suppose $\gamma =(1=c_0, c_1, \ldots, c_l=u)$ is a minimal path in the Cayley graph $\Gamma_W$ connecting the identity to the element $u$. Then $\gamma$ does cross the wall of $K$ exactly once (see Lemma~\ref{le:6.11}), say at the index $1\leq i\leq l-1$. That is $c_i\in K$ and $c_{i+1}\notin K$. From Lemma~\ref{le:6.13} we may then deduce that $K=H(c_i, c_{i+1})$ and may say that the wall of $K$ is defined by $\gamma$. Moreover each such path $\gamma:1\rightsquigarrow u$ defines at most $l-1$ wlls which do not necessarily all lie between $H_1$ and $H_2$. Compare Figure~\ref{fig_57}. 

Choose an element $u$ of minimal length contained in $H_2$ and adjacent to some element $v\notin H_2$. The length of $u$ then gives an upper bound on the number of possible walls between $H_1$ and $H_2$. In particular this number is finite. Such a choice of an element $u$ is possible since roots are convex subsets of $\Gamma_W$. Projecting $1$ to the wall of $-H_2$ we obtain an edge $(u,v)$ having minimal distance to $1$.  
\end{proof}

\begin{figure}[h]
	\begin{center}
		\resizebox{!}{0.35\textwidth}{\includegraphics{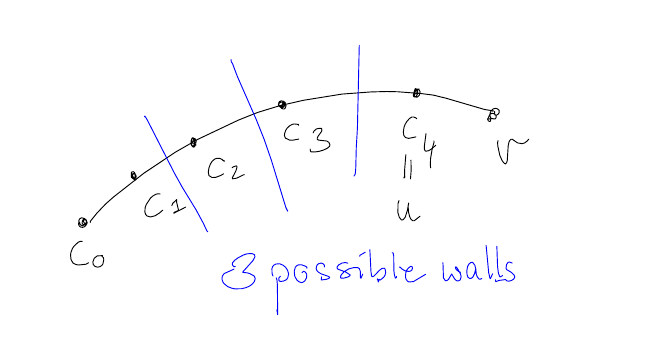}}
		\caption{Illustration of the proof of Proposition~\ref{prop:halfspace}.}
	\label{fig_57}
	\end{center}
\end{figure}

\begin{definition}
	Suppose $G$ acts on sets $X$ and $Y$. We say that a map $f:X\to Y$ is \emph{$G$-equivariant} if $f(g.x)=g.f(x)$ for all $g\in G$, $x\in X$. 
\end{definition}

\begin{prop}
	Suppose $(W,S)$ is a Coxeter group and $X(\cH)$ the associated cube complex defined in \ref{def:4.4} and \ref{def:4.4continued}. Then 
	\begin{enumerate}
		\item There exists a connected component $Y$ of $X(\cH)$ admitting a $G$-equivariant bijectiong $\Phi:\Gamma_W\to Y$. 
		\item The action of $W$ on $Y$ is properly discontinuous. 
	\end{enumerate}
\end{prop}
\begin{proof}
Define a map $f:\Gamma_W\to X(\cH)$ as follows: The set $M_r$ of walls in $\Gamma_W$ corresponds (not in a one-to-one fashion) with pairs of halfspaces $H(u,v)$ and $H(v,u)$ with $u,v$ being adjacent vertices in $\Gamma_W$. Write $\bar H(u,v)$ for such a pair of halfspaces and define a section $\nu_g$ for a fixed $g\in W$ by sending $\bar H(u,v)$ to that halfspace in $\bar H(u,v)$ containing $g$. 

It is clear that for each $g\in \Gamma_W$ the map $\nu_g$ defines a vertex in $H(\cH)$. The group $W$ acts on $\cH$ by $w.H(u,v)\define H(w.u, w.v)$ for all $w\in W$ which induces a transitive action of $W$ on the set $\{\nu_g \vert\ g\in W\}$ of sections by putting $w.\nu_g\define \nu_{w.g}$ for all $w$ and all $g$. 

It hence remains to show that adjacent vertices in $\Gamma_W$ are mapped to adjacent vertices in $X(\cH)$. Observe that for any pair $g\in W$ and $s\in S$ the only wall $\bar H$ on which the sections $\nu_g$ and $\nu_{gs}$ differ is defined by the edge $(g,gs)$. Since (w.o.l.g.) $g\in H$ and $gs\in H^\star$ Lemma~\ref{le:6.11} implies that $H=H(g,gs)$ and the claim follows. 

We need to show that vertex stabilizers $S_v\define \Stab_W(v)$ are finite for arbitrary vertices $v\in Y$. By construction of the map $\Phi$ in (1) each vertex $v\in Y$ is at finite distance of $\nu_1=\Phi(1)$ the image of the identity element. The orbit of $1$ under $S_v$ is the set of all vertices in $Y$ having distance $d(v_1, v)$ to $v$. In particular $S_v.v_1$ is finite. The stabilizer $S_v$ has a finite index subgroup which is contained in $S_{v_1}$ and hence $S_v$ itself has to be finite as $S_{v_1}=\{1\}$.      
\end{proof}

\begin{remark}
	One can show (by detailed inspection of the structure of Coxeter groups) that 
	\begin{enumerate}
		\item $X(\cH)$ is locally finite and finite dimensional. 
		\item The action of $W$ on $X(\cH)$ is cocompact of $W$ contains only finitely many conjugacy classes of subgroups isomorphic to a triangle group\footnote{A triangle group is a Coxeter group with presentation $\langle s_1, s_2, s_3 \vert\ s_i^2=(s_1s_2)^p=(s_2s_3)^q=(s_3s_1)^r \rangle$ where $p,q$ and $r$ are integers $\geq 2$}.   
		\item Word hyperbolic Coxeter groups $W$ act cocopactly on the associated cubical complex $X(\cH)$. 
	\end{enumerate}
\end{remark}

\begin{example}
	An example of a non-cocompact action of $W$ on $X(\cH)$ is the Coxeter group 
	\[\langle r,s,t\vert \ s^2=t^2=r^2=(st)^3=(rs)^3=(rt)^3\rangle\]
	the reflection group corresponding to the tiling og the Euclidean plane by equilateral triangles. Its cubical complex is the standard cubing of $\R^3$ and the quotient of the action is a copy of the real line. 
\end{example}

We are ready to explore another conrete example of a Coxeter group and its associated cubical complex.

\begin{example}
	Cubing $PGL(2,\Z)$: 
	The group $G=PGL(2,\Z)$ is the quotient of $GL(2,\Z)$ by the matrix $\left(\begin{smallmatrix}
	-1 & 0 \\0 & -1\\\end{smallmatrix}\right)$ and is generated by the following three matrices: 
	\[\left(\begin{matrix}
	0&1\\1&0\\
	\end{matrix}\right), 
	\left(\begin{matrix}
	-1&1\\1&0\\
	\end{matrix}\right) \text{ and } 
	\left(\begin{matrix}
	1&0\\0&-1\\
	\end{matrix}\right).\] 

As an abstract group $G$ is isomorphic to the Coxeter group corresponding to the diagram with three nodes where the first two are connected by an edge with no label and the second and third node are connected with the label $\infty$. Further $G$ acts as Möbius trabsformation on the upper half-plane model of $\mathbb{H}^2$. That is 
\[\left(\begin{smallmatrix}a&b\\c&d\end{smallmatrix}\right).z =  \frac{az+b}{cz+d} .\]

For a picture of a fundamental domain of this action compare the turquoise region $C$ in Figure~\ref{fig_58}.

\begin{figure}[h]
	\begin{center}
		\resizebox{!}{0.7\textwidth}{\includegraphics{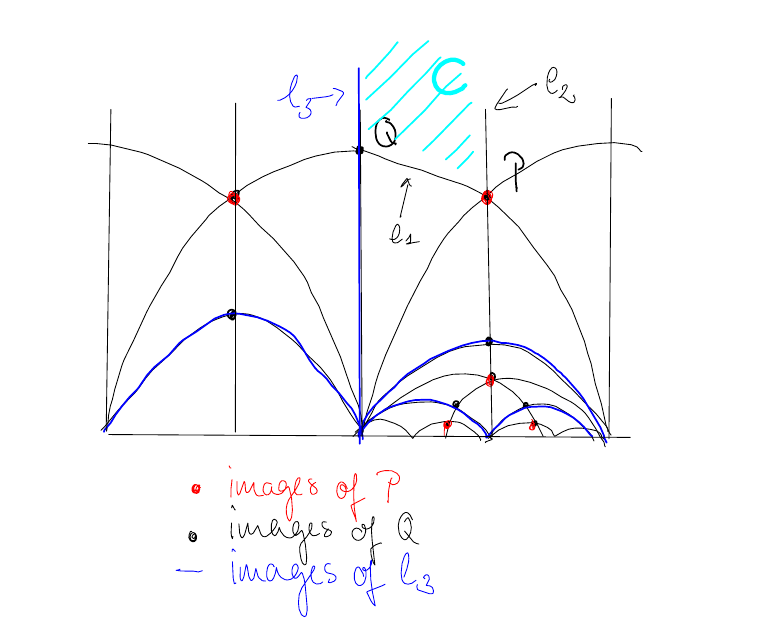}}
		\caption{Tiling of the hyperbolic plane induced by the action og $PGL(2,\Z)$. Region $C$ is a fundamental domain for this action.}
	\label{fig_58}
	\end{center}
\end{figure}

Apply now Sageev's construction to the complex forming a tiling of $\mathbb{H}^2$ by copies of the fundamental domain. Maximal cubes then correspond to families of pairwise intersecting hyperplanes. There are two classes/orbits of such families. For one the hyperplanes intersecting in $p$ or tranbslates of $p$ (see Figure~\ref{fig_58}) and on the other hand those intersecting in $q$  and its translates (see same figure).
The first class induces $3$-cubes while the second induces $2$-cubes in $X(\cH)$. Cubes corresponding to adjacent vertices $p,q$ share a face of dimension one. A piece of the resulting cubical complex is shown in Figure~\ref{fig_59} where it is also shown how the cubes fit together with the tiling of $\mathbb{H}^2$. One can also see in this picture that the action $G$ on the cube complex is in fact cocompact (look at the turquoise shaded region and its corresponding cubes).   
\end{example}

\begin{figure}[h]
	\begin{center}
		\resizebox{!}{0.7\textwidth}{\includegraphics{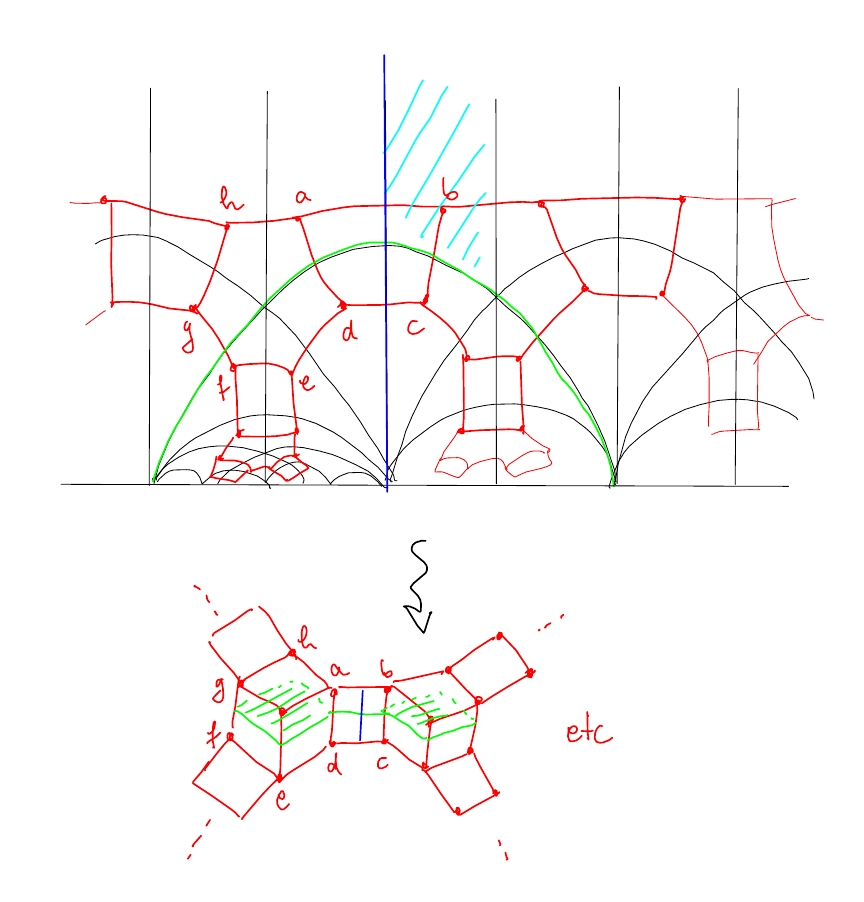}}
		\caption{The cube complex associated with $PGL(2,\Z)$.}
	\label{fig_59}
	\end{center}
\end{figure}

We end this section with two results and an open problem. 

\begin{thm}[Caprace, Niblo, Reeves]
	If a finitely generated Coxeter group $G$ does not contain a subgroup isomorphic to $(3,3,3)$ or $(4,4,2)$ then its action on the associated cubical complex is cocompact. 
\end{thm}

\begin{question}
	It is not knwon yet which groups do act on CAT(0) cubical complexes. For example one problem to study would be whether all finitely generated Artin groups act cocompactly and by isometries on some CAT(0) cube complex. It is e.g. not yet knwon whether all Mapping class groups have such an action. Moreover, once a group admits an action on a CAT(0) cubical complex one might ask if it is possible that the action is proper or fixed point free and what the minimal dimension of such a complex is. 
\end{question}

\begin{thm}[Bergeron-Wise~\cite{BergeronWise}]
	Let $M$ be a closed hyperbolic 3-manifold. Then $\pi_1(M)$ acts freely and cocompactly on a CAT(0) cubical complex. 	
\end{thm}

%
%

\newpage
\section{Tits Alternative}

In this chapter of our lecture we will discuss the Tits alternative in the setting of cube complexes. We will need some deep results from the literature - proving them would go way beyond the scope of this lecture. But excepting these beautiful results we will see some interesting interplay between group theoretic results and geometric aspects of cube complexes.

\begin{definition}\label{def:8.1}
 A group $G$ satisfies the \emph{Tits alternative} if for all its subgroups $H$ either $H$ is virtually special or $H$ contains a non-abelian free subgroup. 
\end{definition}

\begin{definition}\label{def:8.2}
 A group $H$ is \emph{virtually (P)} for some property (P) if $H$ contains a finite index subgroup having (P). 
\end{definition}

\begin{example}\label{ex:8.3}
Examples of groups satisfying the Tits alternative:
\begin{itemize}
 \item Tits ('72): finitely generated linear groups
 \item Gromov / Ghys-de la Harpe: hyperbolic groups
 \item Bestvina-Feighn-Handel: $Out(F_n)$
 \item Ballmann-Swiatkowski, Sageev-Wise, Xie: some classes of $\cat(0)$ groups 
\end{itemize}

It is conjectured that all $\cat(0)$--groups satisfy the Tits alternative. 

Here are some non-examples:
\begin{itemize}
 \item Pak: Grigorchuk's groups
 \item Brin-Squier: Thompson's group $F$
 \item de la Harpe: some Burnside groups
\end{itemize}

Here Thompsons group $F$ is the group given by the following presentation:
$$F=\{A, B \vert\, [AB^{-1},A^{-1}BA ]=[AB^{-1}, A^{-2}BA^2]=1 \}. $$
Despite the innocent appearence of $F$ when defined in this way this group is an important (counter-)example in geometric group theory which has lot's of inetresting properties. See for example the introduction by Cannon, Floyd and Parry \cite{CFP}.
\end{example}

We will prove the following theorem:

\begin{thm}\label{thm:8.4}
Let $G$ be a finitely generated group acting properly on a finite dimensional $\cat(0)$ cube complex $X$ and suppose that there is a bound on the order of its finite subgroups. Then for all subgroups $H < G$ either $H$ is virtually special or $H$ contains a (non-abelian) free subgroup of rank 2. That is $G$ satisfies the Tits alternative. 
\end{thm}

\begin{remark}\label{rem:8.5}
The assumption of finite dimensionality of $X$ is necessary since Thompson's group acts on an infinite dimensional $\cat(0)$ cube complex but is known to not satisfy the Tits alternative. 

Also necessary is the assumption on the bound of the order of subgroups of $G$. To see this consider the following example:
Let $G=G_1\subset G_2\subset \ldots $ be an ascending sequence of finite groups $G_i$. Associated to $G$ there is a coset tree $T$  built from left-cosets of the subgroups $G_i$ in $G$ (compare \cite{Serre} for details). 
By construction $G$ admits a proper left-action on $T$. But $G$ does not satisfy the assertion of Theorem~\ref{thm:8.4}.
\end{remark}

We will now start preparing ourselfs for the proof of Theorem~\ref{thm:8.4} by quickly recalling basic stuff on three topics: 
\begin{itemize}
 \item (relative) ends of groups
 \item Fuchsian groups and
 \item HNN-extensions and free amalgamations.
\end{itemize}

\subsection*{Ends of groups}

We will introduce a notion that will allow us to study the behaviour at infinity of a group $G$.

\begin{definition}\label{def:8.6}
We say that a locally finite graph $\Gamma$ has (at most) \emph{$m$ ends} if for each finite set of edges $F$ in $\Gamma$ one has that $m$ is the smallest integer such that $\Gamma\setminus  F$ has at most $m$ different connected components. This will be denoted by $\Ends(\Gamma)=m$. 
A group $G$ has (at most) \emph{$m$ ends} if $\Ends(G)\define\Ends(\Gamma(G,S))=m$ with $\Gamma(G,S)$ being a Cayley graph of $G$. 
\end{definition}

Once can show that the number of ends of a group is well defined. 
\begin{prop}\label{prop:8.8}
The number of ends $\Ends(G)$ is quai-invariant, i.e.\ independent of the choice of a generating set for $G$. 
\end{prop}

Let's now have a look at some examples. 
\begin{thm}\label{thm:8.8}
The number of ends $\Ends(G)$ of a group $G$ is quasi-invariant, i.e.\ independet of the choice of a generating set $S$ of $G$.  
\end{thm}
\begin{proof}
For a proof of this fact see \cite[I.8.29]{BH}. 
\end{proof}

\begin{example}\label{ex:8.7}
 \begin{enumerate}
  \item The group $G=<s,t\vert\, s^3=t^3=1>$ is a finite reflection group with $\Ends(\Gamma(G, \{s,t\})=0$. 
  \item Consider $G=\Z^2$ presented as $<a,b\vert\, ab=ba>$, then $\Ends(G)=1$. 
  \item $G=\Z$ with generating set $\{1\}$ or $\{2,3\}$ has two ends.
  \item Finally the free group $F_2=<a,b>$ is an example of a group with infinitely many ends. 
  \end{enumerate}
\end{example}

How about examples for groups with 3,4,5,... ends? Any ideas? Better not, since the following theorem can be shown:

\begin{thm}(Hopf 1944)\label{thm:8.9}
Each finitely generated group $G$ satisfies $$\Ends(G)\in \{0,1,2, \infty\}.$$  
\end{thm}
\begin{proof}
The proof (taken from \cite[p.146/147]{BH}) is by contradiction. 
Let $G$ be a group with $\Ends(G)<\infty$ and suppose $e_1, e_2, e_3$ are three pairwise diefferent ends of $G$. 
Write $\Gamma$ for a Cayley graph of $G$ on which $G$ acts by left-multiplication. This action induces a homeomorphism $\psi:G\to Homeo(\Ends(G))$ whose kernel will be denoted by $H=\ker(\psi)$. 

Choose geodesic rays $r_i:[0,\infty) \to \Gamma$, $i=1,2$ such that $r_i(0)$ is the identity in $G$ and $r_i$ lies eventually in the connected component defined by $e_i$, $i=1,2$. 

Since $H$ has finite index in $G$ there exist a constant $\epsilon$ such that all vertices of $\Gamma$ are contained in $B_\epsilon(H)$. 

One can show that there exists a proper map $r_3:[0,\infty)\to \Gamma$ such that $r_3$ lies eventually in the connected component defined by $e_3$, such that $d(r_3(n), 1)\geq n$ and such that $r_3(n)\in H$ for all $n$. Recall that for a proper map preimages of compact sets are compact. Consider the sequence $\gamma_n\define r_3(n)$, $n\in \N$.  

Choose $\rho> 0$ such that the three sets $r_i([\rho, \infty)$ are contained in different connected components of $\Gamma\setminus B_\rho(1)$. Such a choice is possible for $\rho$ big enough since $\Gamma$  has three ends. 

For $t,t'>2\rho$ we have that $d(r_1(t), r_2(t')) > 2\rho$ as a path connecting the two points has to go via $B_\rho(1)$. 

The action of $H$ on the set of ends of $G$  is trivial hence $\gamma_i .r_i$ is contained in the same end as $r_i$ for both $i=1$ and $i=2$. 

For big enought $n>3\rho$ the point $\gamma_n.r_i(0)$ is contained in a different component as $r_i([\rho,\infty))$ for $i=1,2$. We conclude that then the path $\gamma_n.r_i$ runs through $B_\rho(1)$ for both $i$ and lies finally in the same connected component as $r_i$. 
Therefore 
$$\gamma_n.r_1(t)\in B_\rho(1) \text{ and } \gamma_n.r_2(t')\in B_\rho(1)$$ for suitably chosen $t,t'>2\rho$ (depending on the sequence $(\gamma_n)$) and we obtain that 
$$d(\gamma_n.r_1(t), \gamma_n.r_2(t')) < 2\rho.$$

But on the other hand $\gamma_n$ is an isometry and thus 
$$d(\gamma_n.r_1(t), \gamma_n.r_2(t'))=d(r_1(t), r_2(t')) > 2\rho$$
and we arrive at a contradiction. 
\end{proof}

\begin{remark}(Hopf'44, Stallings '68)\label{rem:8.10}
Hopf classified the groups having a given number of ends as follows: 
\begin{itemize}
 \item $\Ends(G)=0 \Leftrightarrow G$ is finite. 
 \item $\Ends(G)=1 \Leftrightarrow G$ is virtually $\Z$. 
 \item $\Ends(G)=2 \Leftrightarrow G$ is virtually infinite cyclic.
\end{itemize}

Stallings showed that the following three conditions are equivalent:
\begin{enumerate}
 \item $\Ends(G)>1$
 \item $G=A\star_C B$ a free amalgamation or $G=A_{\star C}$ a HNN-extension with $C$ being finite, $\vert A\diagup C\vert \geq 3$ and $\vert B\diagup C\vert \geq 2$.
 \item $G$ admits a non-trivial  action\footnote{i.e.\ without global fixed points.} on a simplicial tree. 
\end{enumerate}

The equivalence of the last and second to last item in Stallings result are shown using Basse-Serre theory. In general the proofs of these facts are rather involved and will not be presented here. Note that Niblo \cite{Niblo}  gave a very nice geometric proof of Stallings theorem. 
\end{remark}

We now turn to a relative notion of ends of a group. 

\begin{definition}\label{def:8.11}
Let $G$ be again a finitely generated group and $H$ a subgroup in $G$. The number of \emph{ends of $G$ relative to $H$}, denoted by $\Ends(G,H)$ is the number of ends $\Ends(\Gamma/H)$ of the quotient of the Cayley graph $\Gamma$ of $G$ by the natural left-action of $H$.  
\end{definition}

Equivalently we could define $\Ends(G,H)$ to be the number of ends $\Ends(X)$ where $X$ is the graph with vertices the left-cosets of $H$ with edges between $Hx$ and $Hxg$ for all $g\in S$, where $S$ is a finite generating set of $G$. 

Analogously to Stallings result we have the (even more difficult) following theorem which is due to {Dunwoody and Swenson.}

\begin{thm}\label{thm:8.12}
Let $G$ be a finitely generated group and $H< G$ a subgroup.  If $H$ is a virtually polycyclic group with $\Ends(G,H)>1$ then one of the following possibilities holds
\begin{enumerate}
 \item $G$ is virtually polycyclic
 \item ther exists a short exact sequence $1\to P\to G\to G/P \to 1$ where P is virtually polycyclic and $G/P$ not an elementary Fuchsian group. 
 \item $G$ decomposes as an amalgamation $A\star_C B$ or HNN extension $A_{\star C}$ over a virtually polycyclic group $C$.
\end{enumerate}
\end{thm}

Let me introduce some of the notions mentioned in this theorem.

\begin{definition}\label{def:8.13}
 A group $G$ is \emph{polycyclic} if there exists a sequence of subgroups $G=G_0, G_1, \ldots, G_n=1$ such that
 \begin{itemize}
  \item $G_{i+1}$ is normal in $G_i$ for all $i$
  \item the quotient $G_i\diagup G_{i+1}$ is cyclic for all $i$. 
 \end{itemize}
\end{definition}

A good way to think about a poycyclic group is thinking of a ``tower of cyclic groups'' which is justified by the fact that iterated semidirect products of cyclic groups are polycyclic. 

Other (important) examples of polycyclic groups are finitely generated abelian groups. 

\subsection*{Fuchsian groups}

\begin{definition}\label{def:8.14}
A group $G$ is a \emph{Fuchsian group} if $G$ is a discrete subgroup of $PSL_2(\R)$.  
\end{definition}

Any Fuchsian group does in particular admit a discrete action on the hyperbolic plane $\h^2$. 

Consider the action of a Fuchsian group on the Poincar\'e disc model $D$ of the hyperbolic plane $D$. Viewing $D$ as a disk in the Euclidean plane orbit $G.z$, for some $z\in D$ of the $G$ action may have \emph{limit points} in the boundary circle. One can prove

\begin{prop*}
The set of limit points of the action of $G$ on $D$ is independent of the choice of the basepoint $z$ and $G$ either has $0,1,2$ or $\infty$ limit points.  
\end{prop*}

\begin{definition}\label{def:8.14a}
We say that $G$ is \emph{elementary Fuchsian} if it has a  finite number of limit points. 
\end{definition}

Moreover if $G$ has infinitely many limimt points then it contains a copy of the free group $F_2$ on two generators. The Cayley graph of $F_2$ may be embedded in $D$ in such a way that its endpoints correspond to the limit points. 

\begin{example}\label{ex:8.14b}
The group $PSL_2\Z$ is a Fuchsian group having $\infty$-many limit points.   
\end{example}

\subsection*{HNN-extensions and free amalgamations}

We will now define certain constructions to produce groups. First we define free amalgamated products of a pair of groups $G_1, G_2$ which ``share'' a subgroup. This constructions gives us a group $G$ with the property that $G_i$ embeds in $G$ and the isomorphic subgroups are conjugate in $G$. 
 
\begin{definition}\label{def:8.15a}
Let $G_i\define <S_i\vert\, R_i>$, $i=1,2$ be two groups with subgroups $H_i< G_i$ such that there is an isomorphism $\phi:H_1\to H_2$.  Then the \emph{free amalgamated product} of $G_1$ and $G_2$ over $H\cong H_i$, $i=1,2$ is defined by
\begin{equation*}
 G_1\star_HG_2 \define <S_1\sqcup S_2\vert\, R_1\sqcup R_2\sqcup\{\phi(h)h^{-1} \,\forall h\in H_1\} >.
\end{equation*}
\end{definition}

\begin{remark}\label{def:8.15b}
It is clear that $G_i$ is a subgroup of the amalgam. More generally an amalgamated product of more than two groups is described by the formula
\begin{equation*}
 \star_{i\in I}G_i \define <\bigsqcup_{i\in I} S_i\,\vert\, \bigsqcup_{i\in I} R_i \sqcup \{\phi_i(h)\phi_j(h)^{-1} \;\forall h\in H_1\;\forall i,j\in I\} >.
\end{equation*}
\end{remark}

Next let us define HNN-extensions which mimic amalgamated products but with two isomorphic subgroups of a single group. The main feature of this constructions is that a given group $G$ is embedded in a larger group $G'$ such that its isomorphic subgroups are conjugate in $G'$ through a fixed isomorphism. 

\begin{remark}\label{rem:8.15c}
The name HNN-extension refers to Graham {H}igman and Bernhard and Hanna {N}eumann who first introduced this concept. 
\end{remark}

\begin{definition}\label{def:8.15d}
Suppose $G=<S\,\vert\, R>$ is a group and $\phi:H_1\to H_2$ an isomorphism between two subgroups of $G$. The \emph{HNN-extension} of $G$ by a stable generator conjugating $H_1$ with $H_2$ is defined by 
\begin{equation*}
G\star_H\define <S\sqcup \{t\}> \,\vert\, R\sqcup\{tht^{-1}=\phi(h) \; \forall h\in H_1\} 
\end{equation*}
\end{definition}
 
A common generalization of free amalgamations and HNN-extensions is the concept of a graph of groups. See \cite{BH} for more on this topic. 

We now dip our toes into Bass-Serre theory and define trees associated to the two constructions. 

\begin{definition}\label{rem:8.15d}
The \emph{Bass-Serre tree} associated to a free amalgamated product $G=A\star_C B$ is defined as follows:
$$T\define (G\times[0,1])\diagup_\sim $$
where the equivalence relation $\sim$ is induced by the three relations 
$(ga, 0)\sim (g,0)$, $(gb,1)\sim (g,1)$ and $(gh,t)\sim (g,t)$ for all group elements $g\in G$, $a\in A$, $b\in B$ and $h\in C$ and parameters $t\in [0,1]$.  
The left-translation of $G$ on $G\times[0,1]$ are compatible with these relations. I.e.\ $G$ acts on $T$ by isometries. 
\end{definition}
Compare also \cite[p.355 and Thm. II.11.18]{BH}.

\begin{definition}\label{rem:8.15e}
The \emph{Bass-Serre tree} associated to an HNN-extension $G=A \star_C$ is defined as follows:
$$T\define (G\times[0,1])\diagup_\sim $$
where the equivalence relation $\sim$ is induced by $(g,s)\sim (g\phi(h),s)$, $(g,0)\sim(ga, 0) \sim(gt,1)$ for all $g\in G$, $a\in A$, $h\in C$ and $s\in[0,1]$ and with $t$ being the conjucating parameter from Definition~\ref{def:8.15d}. 
\end{definition}

Again $G$ acts on $T$ by isometries and the quotient of this action is a complex made out of a single vertex and an edge with both ends glued to that vertex. The stabilizer of an edge in $T$ is isomorphic to $C$, the stabilizer of a vertex isomorphic to $A$. The number of edges connected to a same vertex in $T$ equals $[A:C]$.

We will need the following normal form theorems due to Britton.

\begin{thm}\label{thm:8.15f} 
Suppose $G=A\star_C B$. Choose $a_i\in A$ and $b_i\in B$, $i=0,1,\dots,n$ such that $a_i\notin C$ for all $i>0$ and $b_i\notin C$ for all $i< n$. Then $$a_0b_0a_1b_1\dots a_nb_n \neq 1$$ in $G$. Such a presentation of an element is \emph{of reduced form} and every $g\neq 1$ can be writen in reduced form.
\end{thm}
\begin{thm}\label{thm:8.15g} 
For $G=A\star_C=<S\sqcup \{t\} \,\vert\, R\sqcup\{tht^{-1}=\phi(h), \; \forall h\in C_1\}>$, where the group $A=<S\vert\, R>$ and $C\cong C_1\cong C_2$ with $C_i$ a subgroup of $A$ and $\phi:C_1\to C_2$ an isomorphism. Then  
$$g=a_0t^{m_1}a_1t^{m_2}a_2 \dots t^{m_n}a_n\neq 1$$
for all $m_i\in \Z\setminus\{0\}$ and all $a_i\in A$ chosen such that $a_i\notin C_1$ if $m_i<0$ as well as $a_i\notin C_2$ if $m_i>0$.  
Again such an element $g$ is \emph{of reduced form} and every non-trivial $g\in G$ may be written in reduced form. 
\end{thm}

\begin{remark}\label{thm:8.15h}
>From Theorems~\ref{thm:8.15f} and  \ref{thm:8.15g} one can deduce
\begin{itemize}
 \item The natural homeomorphisms $A\to A\star_C$ and $A,B\to A\star_C B$ are injective. 
 \item If $C_1\neq A$ and $C_2\neq A$ then $A \star_C$ contains a subgroup isomorphic to $F_2$. 
 \item Is $[A:C]>2$ and $[B:C]>2$ then $A\star_C B$ contains a subgroup isomorphic  to $F_2$. 
\end{itemize}
\end{remark}

The following theorem will directly imply the Tits alternative \ref{thm:8.4}. Just apply \ref{thm:8.16} to any subgroup $H$ of $G$ in the setting of \ref{thm:8.4}.

\begin{thm}\label{thm:8.16}
 Let $G$ be a finitely generated group acting properly discontinous on a finite dimensional cubical complex $X$ and suppose $G$ has an upper bound on the order of its finite subgroups. Then either 
 \begin{enumerate}
  \item\label{8.16.1} $G$ contains a free subgroup of order two or
  \item\label{8.16.2} $G$ is virtually finitely generated abelian. 
 \end{enumerate}
\end{thm}
\begin{proof}
The proof is by induction on $n=\dim(X)$.  For $n=0$ $G$ has to be finite (since otherwise a properly discontinuous action does not exist) and we are in case (\ref{8.16.2}). 
Suppose the theorem is true for $\dim(X)< n$. Then again, if $G$ is finite we are in case (\ref{8.16.2}). Hence w.l.o.g.\ we may assume that $G$ is infinite. By assumption $G$ does not have a global fixed point on $X$ and we may apply the following result:

\begin{thm}\cite{Sageev}\label{thm:8.18}
If $G$ is a group acting without a global fixed point on a $\cat(0)$ cubical complex $X$ of dimension $<\infty$, then there exists a hyperplane $J$ in $X$ such that $\Ends(G, \Stab_G(J))>1$.  
\end{thm}

We put $H\define \Stab_G(J)$ with $J$ the hyperplane obtained from the theorem just mentioned. The action of $H$ on the $\cat(0)$ cube complex $J$ is properly discontinuous and we deduce from the induction hypothesis that $H$ either is virtually f.g.\ abelian or contains a free subgroup of rank two. 

Suppose \ref{8.16.1} is true. Then there exists $F_2<H<G$ and the asserion follows. If \ref{8.16.2} is satisfied by $H$, then $H$ is in particular virtually solvable and we may apply the algebraic torus theorem \ref{thm:8.12}. Hence $G$ satisfies one of the following three possibilities:
\begin{enumerate}
 \item\label{8.18.1} $G$ is virtually polycyclic.
 \item\label{8.18.2} $G$ contains a non-elementary abelian Fuchsian quotient with virtually polycyclic kernel. 
 \item\label{8.18.3} $G$ decomposes as a product (free amalgamation or HNN-extension) over a virtually polycyclic subgroup. 
\end{enumerate}
We are dealing with each of the three possibilities separately and will show that in any case $G$ satisfies the assertion of our theorem.

\emph{Case \ref{8.18.1}:} We apply a Lemma of Bridson (see \ref{le:8.19} below) to  prove that $G$ is virtually abelian. Item \ref{8.19.1} implies that item \ref{8.19.2} is applicable and since polycyclic implies solvable the assertion follows. Here's the lemma:
\begin{lemma}\label{le:8.19}
\begin{enumerate}
 \item\label{8.19.1} Suppose $G$ admits a cellular action on a $\cat(0)$ complex $X$ which is built out of finitely many shapes. Then the action of $G$ on $X$ is semisimple with a discrete set of translation lengths. 
 \item\label{8.19.2} Suppose $G$ admits a semisimple action, properly discontinuous action on a $\cat(0)$ space $X$. Then every virtually solvable subgroup of $G$ is virtually abelian. (In particular $G$ is virtually abelian if it is virtually solvable.)
\end{enumerate}
\end{lemma}

A \emph{semisimple action} is an action for which each group element is a semisimple isometry\footnote{The class of semisimple isometries of a metric space is split into two subclasses. Isometries $f$ with $\vert f\vert>0$ are called \emph{hyberbolic} and isometries with translation length $\vert f\vert=0$ are \emph{elliptic}. Non-semisimple isometries are called \emph{parabolic}.}, that is there exists $x_0\in X$ such that $d(g.x_0, x_0)$ equals the translation length $\vert g\vert\define \inf\{d(g.x,x)\,\vert\, x\in X\}$ of $g$.

\emph{Case \ref{8.18.2}:} This is the easiest case as the non-elementary Fuchsian quotient does contain a copy of $F_2$. Hence $G$ does. 

\emph{Case \ref{8.18.3}:} Suppose $G=A\star_P B$ or $G=A\star_P$ for some virtually polycyclic $P$. Then by Lemma~\ref{le:8.19}.\ref{8.19.1} the action of the group $G$ on $X$ is semisimple and item \ref{8.19.2} implies that $P$ needs to be virtually abelian. 

Let us first consider the case with $G=A\star_P B$. The normal form theorem \ref{thm:8.15f} implies that if $[A:P]> 2$ or $[B:P]>2$ then $G$ does contain a free subgroup of rank two. 
We may hence suppose that $[A:P]=2=[B:P]$. 
In this case the Bass-Serre tree $T$ is a line and any edge in $T$ is stabilized by a conjugate of $P$. This imples then that there exists an intex two subgroup $G'$ of $G$ which acts by translation on $T$ such that $\ker(G'\to \Z) \cong P$. Hence $G'\cong P\ltimes \Z$ is polycyclic and $G$ is thus vortually polycyclic and (by Lemma~\ref{le:8.19}) virtually abelian. 

Suppose now $G=C\star_P$. Then there exist subgroups $P_1, P_2$ of $C$ isomorphic to $P$ via $\phi:C_1\to C_2$ isomorphic to one another. The normal form theorem \ref{thm:8.15g} now implies that if $[C:P_1]>1$ and $[C:P_2]>1$ then $G$ contains a free subgroup of rank two.
One can show that if $[C:P_1]=1$ then $[C:P_2]=1$ and vice versa. 
And we may conclude that $G$ is isomorphic to $P\ltimes\Z$, is virtually polycyclic an hence virtually abelian. Thus the assertion. 
\end{proof}

\newpage

\section{Appendix: Phylogenetic trees}  

This section (which actually appeared after Section 6 in my lecture) is of completely different flavor as the rest of these notes. The reason for this is, that this was our "Christmas lecture". The last class I gave in the week of Christmas eve I wanted to show the students an  application of cube complexes which appeared in the literature (more or less) recently.
%


This section is based on {\cite{BHV}} which is maybe the most prominent example of an application of cubical complexes outside geometric group theory. The starting point is the following: 

{\bf Setting:}\newline
 One wants to study and graphically illustrate hierarchical connections between different species by a given uncertainty of the given data or unvcertainty in the process of creation of the data one is working with. For example if data is based on statistical methods different attempts to measure the evolution of the species will lead to different models and different graphical results. 

To give you two examples to have in mind you may think of either the genetic evolution of biological species and how theit genetic code differs or the evolution of languages where it is hard to measure how they differ and whether or not two languages have common ancestors. 

We will use \emph{rooted trees} to present the evolution of a set of species with the following interpretations/connections: 

\begin{itemize}
	\item The leaves are labeled and each leave stands for an existing species. 
	\item Nodes having a common ancestor are  closely related. 
\end{itemize}

Moreover we do want to incode more than just having a common ancestor so we scale the trees and let them be metric spaces with different edge lengths. 

\begin{itemize}
	\item the lengths of edges in the tree encode additional information (e.g. time needed to branch, amount of genes different from the ancestor, number of mutations needed, ...)
\end{itemize}


{\bf Problems:} The main difficulties when looking at such models are 

\begin{itemize}
	\item How can we measure distances between two trees?
	\item How can we interpolate between two trees?
	\item How can we calculate probability of appearance of a given set of trees in the space of all possible trees? Can one put a measure on the space of all trees? 
\end{itemize}

The model suggested in \cite{} does allow for all these things. Let's have a more detailed look. 

\begin{definition}
	An \emph{n-tree} is a simplicial tree $T$ with a fixed vertex of valency $\geq 2$ called \emph{root}, $n$ vertices of valency $1$, the \emph{leaves}, and valency $\geq 3$ for all other (interior) vertices. An edge in $T$ is called \emph{interior}. if it does not contain a leave. 
	
	A \emph{labeling} of an $n$-tree $T$ is a fixed bijection from $\{1,2,\ldots, n\}$ to the set of leaves of $T$. If we equip $T$ with a fixed labeling we say that $T$ is \emph{labeled}. 
	
	A \emph{metric} tree $T$ is a fixed geometric realization of $T$ where all edge-lengths are chosen in $\R^+$. 
	  
\end{definition}

\begin{figure}[h]
	\begin{center}
		\resizebox{!}{0.3\textwidth}{\includegraphics{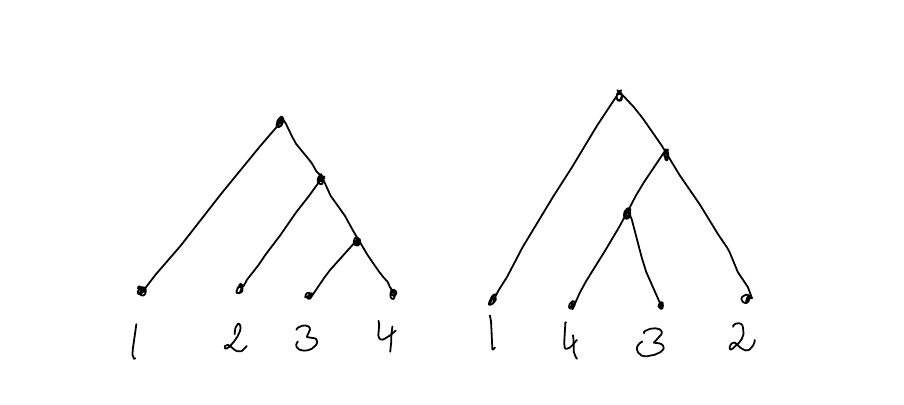}}
		\caption{Identical metric labeled 4-trees.}
	\label{fig_60}
	\end{center}
\end{figure}

\begin{remark}
	Alternatively we may think of a metric tree as a combinatorial tree with an additional  labeling of its edges where the labels are just the lengths of these edges in the geometric realization.  
\end{remark}

\begin{example}
The two trees in Figure~\ref{fig_60} are identical metric labeled trees. 	
\end{example}	
	
\begin{prop}
	An $n$-tree has $\leq n-2$ interior edges where equality holds for binary trees. 
\end{prop}

\begin{figure}[h]
	\begin{center}
		\resizebox{!}{0.4\textwidth}{\includegraphics{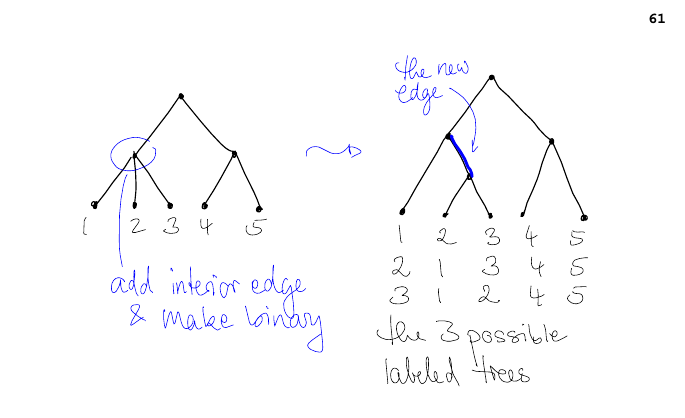}}
		\caption{Enlargement of a labeled 5-tree to a binary one.}
	\label{fig_61}
	\end{center}
\end{figure}

\begin{proof}
	(Sketch) The proof is by induction on $n$. For $n=2$ there is only one possible tree with no interior edges. An $n$-tree may be thought of as arising from an $(n-1)$-tree by "gluing" an additional leave. It is easy to see (draw yourself some pictures or look at Figure~\ref{fig_62}) that this will result in at most one more interior edge when gluing two edges to an existing leave. Hence we have at most $(n-1)+1-2=n-2$ interior edges in the $n$-tree. Binary trees can not be enlarged (while non-binary ones can) without adding leaves so they are the ones with the maximal amount of interior edges. Compare Figure~\ref{fig_61}.  
\end{proof}

\begin{figure}[h]
	\begin{center}
		\resizebox{!}{0.4\textwidth}{\includegraphics{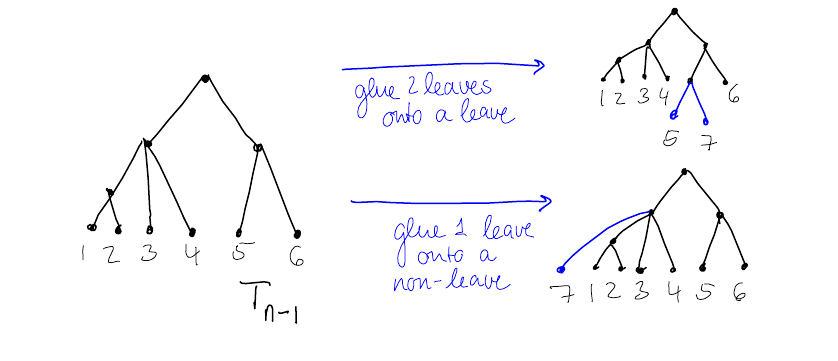}}
		\caption{Constructing an n-tree from an $(n-1)$-tree.}
	\label{fig_62}
	\end{center}
\end{figure}

\begin{definition}[Space of phylogenetic trees]
	
Let $T$ be a binary labeled $n$-tree with interior edges enumerated by $e_1, e_2, \ldots, e_{n-2}$. Associate to $T$ a quadrant $Q_T\cong[0,\infty)^{n-2}$ where each point $p=(p_1, p_2, \ldots, p_{n-2})^t$ in $Q_T$ corresponds to a metric tree which is a geometric realization of $T$ whose interior edge $e_i$ has length $p_i$ for all $i$. 

Note that the faces of $Q_T$ correspond to non-binary $n$-trees where certain interior edges are collapsed (habe length 0). Define the \emph{space of phylogenetic $n$-trees $\mathcal{T}_n$} to be the quotient 
$$
\mathcal{T}_n \define \bigcup_{T\in \mathcal{B}_n} Q_T \diagup \sim , 
$$
where $\mathcal{B}_n$ stands for the set of binary $n$-trees and where we denote by $\sim$ the equivalence relation on points in quadrants induced by equivalence of labeled metric $n$-treed. That is points $p\in Q_T$ and $p'\in Q_{T'}$ are identified in the quotient if and only if $p$ and $p'$ represent the same labeled metric $n$-tree. 
\end{definition}

\begin{prop}[(Schröder 1870)]
	There exist $(2n-3)!! = \frac{(2n-2)!}{2^{n-2}(n-1)!}$ binary labeled $n$-trees. 
\end{prop}

\begin{example}
	\begin{enumerate}
		\item The set $\mathcal{B}_3$ contains $(2\cdot 3 -3)!!=3$ labeled binary $3$-trees. Hence $\mathcal{T}_3$ contains three quadrants of dimension one (i.e.rays) glued together at the origin.
		\item The set $\mathcal{B}_4$ contains fifteen elements and each tree defines a $2$-dimensional quadrant. At each face of dimension one meet three such quadrants.  
	\end{enumerate}
\end{example}

Note that there are many embeddings of $\mathcal{T}_3$ into $\mathcal{T}_4$ (and in general of $\mathcal{T}_k$ into $\mathcal{T}_n$ for arbitrary $k<n$.  

\begin{thm}
	The link of the origin in $\mathcal{T}_4$ is the Petersen graph. 
\end{thm}

The proof of this theorem is left as an exercise and we now focus on metric properties of $\mathcal{T}_n$.

\begin{definition}
We define the metric $d$ on $\mathcal{T}_n$ to be the unique length metric such that on each quadrant the restriction $d\vert_{Q_T}$ equals the restriction of the euclidean metric.  	 
\end{definition}

\begin{thm}\label{thm:tree-space CAT(0)}
	For all $n$ the space $\mathcal{T}_n$ carries the structure of a CAT(0) cubical complex. 
\end{thm}

One can either prove the result directly or apply the following theorem to $Y=\lk_{\mathcal{T}_n}(0)$.

\begin{thm}[Berestovskii]
	The CAT(0) cone over a metric space $Y$ is a CAT(0) space if and only if $Y$ is CAT(1). 
\end{thm}

In case $n=3$ one can easily check that the Petersen graph is a CAT(1) space and hence conclude the statement for Theorem~\ref{thm:tree-space CAT(0)} for $n=3$.

%
%
%
%

\newpage
\phantomsection
\renewcommand{\refname}{Bibliography}
\bibliography{literaturliste}
\bibliographystyle{alpha}

\end{document}